\documentclass[11pt]{article}

\usepackage[T1]{fontenc}
\usepackage{lmodern}
\usepackage{amsmath,amsfonts,amsthm,amssymb,mathtools}
\usepackage{mathrsfs}
\usepackage{enumitem}
\usepackage[margin=2.5cm]{geometry}
\usepackage{microtype}
\usepackage{xcolor}
\usepackage{hyperref}
\hypersetup{
    colorlinks=true,
    citecolor=blue,
    linkcolor=blue,
    urlcolor=blue
}

\newtheorem{theorem}{Theorem}[section]
\newtheorem{lemma}[theorem]{Lemma}

\newtheorem{corollary}[theorem]{Corollary}
\newtheorem{conjecture}[theorem]{Conjecture}
\theoremstyle{definition}

\allowdisplaybreaks[4]
\numberwithin{equation}{section}

\title{\Large\bfseries A complete solution to the Tokushige measure conjecture and its stability%
\thanks{Lihua Feng was supported by the NSFC (Nos. 12271527 and 12471022)
and NSF of Qinghai Province (No. 2025-ZJ-902T).
E-mail addresses: \url{wuyjmath@163.com} (Y. Wu),
\url{fenglh@163.com} (L. Feng).}}
\author{
{\small Yongjiang Wu,\quad Lihua Feng\thanks{Corresponding author.}}\\[2mm]
\small School of Mathematics and Statistics, HNP-LAMA,
Central South University\\
\small Changsha, Hunan, 410083, China
}
\date{\today}

\begin{document}
\maketitle

\begin{abstract}
We resolve three conjectures proposed by Tokushige in 2013 about cross $t$-intersecting families of subsets and integer sequences. For $0<p<1$ and $\mathcal F\subseteq2^{[n]}$, define the
$p$-biased measure by $\mu_p(\mathcal F)=\sum_{F\in\mathcal F}p^{|F|}(1-p)^{n-|F|}$. Two families $\mathcal F_1,\mathcal F_2\subseteq2^{[n]}$ are cross $t$-intersecting if $|F_1\cap F_2|\geq t$ for every $F_1\in\mathcal F_1$ and $F_2\in\mathcal F_2$. We prove that, for every $n\geq t\geq 2$ and $0<p_1,p_2\leq1/(t+1)$,  such a  pair satisfies $\mu_{p_1}(\mathcal F_1)\mu_{p_2}(\mathcal F_2)\leq(p_1p_2)^t$. When $p_1,p_2<1/(t+1)$, equality holds if and only if both families are the same $t$-star $\mathcal S_T=\{F\subseteq[n]:T\subseteq F\}$ for some $T\in\binom{[n]}{t}$.
 Together with the previously known case $t=1$,
 this completely resolves Tokushige's measure conjecture.
We also prove that every pair whose measure product is close to the maximum must be close to a common $t$-star. More precisely, if  $p_1,p_2<1/(t+1)$ and $\mu_{p_1}(\mathcal F_1)\mu_{p_2}(\mathcal F_2)>(1-\varepsilon)^2(p_1p_2)^t$, then there exists $T\in\binom{[n]}{t}$ such that $\mu_{p_i}(\mathcal F_i\mathbin{\triangle}\mathcal S_T)<C\varepsilon$ for $i=1,2$, where $C$ depends only on $t,p_1,p_2$. This improves Tokushige's conjectured $C\sqrt{\varepsilon}$ estimate to $C\varepsilon$.

For integer sequences, we prove that if every sequence in $\mathcal H_1\subseteq[m]^n$ agrees with every sequence in $\mathcal H_2\subseteq[m]^n$ in at least $t$ coordinates, then $|\mathcal H_1||\mathcal H_2|\leq m^{2(n-t)}$ for all $n\geq t\geq1$ and $m\geq t+1$. We further obtain a more general result in which a separate agreement requirement is imposed for each possible value. This extends a theorem of Frankl and Kupavskii and  recovers
their earlier cross intersection--union product theorem.
\end{abstract}

{\bf AMS Classification}:  05D05; 05C65 

{\bf Keywords}:
Cross $t$-intersecting families; Biased measures; Stability; Integer
sequences

\tableofcontents

\section{Introduction}

Let $[n]=\{1,2,\ldots,n\}$, and let $2^{[n]}$ denote its power set. For $0\leq k\leq n$,
write $\binom{[n]}{k}$ for the family of all $k$-element subsets of
$[n]$. A family $\mathcal F\subseteq2^{[n]}$ is called \textit{$k$-uniform} if
$\mathcal F\subseteq\binom{[n]}{k}$, and it is called
\textit{$t$-intersecting} if
\[
 |F\cap F'|\geq t
 \qquad\text{for all }F,F'\in\mathcal F.
\]
When $t=1$, the family is simply called \textit{intersecting}.
Erd\H{o}s, Ko and Rado \cite{ErdosKoRado} proved that, if $n\geq2k$,
then every intersecting family
$\mathcal F\subseteq\binom{[n]}{k}$ satisfies
$
 |\mathcal F|\leq\binom{n-1}{k-1}.
$
They also showed that, for sufficiently large $n$, every $t$-intersecting
family $\mathcal F\subseteq\binom{[n]}{k}$ satisfies
$
 |\mathcal F|\leq\binom{n-t}{k-t}.
$ The bound is attained by the
family of all $k$-sets containing a fixed $t$-set. Frankl
\cite{Frankl1978} proved that, for $t\geq15$, the same bound holds
whenever $n\geq(t+1)(k-t+1)$, and Wilson \cite{Wilson} subsequently
removed the restriction on $t$. Ahlswede and Khachatrian
\cite{AhlswedeKhachatrian} later determined the maximum size of a
$t$-intersecting uniform family for all parameters in their Complete
Intersection Theorem. 
For further background, see Ellis \cite{EllisSurvey} and Frankl and
Tokushige \cite{FranklTokushigeSurvey,FranklTokushigeBook}.

Two families $\mathcal F_1,\mathcal F_2\subseteq2^{[n]}$ are called
\textit{cross $t$-intersecting} if
\[
 |F_1\cap F_2|\geq t
 \qquad\text{for all }F_1\in\mathcal F_1
 \text{ and }F_2\in\mathcal F_2.
\]
When $t=1$, they are simply called \textit{cross-intersecting}. For uniform
families, a natural extremal problem is to determine the largest
possible value of $|\mathcal F_1||\mathcal F_2|$. Pyber
\cite{Pyber} proved that, if $n\geq2k$ and
$\mathcal F_1,\mathcal F_2\subseteq\binom{[n]}{k}$ are
cross-intersecting, then
$
 |\mathcal F_1||\mathcal F_2|
 \leq\binom{n-1}{k-1}^{2}.
$
Matsumoto and Tokushige \cite{MatsumotoTokushige} obtained the sharp
extension to different uniformities. Product-type extensions involving several families were studied by Tokushige \cite{TokushigeProduct}.
For general $t$, the natural EKR-type question asks whether, in the
range $n\geq(t+1)(k-t+1)$, every cross $t$-intersecting pair
$\mathcal F_1,\mathcal F_2\subseteq\binom{[n]}{k}$ satisfies
$
 |\mathcal F_1||\mathcal F_2|
 \leq\binom{n-t}{k-t}^{2}.
$
The bound is attained when both families consist of all $k$-sets
containing the same fixed $t$-set. Tokushige \cite{TokushigeCross} first established this bound in a
restricted asymptotic regime. Using the
eigenvalue method, Tokushige \cite{TokushigeEigenvalue} subsequently
extended the result to
$
 \frac{k}{n}<1-2^{-1/t}.
$
Frankl, Lee, Siggers and Tokushige
\cite{FranklLeeSiggersTokushige} obtained important partial results
towards the sharp range. Zhang and Wu \cite{ZhangWu} subsequently
proved the sharp range result for every $t\geq3$, and Tanaka and
Tokushige \cite{TanakaTokushige} completed the remaining case $t=2$.
Extensions to different uniformities appear in
\cite{ChenLiWuZhang,HeLiWuZhang}, while size-sensitive and
nontrivial refinements were obtained in
\cite{FranklKupavskiiSizeSensitive,WuXiong}.

These uniform results lead naturally to the corresponding
biased measure problem on $2^{[n]}$. In this setting, the two families may be assigned different biases, and one seeks to maximize the product of
their measures subject to the cross $t$-intersection condition.

\subsection{The biased measure conjecture and stability}

For $0<p<1$ and $\mathcal F\subseteq 2^{[n]}$, define the
\textit{$p$-biased measure} or \textit{product measure} of $\mathcal F$ by
\[
 \mu_p(\mathcal F)
 =
 \sum_{F\in\mathcal F}p^{|F|}(1-p)^{n-|F|}.
\]
Equivalently, if $X_p$ is obtained by including each element of
$[n]$ independently with probability $p$, then
\[
 \mu_p(\mathcal F)=\Pr(X_p\in\mathcal F).
\]
The study of intersection problems under biased measures was initiated
by Fishburn, Frankl, Freed, Lagarias and Odlyzko
\cite{FishburnFranklFreedLagariasOdlyzko}. Friedgut
\cite{Friedgut} later established uniqueness and stability results
for $t$-intersecting families under biased measures. Filmus
\cite{Filmus} subsequently determined the maximum
$\mu_p$-measure of a $t$-intersecting family and characterized all
extremal families for every $0<p<1$.

For $T\in\binom{[n]}{t}$, define
\[
 \mathcal S_T
 =
 \{F\subseteq[n]:T\subseteq F\}.
\]
The family $\mathcal S_T$ is called the \textit{$t$-star} with center $T$, and
$
 \mu_p(\mathcal S_T)=p^t.
$
Friedgut \cite{Friedgut} proved that, if
$0<p<1/(t+1)$ and $\mathcal F\subseteq2^{[n]}$ is
$t$-intersecting, then
\[
 \mu_p(\mathcal F)\leq p^t,
\]
with equality if and only if
$\mathcal F=\mathcal S_T$ for some
$T\in\binom{[n]}{t}$.
The threshold $1/(t+1)$ is sharp. If $n\geq t+2$, let
\[
 \mathcal A_1(n,t)
 =
 \{F\subseteq[n]:|F\cap[t+2]|\geq t+1\}.
\]
This family is $t$-intersecting, and
\[
 \mu_p\bigl(\mathcal A_1(n,t)\bigr)-\mu_p(\mathcal S_T)
 =
 p^t(1-p)\bigl((t+1)p-1\bigr).
\]
Hence, $\mathcal A_1(n,t)$ has the same measure as a $t$-star at
$p=1/(t+1)$ and a larger measure when $p>1/(t+1)$. Thus, the
$t$-star is no longer unique at the endpoint and is not extremal
beyond it.

For cross $t$-intersecting families measured with possibly different
biases $p_1$ and $p_2$, a common $t$-star provides the natural
extremal example. Indeed, for every $T\in\binom{[n]}{t}$, the pair
$(\mathcal S_T,\mathcal S_T)$ is cross $t$-intersecting and satisfies
\[
 \mu_{p_1}(\mathcal S_T)\mu_{p_2}(\mathcal S_T)
 =
 (p_1p_2)^t.
\]
Motivated by Friedgut's result and this example, Tokushige
\cite[Conjecture~1.10]{Tokushige} proposed the following conjecture.

\begin{conjecture}[Tokushige \cite{Tokushige}]\label{conj:product-measure}
Let $n \geq t \geq 1$ be integers and let $p_1, p_2 \in (0, \frac{1}{t+1})$. Let $\mathcal{F}_1, \mathcal{F}_2 \subseteq 2^{[n]}$ be cross $t$-intersecting.  Then $$\mu_{p_1}(\mathcal{F}_1)\mu_{p_2}(\mathcal{F}_2) \leq (p_1p_2)^t,$$ with equality if and only if  $\mathcal{F}_1 = \mathcal{F}_2 = \mathcal S_T$ for some $T \in \binom{[n]}{t}$.
\end{conjecture}

When $p_1=p_2=p$, Tokushige \cite{Tokushige} verified the
conjecture for every $t\geq1$ whenever
$p<1-2^{-1/t}$ and established a corresponding stability theorem.
Filmus \cite[Theorem~3.28 and Section~10.1]{FilmusThesis} obtained
the same extremal result in this range, extended the product bound
to the endpoint $p=1-2^{-1/t}$, and, for $t\geq2$, asked whether
$p^{2t}$ remains the sharp bound throughout $p\leq1/(t+1)$.
Frankl, Lee, Siggers and Tokushige
\cite{FranklLeeSiggersTokushige} proved the conjecture when
$p_1=p_2$ and $t\geq14$, and Tanaka and Tokushige
\cite{TanakaTokushige} subsequently proved it when
$p_1=p_2$ and $t=2$.

For $t=1$, Suda, Tanaka and Tokushige
\cite{SudaTanakaTokushige} proved a stronger result in which the
elements of $[n]$ are chosen independently, with their inclusion
probabilities allowed to vary both across elements and between the
two families. Thus, the case $t=1$ of
Conjecture~\ref{conj:product-measure} was already known.

In an earlier version of this paper, we observed that, for $t\geq3$,
the conjectured inequality for arbitrary
$p_1,p_2\in(0,1/(t+1))$ follows, by a limiting argument, from the
corresponding theorem for uniform families \cite{HeLiWuZhang}.
However, this argument does not yield the equality characterization.
 For $t=2$, applying the same argument to the
available uniform result gives the inequality only under additional
restrictions on $p_1$ and $p_2$ \cite{ChenLiWuZhang}.

In this paper, we prove
Conjecture~\ref{conj:product-measure} directly for all $t\geq2$.
Our proof is independent of the uniform
results, covers arbitrary
$p_1,p_2\leq1/(t+1)$, and determines all equality cases in the open
range. Combining this result with the known case $t=1$ gives the
following unified statement.

\begin{theorem}\label{thm:product-measure}
Let $n\geq t\geq1$, let
$0<p_1,p_2\leq1/(t+1)$, and let
$\mathcal F_1,\mathcal F_2\subseteq2^{[n]}$ be cross
$t$-intersecting. Then
\[
 \mu_{p_1}(\mathcal F_1)\mu_{p_2}(\mathcal F_2)
 \leq(p_1p_2)^t.
\]
If $p_1,p_2<1/(t+1)$, equality holds if and only if
$
 \mathcal F_1=\mathcal F_2=\mathcal S_T
$
for some $T\in\binom{[n]}{t}$.
\end{theorem}

The restriction to the open range in the equality statement is
necessary, since
$(\mathcal A_1(n,t),\mathcal A_1(n,t))$ is also extremal when
$n\geq t+2$ and $p_1=p_2=1/(t+1)$.

The exact theorem identifies the maximizers. It is then natural to
ask whether every nearly extremal pair must be close to one of these
maximizers, with an estimate independent of the dimension. Extending the Friedgut stability theorem  \cite[Theorem~1.6]{Friedgut},
Tokushige \cite[Conjecture~1.11]{Tokushige}  proposed the
following dimension-free stability conjecture for cross
$t$-intersecting families. Here and below,
$\triangle$ denotes symmetric difference.

\begin{conjecture}[Tokushige  \cite{Tokushige}]\label{conj:stability}
Let $t\geq1$ and let
$p_1,p_2\in(0,1/(t+1))$. Then there exists a constant
$C=C(t,p_1,p_2)>0$ such that the following holds for every
$n\geq t$ and every $\varepsilon>0$. If
$\mathcal F_1,\mathcal F_2\subseteq2^{[n]}$ are cross
$t$-intersecting and
\[
 \mu_{p_1}(\mathcal F_1)\mu_{p_2}(\mathcal F_2)
 >
 (1-\varepsilon)^2(p_1p_2)^t,
\]
then there exists $T\in\binom{[n]}{t}$ such that
$
 \mu_{p_i}\bigl(\mathcal F_i\mathbin{\triangle}\mathcal S_T\bigr)
 <
 C\sqrt{\varepsilon}
$ for $ i=1,2.$
\end{conjecture}

We prove the following stronger form, with $C\varepsilon$ in place
of $C\sqrt{\varepsilon}$.

\begin{theorem}\label{thm:stability}
Let $t\geq1$ and let
$p_1,p_2\in(0,1/(t+1))$. Then there exists a constant
$C=C(t,p_1,p_2)>0$ such that the following holds for every
$n\geq t$ and every $\varepsilon>0$. If
$\mathcal F_1,\mathcal F_2\subseteq2^{[n]}$ are cross
$t$-intersecting and
\[
 \mu_{p_1}(\mathcal F_1)\mu_{p_2}(\mathcal F_2)
 >
 (1-\varepsilon)^2(p_1p_2)^t,
\]
then there exists $T\in\binom{[n]}{t}$ such that
$
 \mu_{p_i}\bigl(\mathcal F_i\mathbin{\triangle}\mathcal S_T\bigr)
 <
 C\varepsilon
$ for 
$
i=1,2.
$
\end{theorem}

Let $C$ be the constant in Theorem~\ref{thm:stability}, enlarged if
necessary so that $C\geq1$. If $0<\varepsilon\leq1$, then
\[
 \mu_{p_i}\bigl(\mathcal F_i\mathbin{\triangle}\mathcal S_T\bigr)
 <C\varepsilon
 \leq C\sqrt{\varepsilon},
 \qquad i=1,2.
\]
 If
$\varepsilon>1$, the same conclusion holds trivially: for any
$T\in\binom{[n]}{t}$,
\[
 \mu_{p_i}\bigl(\mathcal F_i\mathbin{\triangle}\mathcal S_T\bigr)
 \leq1<C\sqrt{\varepsilon},
 \qquad i=1,2.
\]
Thus, Theorem~\ref{thm:stability} implies Conjecture~\ref{conj:stability} and improves
the dependence on $\varepsilon$ from $\sqrt{\varepsilon}$ to
$\varepsilon$.

\subsection{Cross-intersection for integer sequences}

We next consider a related conjecture of Tokushige concerning
cross-intersecting families of integer sequences.
For an integer $m\geq2$, let $[m]^n$ denote the set of all sequences
$\boldsymbol{x}=(x_1,\ldots,x_n)$ with $x_r\in[m]$. A family
$\mathcal H\subseteq[m]^n$ is called \textit{$t$-intersecting} if every two of
its members agree in at least $t$ coordinates. Two families
$\mathcal H_1,\mathcal H_2\subseteq[m]^n$ are called \textit{cross
$t$-intersecting} if, for every
$\boldsymbol{x}=(x_1,\ldots,x_n)\in\mathcal H_1$ and
$\boldsymbol{y}=(y_1,\ldots,y_n)\in\mathcal H_2$,
\[
 \left|\{r\in[n]:x_r=y_r\}\right|\geq t.
\]
Frankl and F\"uredi \cite{FranklFuredi} proved that, when
$m\geq t+1$, every $t$-intersecting family
$\mathcal H\subseteq[m]^n$ satisfies
\[
 |\mathcal H|\leq m^{n-t}.
\]
Ahlswede and Khachatrian \cite{AhlswedeKhachatrianHamming} and Frankl and Tokushige
\cite{FranklTokushigeSequences} independently determined the maximum
size for all parameters. Tokushige
\cite[Conjecture~1.12]{Tokushige} proposed the following cross
extension of the Frankl--F\"uredi bound.

\begin{conjecture}[Tokushige \cite{Tokushige}]\label{conj:integer-sequences}
Let $n,m,t$ be integers with $n\geq t\geq1$ and $m\geq t+1$. If
$\mathcal H_1,\mathcal H_2\subseteq[m]^n$ are cross
$t$-intersecting, then
\[
 |\mathcal H_1||\mathcal H_2|
 \leq m^{2(n-t)}.
\]
\end{conjecture}

Tokushige \cite{Tokushige} proved the conjecture when
$
 m>\bigl(1-2^{-1/t}\bigr)^{-1}.
$
Frankl, Lee, Siggers and Tokushige
\cite{FranklLeeSiggersTokushige} subsequently proved it for  $t\geq14$ and 
$m\geq t+1$. They also observed that an earlier result of Moon \cite{Moon}
already implies the conjectured bound, together with the equality
characterization, for all $t\geq2$ and $m\geq t+2$.
Pach and Tardos \cite{PachTardos} proved the
conjecture for $t=1$ and every $m\geq2$, and at the boundary
$m=t+1$ for every $t\geq4$. 

The following theorem establishes
Conjecture~\ref{conj:integer-sequences} for all admissible parameters
and characterizes the equality cases when $m>t+1$.

\begin{theorem}\label{thm:integer-sequences}
Let $n\geq t\geq1$ and $m\geq t+1$. If
$\mathcal H_1,\mathcal H_2\subseteq[m]^n$ are cross
$t$-intersecting, then
\[
 |\mathcal H_1||\mathcal H_2|
 \leq m^{2(n-t)}.
\]
If $m>t+1$, equality holds if and only if there exist
$T\in\binom{[n]}{t}$ and $a_r\in[m]$ such that
\[
 \mathcal H_1=\mathcal H_2
 =
 \bigl\{
 \boldsymbol{x}\in[m]^n:
 x_r=a_r\text{ for every }r\in T
 \bigr\}.
\]
\end{theorem}

\subsection{A cross extension of the intersection--union theorem}

The classical intersection--union theorem was established
independently in several works; see Daykin and Lov\'asz
\cite{DaykinLovasz} and Seymour \cite{Seymour}. It states that if
$\mathcal H\subseteq2^{[n]}$ satisfies
$
 H\cap H'\neq\emptyset
$
and
$
 H\cup H'\neq[n]
$
for all $H,H'\in\mathcal H$,
then
\[
 |\mathcal H|\leq2^{n-2}.
\]
To express these conditions in terms of integer sequences, associate
with each $H\subseteq[n]$ the sequence
$\boldsymbol{x}(H)=(x_1(H),\ldots,x_n(H))\in[2]^n$ defined by
\[
 x_r(H)=
 \begin{cases}
  1,&r\in H,\\
  2,&r\notin H.
 \end{cases}
\]
Consequently, a family satisfying the
intersection--union condition corresponds to a $2$-intersecting
family in $[2]^n$. The theorem of Frankl and F\"uredi
\cite{FranklFuredi} gives the bound $m^{n-t}$ for a
$t$-intersecting family in $[m]^n$ when $m\geq t+1$. It does not
cover the present case, since here $m=t=2$. To address this case, Frankl and Kupavskii
\cite{FranklKupavskii} introduced a refinement of the usual
intersection condition that prescribes, for each symbol, the number
of coordinates at which it occurs in both sequences.

Let $t_1,\ldots,t_m$ be nonnegative integers. A family
$\mathcal H\subseteq[m]^n$ is called
\textit{$(t_1,\ldots,t_m)$-intersecting} if, for every
$\boldsymbol{x}=(x_1,\ldots,x_n),\boldsymbol{y}=(y_1,\ldots,y_n)\in\mathcal H$ and every
$i\in[m]$,
\[
 \left|\{r\in[n]:x_r=y_r=i\}\right|\geq t_i.
\]
Frankl and Kupavskii \cite{FranklKupavskii} proved that if
$
 n\geq\sum_{i=1}^m t_i
$
and
$
 m\geq t_i+1
$
for every 
$i\in[m]$,
then every $(t_1,\ldots,t_m)$-intersecting family
$\mathcal H\subseteq[m]^n$ satisfies
\[
 |\mathcal H|
 \leq m^{\,n-\sum_{i=1}^m t_i}.
\]
When $m=2$ and $t_1=t_2=1$, this result is precisely the 
intersection--union theorem.

We now introduce the cross version of this notion. Two families
$\mathcal H_1,\mathcal H_2\subseteq[m]^n$ are called \textit{cross
$(t_1,\ldots,t_m)$-intersecting} if, for every
$\boldsymbol{x}=(x_1,\ldots,x_n)\in\mathcal H_1$,
$\boldsymbol{y}=(y_1,\ldots,y_n)\in\mathcal H_2$, and $i\in[m]$,
\[
 \left|\{r\in[n]:x_r=y_r=i\}\right|\geq t_i.
\]

The following result extends the theorem of Frankl and Kupavskii to
this cross setting.

\begin{theorem}\label{thm:vector-sequence}
Let $m\geq2$, and let $n, t_1,\ldots,t_m$ be nonnegative integers
satisfying
$
 n\geq\sum_{i=1}^m t_i
$
and
$
 m\geq t_i+1
$
for every 
$i\in[m]$.
If $\mathcal H_1,\mathcal H_2\subseteq[m]^n$ are cross
$(t_1,\ldots,t_m)$-intersecting, then
\[
 |\mathcal H_1||\mathcal H_2|
 \leq
 m^{\,2\left(n-\sum_{i=1}^m t_i\right)}.
\]
If $m>t_i+1$ for every $i\in[m]$ with $t_i>0$,
then equality holds if and only if there exist pairwise disjoint sets
$T_1,\ldots,T_m\subseteq[n]$ with $|T_i|=t_i$ such that
\[
 \mathcal H_1=\mathcal H_2
 =
 \left\{
 \boldsymbol{x}\in[m]^n:
 x_r=i\text{ for every }i\in[m]\text{ and }r\in T_i
 \right\}.
\]
\end{theorem}

Taking $m=2$ and $t_1=t_2=1$ recovers the cross
intersection--union theorem of Frankl and Kupavskii
\cite{FranklKupavskiiDownsets}.

\begin{corollary}\label{cor:cross-iu}
Let $n\geq2$, and let
$\mathcal F_1,\mathcal F_2\subseteq2^{[n]}$. Suppose that
$
 F_1\cap F_2\neq\emptyset
$
and
$
 F_1\cup F_2\neq[n]
$
for every $F_1\in\mathcal F_1$ and $F_2\in\mathcal F_2$. Then
\[
 |\mathcal F_1||\mathcal F_2|
 \leq2^{2n-4}.
\]
\end{corollary}

\subsection{Proof strategy and organization}

For $t\geq2$, we prove Theorem~\ref{thm:product-measure} using the
generating set method \cite{AhlswedeKhachatrian,Filmus,ZhangWu}.
Let $(\mathcal F_1,\mathcal F_2)$ be a pair maximizing the product of
the two measures. After making both families increasing and applying
simultaneous shifting, we consider their minimal generators. Let $s$
be the largest coordinate appearing in these generators. If $s=t$,
then both families are the same $t$-star, so it remains to exclude
$s>t$. The generators containing $s$ yield two modified pairs that
remain cross $t$-intersecting. Since the original pair is extremal,
neither modification can increase the product of the two measures.
The resulting inequalities, together with estimates for upper
shadows, rule out $s>t$ for $t\geq3$, apart from one exceptional case
that is treated separately. For $t=2$, we obtain stronger estimates
from the Kruskal--Katona theorem \cite{Katona,Kruskal} and then handle
the remaining cases directly.

To prove Theorem~\ref{thm:stability}, we apply
Theorem~\ref{thm:product-measure} at suitable nearby pairs of biases.
Together with the biased edge isoperimetric stability theorem of
Ellis, Keller and Lifshitz \cite{EllisKellerLifshitz}, the resulting
estimates show that each family is close to a $t$-star. The cross
$t$-intersection condition then implies that the two stars have the
same center.

Finally, for Theorem~\ref{thm:integer-sequences}, we use the
reduction of Frankl, Lee, Siggers and Tokushige
\cite{FranklLeeSiggersTokushige} to obtain cross $t$-intersecting
 subfamilies of $2^{[n]}$. A direct counting argument and
Theorem~\ref{thm:product-measure} then give the desired bound. For
Theorem~\ref{thm:vector-sequence}, the correlation inequality of
Frankl and Kupavskii~\cite{FranklKupavskii} separates the conditions
corresponding to the different symbols, after which
Theorem~\ref{thm:integer-sequences} applies. 

The remainder of the paper is organized as follows. In
Section~\ref{sec:generating-set}, we introduce the generating set
method and establish the required upper shadow estimates. In
Sections~\ref{sec:tge3} and~\ref{sec:t2}, we prove
Theorem~\ref{thm:product-measure} for $t\geq3$ and $t=2$,
respectively. In Section~\ref{sec:stability}, we prove
Theorem~\ref{thm:stability}. Finally, in
Section~\ref{sec:integer-sequences}, we prove the results for integer
sequences.

\section{The generating set method}\label{sec:generating-set}

We first make two reductions concerning
Theorem~\ref{thm:product-measure}. The case $t=1$, including the
equality statement in the open range, follows from
\cite[Theorem~2]{SudaTanakaTokushige} by taking both probability
vectors to be constant.
For $t\geq2$, it suffices to prove the theorem when
\[
 0<p_1,p_2<\frac{1}{t+1}.
\]
Indeed, fix $0<p_1,p_2\leq1/(t+1)$.  For every integer $m\geq2$, put
\[
 p_i^{(m)}
 =
 \left(1-\frac1m\right)p_i,
 \qquad i=1,2.
\]
Once the open range inequality has been established, it gives
\[
 \mu_{p_1^{(m)}}(\mathcal F_1)
 \mu_{p_2^{(m)}}(\mathcal F_2)
 \leq
 \bigl(p_1^{(m)}p_2^{(m)}\bigr)^t.
\]
Since the measure of a fixed family is continuous in its bias,
letting $m\to\infty$ yields
\[
 \mu_{p_1}(\mathcal F_1)\mu_{p_2}(\mathcal F_2)
 \leq(p_1p_2)^t.
\]
Thus, it remains to prove the open range statement for $t\geq2$.

\subsection{Shifting and minimal generators}\label{sec:compression}

A family $\mathcal{F}\subseteq2^{[n]}$ is called  \textit{increasing} if $F\in\mathcal{F}$ and $F\subseteq G$ imply $G\in\mathcal{F}$.  It is
called \textit{decreasing} if $F\in\mathcal F$ and $G\subseteq F$
imply $G\in\mathcal F$.  The \textit{upset} of $\mathcal{F}$ is  defined by
\[
 \mathcal F^\uparrow
 =
 \{G\subseteq[n]:F\subseteq G\text{ for some }F\in\mathcal F\}.
\]
If $\mathcal A,\mathcal B\subseteq2^{[n]}$ are cross
$t$-intersecting, then $\mathcal A^\uparrow$ and
$\mathcal B^\uparrow$ are also cross $t$-intersecting. Moreover,
$\mu_p(\mathcal F^\uparrow)\geq\mu_p(\mathcal F)$ for every
$\mathcal F\subseteq2^{[n]}$.

We use the standard shifting method; see
\cite{FranklShifting,WangShifting2026} for background and recent
developments. 
For $1\leq i<j\leq n$ and $F\in\mathcal F$, define the \textit{$ij$-shift} by
\[
 S_{ij}(F)
 =
 \begin{cases}
 (F\setminus\{j\})\cup\{i\},
 & j\in F,\ i\notin F,\text{ and }
   (F\setminus\{j\})\cup\{i\}\notin\mathcal F,\\
 F,&\text{otherwise},
 \end{cases}
\]
and put
\[
 S_{ij}(\mathcal F)
 =
 \{S_{ij}(F):F\in\mathcal F\}.
\]
A family $\mathcal F$ is called \textit{shifted} if $S_{ij}(\mathcal F)=\mathcal F$ for
every $i<j$.  We shall also use
the elementary fact that $S_{ij}(\mathcal F)$ is increasing whenever
$\mathcal F$ is increasing.

\begin{lemma}[\cite{FranklLeeSiggersTokushige}]\label{lem:shifting-reduction}
Let $\mathcal A,\mathcal B\subseteq2^{[n]}$ be  cross
$t$-intersecting. Then there is a cross $t$-intersecting pair
$(\mathcal A',\mathcal B')$ such that both families are 
shifted, and for every $0<p<1$,
\[
 \mu_p(\mathcal A')=\mu_p(\mathcal A),
 \qquad
 \mu_p(\mathcal B')=\mu_p(\mathcal B).
\]
\end{lemma}

The following consequence of shifting is contained in
\cite[Lemma~2.3 (iv)]{FranklLeeSiggersTokushige}.

\begin{lemma}\label{lem:inverse-shift-star}
Let $1\leq t\leq n$, let $1\leq i<j\leq n$, and let
$\mathcal F\subseteq2^{[n]}$ be increasing. If
$
 S_{ij}(\mathcal F)=\mathcal S_T
$
for some $T\in\binom{[n]}{t}$, then
$
 \mathcal F=\mathcal S_{T'}
$
for some $T'\in\binom{[n]}{t}$.
\end{lemma}
\begin{proof}
An $ij$-shift preserves the number of members of each size. Since
$\mathcal S_T$ has exactly one $t$-element member, so does
$\mathcal F$, denote it by $T'$. As $\mathcal F$ is increasing,
$\mathcal S_{T'}\subseteq\mathcal F$. Moreover,
$
 |\mathcal F|
 =
 |S_{ij}(\mathcal F)|
 =
 |\mathcal S_T|
 =
 2^{n-t}
 =
 |\mathcal S_{T'}|.
$
Thus, $\mathcal F=\mathcal S_{T'}$.
\end{proof}

Let $\mathcal F\subseteq2^{[n]}$ be increasing. A
\textit{generating set} of $\mathcal F$ is an inclusion-minimal member
of $\mathcal F$. The collection of all generating sets of
$\mathcal F$ is called its \textit{generating family} and is denoted
by $G(\mathcal F)$. Since $\mathcal F$ is increasing,
\[
 \mathcal F
 =
 \bigcup_{E\in G(\mathcal F)}
 \{H\subseteq[n]:E\subseteq H\}.
\]
For $1\leq r\leq n$ and $H\subseteq[r]$, define
\[
 \mathcal D_r(H)
 =
 \{F\subseteq[n]:F\cap[r]=H\},
\]
and, for $\mathcal H\subseteq2^{[r]}$, define
\[
 \mathcal D_r(\mathcal H)
 =
 \bigcup_{H\in\mathcal H}\mathcal D_r(H).
\]

Let $n\geq t\geq1$, and let $\mathcal A,\mathcal B\subseteq2^{[n]}$ be nonempty increasing
families that are cross $t$-intersecting. Set
\[
 s
 =
 \max\{\max E:E\in G(\mathcal A)\cup G(\mathcal B)\},
\]
where $\max E$ denotes the largest element of $E$. A generating set
containing $s$ is called a \textit{boundary generating set}. For
$t\leq r\leq s$, define
\[
\mathcal X_r
 =
 \{E\in G(\mathcal A):s\in E,\ |E|=r\},
 \qquad
\mathcal Y_r
 =
 \{F\in G(\mathcal B):s\in F,\ |F|=r\}.
\]
and
\[
\mathcal X_r^-=\{E\setminus\{s\}:E\in \mathcal X_r\},
 \qquad
\mathcal Y_r^-=\{F\setminus\{s\}:F\in \mathcal Y_r\}.
\]
A cross $t$-intersecting pair $(\mathcal A,\mathcal B)$ is called
\textit{mutually inclusion-maximal} if neither family can be enlarged,
with the other family fixed, while preserving cross $t$-intersection.
Every mutually inclusion-maximal cross $t$-intersecting pair consists
of two increasing families.

\begin{lemma}\label{lem:boundary}
Let $n\geq t\geq1$, and let
$\mathcal A,\mathcal B\subseteq2^{[n]}$ be nonempty shifted families
forming a mutually inclusion-maximal cross $t$-intersecting pair.
If $\mathcal X_i\neq\emptyset$ and $j=s+t-i$, then $\mathcal  Y_j\neq\emptyset$.
More precisely, for every $E\in \mathcal X_i$, there exists $F\in \mathcal  Y_j$ such
that
\[
 |E\cap F|=t
 \qquad\text{and}\qquad
 E\cup F=[s].
\]
Moreover, both pairs
\begin{equation}\label{eq:boundary-modifications}
 \left(
  \mathcal A\cup\mathcal D_s(\mathcal X_i^-),
  \mathcal B\setminus\mathcal D_s(\mathcal Y_j)
 \right)
 \quad\text{and}\quad
 \left(
  \mathcal A\setminus\mathcal D_s(\mathcal X_i),
  \mathcal B\cup\mathcal D_s(\mathcal Y_j^-)
 \right)
\end{equation}
are cross $t$-intersecting. In addition,
\[
 \mathcal D_s(\mathcal X_i^-)\cap\mathcal A=\emptyset,\qquad
 \mathcal D_s(\mathcal X_i)\subseteq\mathcal A,\qquad
 \mathcal D_s(\mathcal Y_j^-)\cap\mathcal B=\emptyset,\qquad
 \mathcal D_s(\mathcal Y_j)\subseteq\mathcal B.
\]
\end{lemma}
\begin{proof}
By the definition of $s$, every generating set of either family is
contained in $[s]$. We first note that if
$E_0\in G(\mathcal A)$ and $F_0\in G(\mathcal B)$ satisfy
$
 s\in E_0\cap F_0
$
and
$
 |E_0\cap F_0|=t,
$
then
\[
 E_0\cup F_0=[s].
\]
Indeed, if some $x\in[s-1]\setminus(E_0\cup F_0)$ existed, then
shiftedness would imply
$
 (E_0\setminus\{s\})\cup\{x\}\in\mathcal A.
$
However,
$
 \left|
 \bigl((E_0\setminus\{s\})\cup\{x\}\bigr)\cap F_0
 \right|
 =t-1,
$
contradicting cross $t$-intersection. Consequently,
\[
 E_0\cup F_0=[s]
 \qquad\text{and}\qquad
 |E_0|+|F_0|=s+t.
\]

Fix $E\in \mathcal X_i$. Since $E$ is a generating set,
$E\setminus\{s\}\notin\mathcal A$. By mutual
inclusion-maximality, there is $B_0\in\mathcal B$ such that
$
 |(E\setminus\{s\})\cap B_0|\leq t-1.
$
Since $E\in\mathcal A$ and the two families are cross
$t$-intersecting,
\[
 t
 \leq |E\cap B_0|
 \leq |(E\setminus\{s\})\cap B_0|+1
 \leq t.
\]
It follows that
\[
 s\in B_0,\qquad
 |(E\setminus\{s\})\cap B_0|=t-1,
 \qquad
 |E\cap B_0|=t.
\]
Choose $F\in G(\mathcal B)$ with $F\subseteq B_0$. Then
$
 t\leq |E\cap F|\leq |E\cap B_0|=t.
$
Hence, $E\cap F=E\cap B_0$, and in particular $s\in F$. The preceding
observation gives
\[
 E\cup F=[s]
 \qquad\text{and}\qquad
 |F|=s+t-|E|=s+t-i=j.
\]
Therefore, $F\in\mathcal Y_j$.

We next consider the first pair in
\eqref{eq:boundary-modifications}. Let $E\in\mathcal X_i$ and
$A'\in\mathcal D_s(E\setminus\{s\})$. Suppose that
$B'\in\mathcal B$ satisfies $|A'\cap B'|<t$. Since
$E\in\mathcal A$,
$
 t
 \leq |E\cap B'|
 \leq |A'\cap B'|+1
 \leq t.
$
Therefore,
\[
 s\in B',\qquad
 |A'\cap B'|=t-1,
 \qquad
 |E\cap B'|=t.
\]
Choose $F'\in G(\mathcal B)$ with $F'\subseteq B'$. Then
$
 t\leq |E\cap F'|\leq |E\cap B'|=t.
$
Thus, $E\cap F'=E\cap B'$ and $s\in F'$. By the preceding observation,
\[
 E\cup F'=[s]
 \qquad\text{and}\qquad
 |F'|=s+t-i=j.
\]
Hence, $F'\in\mathcal Y_j$.

We claim that $B'\cap[s]=F'$. Otherwise, take
$x\in(B'\cap[s])\setminus F'$. Since $E\cup F'=[s]$, we have
$x\in E$, contradicting $E\cap F'=E\cap B'$. Thus,
$B'\in\mathcal D_s(\mathcal Y_j)$. It follows that every member of
$\mathcal B$ that fails to cross $t$-intersect a set in
$\mathcal D_s(\mathcal X_i^-)$ is removed from the first pair in
\eqref{eq:boundary-modifications}. Hence, this pair is  cross
$t$-intersecting. The second pair is cross $t$-intersecting by
symmetry.

It remains to prove the four set relations. Fix $E\in \mathcal X_i$. If some
$A'\in\mathcal D_s(E\setminus\{s\})$ belonged to $\mathcal A$, then
$A'$ would contain a generating set $K\in G(\mathcal A)$. Since
$K\subseteq[s]$,
$
 K\subseteq A'\cap[s]=E\setminus\{s\},
$
contradicting the minimality of $E$. Hence,
$
 \mathcal D_s(\mathcal X_i^-)\cap\mathcal A=\emptyset.
$
If $A'\in\mathcal D_s(E)$, then $E\subseteq A'$, so increasingness
gives $A'\in\mathcal A$. Therefore,
$
 \mathcal D_s(\mathcal X_i)\subseteq\mathcal A.
$
The corresponding relations for $\mathcal Y_j^-$ and $\mathcal Y_j$ follow
symmetrically.
\end{proof}

\subsection{Boundary inequalities}\label{sec:boundary-tools}

Fix $n\geq t\geq 1$ and $0<p_1,p_2\leq1/(t+1)$, and let $\mathcal A,\mathcal B\subseteq2^{[n]}$ be  cross
$t$-intersecting families maximizing
\[
 \mu_{p_1}(\mathcal A)\mu_{p_2}(\mathcal B).
\]
A common $t$-star has positive product, so both families are nonempty.
Product maximality implies that the pair is mutually
inclusion-maximal and hence increasing. By
Lemma \ref{lem:shifting-reduction},
we may assume that both families are shifted. The resulting pair
remains product-maximizing and  mutually inclusion-maximal and increasing.

Let $s$, $\mathcal X_r$, $\mathcal Y_r$, $\mathcal X_r^-$, and
$\mathcal Y_r^-$ be as defined in the preceding subsection. At least
one boundary layer is nonempty. After interchanging
$(\mathcal A,p_1)$ and $(\mathcal B,p_2)$ if necessary, choose $i$
such that $\mathcal X_i\neq\emptyset$ and put
$
 j=s+t-i.
$
Lemma~\ref{lem:boundary} gives $\mathcal Y_j\neq\emptyset$. Write
\[
 q_1=1-p_1,\qquad q_2=1-p_2,\qquad
 \alpha=\mu_{p_1}(\mathcal A),\qquad
 \beta=\mu_{p_2}(\mathcal B),
\]
and define
\begin{align*}
 x
 &=\mu_{p_1}\bigl(\mathcal D_s(\mathcal X_i)\bigr)
   =|\mathcal X_i|p_1^iq_1^{s-i},
&
 x^+
 &=\mu_{p_1}\bigl(\mathcal D_s(\mathcal X_i^-)\bigr)
   =\frac{q_1}{p_1}x,\\
 y
 &=\mu_{p_2}\bigl(\mathcal D_s(\mathcal Y_j)\bigr)
   =|\mathcal Y_j|p_2^jq_2^{s-j},
&
 y^+
 &=\mu_{p_2}\bigl(\mathcal D_s(\mathcal Y_j^-)\bigr)
   =\frac{q_2}{p_2}y.
\end{align*}
By Lemma~\ref{lem:boundary} and the product maximality,
\[
 (\alpha+x^+)(\beta-y)\leq\alpha\beta,
 \qquad
 (\alpha-x)(\beta+y^+)\leq\alpha\beta.
\]
Equivalently,
\[
 \beta x^+\leq y(\alpha+x^+),
 \qquad
 \alpha y^+\leq x(\beta+y^+).
\]
Since $x,y>0$, multiplying these inequalities yields
\begin{equation}\label{eq:boundary-upper}
 \frac{\alpha}{\alpha+x^+}
 \frac{\beta}{\beta+y^+}
 \leq
 \frac{p_1p_2}{q_1q_2}.
\end{equation}

We next obtain lower bounds for the two factors on the left. For
$\mathcal H\subseteq\binom{[s-1]}{r-1}$, define its \textit{upper shadow} by
\[
 \nabla\mathcal H
 =
 \left\{
  K\in\binom{[s-1]}r:
  H\subseteq K\text{ for some }H\in\mathcal H
 \right\}.
\]
Double counting the pairs $(H,K)$ with
$H\in\mathcal X_i^-$, $K\in\nabla\mathcal X_i^-$, and $H\subseteq K$
gives
\[
 (s-i)|\mathcal X_i^-|
 \leq i|\nabla\mathcal X_i^-|.
\]
Since $|\mathcal X_i^-|=|\mathcal X_i|$, it follows that
\[
 |\nabla\mathcal X_i^-|
 \geq\frac{s-i}{i}|\mathcal X_i|.
\]

For every $K\in\nabla\mathcal X_i^-$, there exist
$E\in\mathcal X_i$ and $\ell\in[s-1]\setminus E$ such that
$
 K=(E\setminus\{s\})\cup\{\ell\}.
$
Since $\ell<s$, shiftedness gives $K\in\mathcal A$. As
$\mathcal A$ is increasing,
\[
 \mathcal D_{s-1}\bigl(\nabla\mathcal X_i^-\bigr)
 \subseteq\mathcal A.
\]
For distinct $K,K'\in\nabla\mathcal X_i^-$, we have
$
 \mathcal D_{s-1}(K)\cap\mathcal D_{s-1}(K')=\emptyset.
$
Moreover,
\[
 \mathcal D_{s-1}\bigl(\nabla\mathcal X_i^-\bigr)
 \cap\mathcal D_s(\mathcal X_i)=\emptyset,
\]
since every set in the first family contains exactly $i$ elements of
$[s-1]$, whereas every set in the second contains exactly $i-1$
elements of $[s-1]$. Consequently,
\[
 \alpha
 \geq
 x+|\nabla\mathcal X_i^-|p_1^iq_1^{s-i-1}\geq
 x\left(1+\frac{s-i}{iq_1}\right).
\]
By symmetry,
\[
 \beta\geq
 y\left(1+\frac{s-j}{jq_2}\right).
\]

For $1\leq r\leq s$ and $0<p<1$, put
\[
 R_{s,r}(p)
 =
 \frac{p(s-rp)}{r(1-p)+(s-r)p}.
\]
Since $a\mapsto a/(a+c)$ is increasing for every $c>0$, and
$
 x^+=\frac{q_1}{p_1}x,
 y^+=\frac{q_2}{p_2}y,
$
the preceding estimates give
\begin{equation}\label{eq:boundary-lower}
\begin{aligned}
 \frac{\alpha}{\alpha+x^+}
 &\geq
 \frac{1+\dfrac{s-i}{iq_1}}
      {1+\dfrac{s-i}{iq_1}+\dfrac{q_1}{p_1}}
 =R_{s,i}(p_1),\\
 \frac{\beta}{\beta+y^+}
 &\geq
 \frac{1+\dfrac{s-j}{jq_2}}
      {1+\dfrac{s-j}{jq_2}+\dfrac{q_2}{p_2}}
 =R_{s,j}(p_2).
\end{aligned}
\end{equation}
Define
\begin{equation}\label{eq:boundary-H}
 H_{s,r}(p)
 =
 \frac{1-p}{p}R_{s,r}(p)
 =
 \frac{(1-p)(s-rp)}
      {r(1-p)+(s-r)p}.
\end{equation}
Combining \eqref{eq:boundary-upper} and
\eqref{eq:boundary-lower} shows that a product-maximizing pair with
boundary sizes $i$ and $j$ must satisfy
\[
 H_{s,i}(p_1)H_{s,j}(p_2)\leq1.
\]
Thus, such a boundary type is excluded whenever
\begin{equation}\label{eq:boundary-criterion}
 H_{s,i}(p_1)H_{s,j}(p_2)>1,
 \qquad i+j=s+t.
\end{equation}

The following scalar estimate verifies this criterion except for one
parameter pair.

\begin{lemma}\label{lem:algebra}
Let $t\geq3$, $s\geq t+3$, and
$t+1\leq i,j\leq s$ satisfy $i+j=s+t$. If
$(t,s)\neq(3,6)$, then
\[
 H_{s,i}(p_1)H_{s,j}(p_2)>1
\]
for all $0<p_1,p_2\leq1/(t+1)$.
\end{lemma}

\begin{proof}
For $r\in\{i,j\}$, differentiation gives
\[
 H_{s,r}'(p)
 =
 \frac{N_{s,r}(p)}
 {\bigl(r(1-p)+(s-r)p\bigr)^2},
\]
where
\[
 N_{s,r}(p)
 =
 -(s^2-sr+r^2)+2r^2p+r(s-2r)p^2.
\]
The other boundary size is at least $t+1$, so $r<s$. Moreover,
\[
 N_{s,r}'(p)
 =
 2r\bigl((1-p)r+p(s-r)\bigr)>0
 \qquad (0\leq p\leq1),
\]
while
\[
 N_{s,r}(1)=-(s-r)^2<0.
\]
Hence, $N_{s,r}(p)<0$ on $[0,1]$, and $H_{s,r}$ is strictly decreasing
on $(0,1)$.

Put $b=1/(t+1)$. It is sufficient to prove
\[
 H_{s,i}(b)H_{s,j}(b)>1.
\]
At $p=b$,
\[
 H_{s,r}(b)
 =
 \frac{t((t+1)s-r)}
 {(t+1)(s+(t-1)r)}.
\]
After clearing the positive denominators, the desired inequality is
equivalent to $\Delta>0$, where
\[
 \Delta
 =
 t^2((t+1)s-i)((t+1)s-j)
 -(t+1)^2(s+(t-1)i)(s+(t-1)j).
\]
Using $i+j=s+t$, we can write
\[
 \Delta
 =
 t(t+1)s\bigl((t^2-t-1)s-2t^2+1\bigr)
 -(t^4-3t^2+1)ij.
\]
Since $t^4-3t^2+1>0$ and
\[
 ij\leq\frac{(s+t)^2}{4},
\]
we obtain
\[
\begin{aligned}
\Delta
&\geq
t(t+1)s\bigl((t^2-t-1)s-2t^2+1\bigr)
-\frac{t^4-3t^2+1}{4}(s+t)^2\\
&=
\frac{1}{4}
\bigl((t^2-t-1)s-t(t^2+t-1)\bigr)
\bigl((3t^2+3t+1)s+t(t^2-t-1)\bigr).
\end{aligned}
\]
The second factor is positive, so it remains to prove that
\[
 (t^2-t-1)s>t(t^2+t-1).
\]
If $t\geq4$, then $s\geq t+3$ implies 
\begin{align*}
 (t^2-t-1)s-t(t^2+t-1)
 &\geq
 (t^2-t-1)(t+3)-t(t^2+t-1)\\
 &=t^2-3t-3>0.
\end{align*}
If $t=3$, then $s\geq t+3$ and $(t,s)\neq(3,6)$ imply
$s\geq7$. Hence,
\[
 (t^2-t-1)s-t(t^2+t-1)
 =5s-33>0.
\]
Therefore, $\Delta>0$ and the result follows.
\end{proof}

Thus, within the range $t\geq3$, $s\geq t+3$, and
$t+1\leq i,j\leq s$, the only parameter pair not covered by
Lemma~\ref{lem:algebra} is $(t,s)=(3,6)$.

\section{The exact theorem for $t\geq3$}
\label{sec:tge3}

\subsection{The cases $i=t$ and $s=t+2$}

We first treat two boundary cases. Both lemmas below are stated for
$t\geq1$, since they will also be used in Section~\ref{sec:t2}.

\begin{lemma}\label{lem:size-t}
Let $n\geq t\geq1$ and $0<p_1,p_2<1/(t+1)$. Let
$\mathcal A,\mathcal B\subseteq2^{[n]}$ be nonempty shifted families
forming a mutually inclusion-maximal cross $t$-intersecting pair.
With the notation of Subsection~\ref{sec:compression}, suppose that
$s>t$ and
$\mathcal X_t\cup\mathcal Y_t\neq\emptyset$. Then
\[
 \mu_{p_1}(\mathcal A)\mu_{p_2}(\mathcal B)<(p_1p_2)^t.
\]
\end{lemma}

\begin{proof}
After interchanging $(\mathcal A,p_1)$ and
$(\mathcal B,p_2)$ if necessary, assume that
$\mathcal X_t\neq\emptyset$. Lemma~\ref{lem:boundary} gives
$\mathcal Y_s\neq\emptyset$. Hence, $[s]\in G(\mathcal B)$.
Since every generating set of $\mathcal B$ is contained in $[s]$, we have $G(\mathcal B)=\{[s]\}$ and hence
\[
 \mathcal B=\{B\subseteq[n]:[s]\subseteq B\},
 \qquad
 \mu_{p_2}(\mathcal B)=p_2^s.
\]
Since $[s]\in\mathcal B$, every $A\in\mathcal A$ satisfies
$|A\cap[s]|\geq t$.  Hence,
$
 \mathcal A
 \subseteq
 \bigcup_{T\in\binom{[s]}t}
 \{A\subseteq[n]:T\subseteq A\}.
$
It follows that
$\mu_{p_1}(\mathcal A)\leq\binom{s}{t}p_1^t$. Writing
$s=t+r$, where $r\geq1$, we obtain
\[
 \frac{\mu_{p_1}(\mathcal A)\mu_{p_2}(\mathcal B)}
      {(p_1p_2)^t}
 \leq\binom{t+r}{t}p_2^r=
 \left(\prod_{\ell=1}^r\frac{t+\ell}{\ell}\right)p_2^r\leq\bigl((t+1)p_2\bigr)^r<1,
\]
as desired.
\end{proof}

We next consider the case $s=t+2$.

\begin{lemma}\label{lem:excess-two}
Let $t\geq1$, put $s=t+2$, and let
$0<p_1,p_2<1/(t+1)$. If
$\mathcal A_0,\mathcal B_0\subseteq2^{[s]}$ are increasing and cross
$t$-intersecting, then
\[
 \mu_{p_1}(\mathcal A_0)\mu_{p_2}(\mathcal B_0)
 \leq(p_1p_2)^t.
\]
Equality holds if and only if $\mathcal A_0$ and $\mathcal B_0$ are
the same $t$-star.
\end{lemma}

\begin{proof}
The conclusion is immediate if either family is empty. Otherwise,
define
\[
 \mathcal C=\{[s]\setminus A:A\in\mathcal A_0\},
 \qquad
 \mathcal E=\{[s]\setminus B:B\in\mathcal B_0\}.
\]
These families are closed under taking subsets, and cross
$t$-intersection is equivalent to
$|C\cup E|\leq2$ for all
$C\in\mathcal C$ and $E\in\mathcal E$.
Put $u=(1-p_1)/p_1$ and $v=(1-p_2)/p_2$. Then
$u,v>t=s-2$. For $z>0$ and $\mathcal H\subseteq2^{[s]}$, write
$W_z(\mathcal H)=\sum_{H\in\mathcal H}z^{|H|}$. For $C=[s]\setminus A$, the $p_1$-weight of $A$ is
$p_1^{s-|C|}(1-p_1)^{|C|}=p_1^su^{|C|}$. Moreover,
$p_1^{s-t}=p_1^2=(1+u)^{-2}$, and the same identities hold for
$p_2$, $v$, and $\mathcal E$.
Taking complements gives
\[
 \frac{\mu_{p_1}(\mathcal A_0)}{p_1^t}
 =
 \frac{W_u(\mathcal C)}{(1+u)^2},
 \qquad
 \frac{\mu_{p_2}(\mathcal B_0)}{p_2^t}
 =
 \frac{W_v(\mathcal E)}{(1+v)^2}.
\]

The families $\mathcal C$ and $\mathcal E$ are decreasing. Since all
weights are positive, we may enlarge them to a maximal pair of
decreasing families satisfying $|C\cup E|\leq2$ for all
$C\in\mathcal C$ and $E\in\mathcal E$.
Both families contain $\emptyset$, so all their members have size at
most two. 

\textbf{Suppose first that both families contain a two-element set.}
The union condition forces every such set on either side to be the
same set $R$. It also forces every member of both families to be
contained in $R$. Maximality now gives
$\mathcal C=\mathcal E=2^R$. Taking complements,
\[
 \mathcal A_0=\mathcal B_0
 =
 \{A\subseteq[s]:[s]\setminus R\subseteq A\}.
\]
Since $|[s]\setminus R|=t$, these families form a common $t$-star,
and equality holds.

\textbf{If neither family contains a two-element set,}  maximality gives
$
 \mathcal C=\mathcal E=\{F\subseteq[s]:|F|\leq1\}.
$
For $z>s-2$,
\[
 \frac{1+sz}{(1+z)^2}<1,
\]
because $(1+z)^2-(1+sz)=z(z-(s-2))>0$. Since
$
 W_u(\mathcal C)=1+su,
 W_v(\mathcal E)=1+sv,
$
and $u,v>s-2$, we have
\[
 \frac{\mu_{p_1}(\mathcal A_0)\mu_{p_2}(\mathcal B_0)}
      {(p_1p_2)^t}
 =
 \frac{1+su}{(1+u)^2}
 \frac{1+sv}{(1+v)^2}
 <1.
\]

\textbf{It remains to assume, after interchanging the two families if
necessary, that $\mathcal C$ contains a two-element set and
$\mathcal E$ does not.} Let
$U=\{r\in[s]:\{r\}\in\mathcal E\}$ and put $m=|U|$. Every
two-element member of $\mathcal C$ must contain $U$, so $m\leq2$.
Maximality gives
\[
 \mathcal E=\{\emptyset\}\cup\binom U1,
 \qquad
 \mathcal C=
 \{\emptyset\}\cup\binom{[s]}1
 \cup\left\{R\in\binom{[s]}2:U\subseteq R\right\}.
\]
If $m=0$, then
$W_u(\mathcal C)/(1+u)^2\leq\binom{s}{2}$, whereas
$W_v(\mathcal E)/(1+v)^2<1/(s-1)^2$. Their product is  less
than $s/(2(s-1))<1$.
If $m=1$, then
$W_u(\mathcal C)/(1+u)^2\leq s-1$, whereas
$W_v(\mathcal E)/(1+v)^2=1/(1+v)<1/(s-1)$. The product is again
strictly less than one.
Finally, suppose that $m=2$. For $z>0$, put
\[
 f_s(z)=\frac{1+sz+z^2}{(1+z)^2},
 \qquad
 g(z)=\frac{1+2z}{(1+z)^2}.
\]
Then
\[
 \frac{\mu_{p_1}(\mathcal A_0)}{p_1^t}=f_s(u),
 \qquad
 \frac{\mu_{p_2}(\mathcal B_0)}{p_2^t}=g(v).
\]
For $z>1$,
\[
 f_s'(z)=\frac{(s-2)(1-z)}{(1+z)^3}<0,
 \qquad
 g'(z)=-\frac{2z}{(1+z)^3}<0.
\]
Since $u,v>t$ and $s=t+2$,
\[
 f_s(u)g(v)
 <f_s(t)g(t)=\frac{(2t^2+2t+1)(2t+1)}{(t+1)^4}=1-\frac{t^4}{(t+1)^4}<1.
\]

 Since adding a set strictly
increases the corresponding weight, equality occurs only in the first
case, namely when $\mathcal A_0$ and $\mathcal B_0$ are the same
$t$-star.
\end{proof}

To apply Lemma~\ref{lem:excess-two} to the pair
$(\mathcal A,\mathcal B)$ from Subsection~\ref{sec:boundary-tools},
put
\[
 \mathcal A_0=\mathcal A\cap2^{[s]},
 \qquad
 \mathcal B_0=\mathcal B\cap2^{[s]}.
\]
Then
$\mathcal A=\mathcal D_s(\mathcal A_0)$ and
$\mathcal B=\mathcal D_s(\mathcal B_0)$.
It follows that 
\[
 \mu_{p_1}(\mathcal A)=\mu_{p_1}(\mathcal A_0),
 \qquad
 \mu_{p_2}(\mathcal B)=\mu_{p_2}(\mathcal B_0).
\]
Moreover, $\mathcal A_0$ and $\mathcal B_0$ are increasing and cross
$t$-intersecting. Consequently, Lemma~\ref{lem:excess-two} gives the required estimate
for $(\mathcal A,\mathcal B)$ when $s=t+2$.

\subsection{The exceptional case $t=3$ and $s=6$}

We consider the case $t=3$ and $s=6$. For
$\mathcal H\subseteq2^{[6]}$ and $z>0$, write
$
 W_z(\mathcal H)=\sum_{H\in\mathcal H}z^{|H|}.
$
The following lemma gives the required estimate.

\begin{lemma}\label{lem:six}
Let $\mathcal C,\mathcal E\subseteq2^{[6]}$ be decreasing families
such that
$
 |C\cup E|\leq3
$
for every $C\in\mathcal C$ and $E\in\mathcal E$. If $u,v>3$, then
\[
 \frac{W_u(\mathcal C)}{(1+u)^3}
 \frac{W_v(\mathcal E)}{(1+v)^3}
 \leq1.
\]
Equality holds if and only if
$
 \mathcal C=\mathcal E=2^R
$
for some $R\in\binom{[6]}{3}$.
\end{lemma}

\begin{proof}
The result is immediate if either family is empty. Assume that both families are nonempty. Since they are decreasing,
they both contain $\emptyset$, and the hypothesis implies that every
member of either family has size at most three.
For a nonempty decreasing family $\mathcal H\subseteq2^{[6]}$, put
\[
 h_r=\left|\mathcal H\cap\binom{[6]}{r}\right|,
 \qquad
 f_{\mathcal H}(z)
 =
 \frac{W_z(\mathcal H)}{(1+z)^3}
 =
 \frac{1+h_1z+h_2z^2+h_3z^3}{(1+z)^3}.
\]
If $h_3=0$, differentiation gives
\[
 (1+z)^4f_{\mathcal H}'(z)
 =
 (h_1-3)+2(h_2-h_1)z-h_2z^2.
\]
The derivative of the polynomial on the right is
$
 -2h_1-2h_2(z-1)\leq0
$
for $z\geq1$, while its value at $z=3$ is
$
 -3-5h_1-3h_2<0.
$
Consequently,
\begin{equation}\label{eq:six-endpoint}
 f_{\mathcal H}(z)
 <
 f_{\mathcal H}(3)
 =
 \frac{1+3h_1+9h_2}{64}
 \qquad (z>3).
\end{equation}

\textbf{Suppose first that neither $\mathcal C$ nor $\mathcal E$ contains a
three-element set.} Let
$
 \mathcal P=\mathcal C\cap\binom{[6]}{2}
$
and
$
 \mathcal Q=\mathcal E\cap\binom{[6]}{2},
$
and put $a=|\mathcal P|$ and $b=|\mathcal Q|$. Every member of
$\mathcal P$ intersects every member of $\mathcal Q$.

If $a=0$, then \eqref{eq:six-endpoint}, together with
$h_1\leq6$ and $b\leq15$, gives
\[
 f_{\mathcal C}(u)f_{\mathcal E}(v)
 <
 \frac{19(19+9\cdot15)}{64^2}
 <1.
\]
The case $b=0$ is symmetric. We may therefore assume, after
interchanging $(\mathcal C,u)$ and $(\mathcal E,v)$ if necessary,
that $1\leq a\leq b$.

We claim that $a+b\leq10$. If $a=1$, at most nine two-element
subsets of $[6]$ intersect the unique member of $\mathcal P$, so
$b\leq9$. If $a=2$, at most six two-element subsets intersect both
members of $\mathcal P$. Finally, if $a\geq3$, choose three members of
$\mathcal P$. There are at most five two-element subsets of $[6]$
that intersect all three, with the maximum attained when the three
chosen sets have a common element. Hence, $b\leq5$, and the claim
follows from $a\leq b$.

By \eqref{eq:six-endpoint} and the arithmetic--geometric mean
inequality,
\[
 f_{\mathcal C}(u)f_{\mathcal E}(v)
 <
 \frac{(19+9a)(19+9b)}{64^2}\leq
 \frac{\bigl(19+\frac{9}{2}(a+b)\bigr)^2}{64^2}
 \leq1.
\]

\textbf{Next suppose that both families contain a three-element set.} Choose
$
 R\in\mathcal C\cap\binom{[6]}{3}
$
and
$
 T\in\mathcal E\cap\binom{[6]}{3}.
$
Since $|R\cup T|\leq3$, we have $R=T$. The same condition shows that
every member of either family is contained in $R$. Since the families
are decreasing and contain $R$, it follows that
$
 \mathcal C=\mathcal E=2^R.
$
In this case equality holds.

\textbf{It remains to consider the case in which exactly one family contains
a three-element set.} By symmetry, assume that $\mathcal E$ does and
$\mathcal C$ does not. Let
$
 S=\bigcup_{C\in\mathcal C}C,
$
put $\sigma=|S|$, and define
$
 \mathcal P=\mathcal C\cap\binom{S}{2}
$
and $a=|\mathcal P|$. Every three-element member of $\mathcal E$
contains every member of $\mathcal C$ and hence contains $S$.
Therefore,
$
 0\leq\sigma\leq3
$
and
$
 0\leq a\leq\binom{\sigma}{2}.
$
Since $\mathcal C$ is decreasing,
$
 W_u(\mathcal C)=1+\sigma u+au^2.
$
Put $M=1+3\sigma+9a$. By \eqref{eq:six-endpoint},
\[
 f_{\mathcal C}(u)<\frac{M}{64}.
\]

For $0\leq r\leq3$, let
$
 e_r=|\mathcal E\cap\binom{[6]}{r}|.
$
Every three-element member of $\mathcal E$ contains $S$, while every
two-element member of $\mathcal E$ intersects every member of
$\mathcal P$. Hence,
\[
 e_1\leq6,\qquad
 e_2\leq b_a,\qquad
 e_3\leq c_\sigma,
\]
where
\[
 c_\sigma=\binom{6-\sigma}{3-\sigma},
 \qquad
 b_0=15,
 \qquad
 b_a=12-3a\quad(1\leq a\leq3).
\]
Indeed, the bound for $a=0$ is immediate. For $a\geq1$, choose $S'\in\binom{[6]}{3}$ with $S\subseteq S'$.
Since $\mathcal P$ consists of $a$ distinct two-element subsets of
$S'$, we have
$
 \left|\bigcap_{P\in\mathcal P}P\right|=3-a.
$
A two-element set meeting every member of $\mathcal P$ is either
contained in $S'$, or has one element in this intersection and the
other in $[6]\setminus S'$. Hence,
$
 e_2\leq3+3(3-a)=12-3a.
$

It follows that
\[
 f_{\mathcal E}(v)
 \leq
 \frac{P_{\sigma,a}(v)}{(1+v)^3},
 \qquad
 P_{\sigma,a}(v)
 =
 1+6v+b_av^2+c_\sigma v^3.
\]

Suppose first that $a\leq2$. Then $M\leq28$. For $a=0,1,2$,
respectively,
$
 Mb_a\leq150,171,168.
$
Similarly, for $\sigma=0,1,2,3$, respectively,
$
 Mc_\sigma\leq20,40,64,28.
$
Together with $M\leq28$ and $6M\leq168$, these inequalities give,
by comparing coefficients,
$
 MP_{\sigma,a}(v)\leq64(1+v)^3.
$
Therefore,
\[
 f_{\mathcal C}(u)f_{\mathcal E}(v)
 <
 \frac{M}{64}
 \frac{P_{\sigma,a}(v)}{(1+v)^3}
 \leq1.
\]

The only remaining possibility is $a=3$, which forces $\sigma=3$.
In this case,
$
 M=37
$
and
$
 P_{3,3}(v)=1+6v+3v^2+v^3.
$
For $v>3$,
\[
 64(1+v)^3-37P_{3,3}(v)
 =
 27v^3+81v^2-30v+27
 >0.
\]
Since
$
 f_{\mathcal C}(u)<37/64
$
and
$
 f_{\mathcal E}(v)\leq P_{3,3}(v)/(1+v)^3,
$
the preceding inequality gives
\[
 f_{\mathcal C}(u)f_{\mathcal E}(v)
 <
 \frac{37P_{3,3}(v)}{64(1+v)^3}
 <1.
\]

Therefore, equality occurs
precisely when
$
 \mathcal C=\mathcal E=2^R
$
for some $R\in\binom{[6]}{3}$.
\end{proof}

We apply Lemma~\ref{lem:six} to the pair
$(\mathcal A,\mathcal B)$ from
Subsection~\ref{sec:boundary-tools} in the remaining case $t=3$ and
$s=6$. Let
\[
 \mathcal A_0=\mathcal A\cap2^{[6]},
 \qquad
 \mathcal B_0=\mathcal B\cap2^{[6]}.
\]
Then
$
 \mathcal A=\mathcal D_6(\mathcal A_0)
$
and
$
 \mathcal B=\mathcal D_6(\mathcal B_0).
$
Therefore,
\[
 \mu_{p_1}(\mathcal A)=\mu_{p_1}(\mathcal A_0),
 \qquad
 \mu_{p_2}(\mathcal B)=\mu_{p_2}(\mathcal B_0).
\]
The families $\mathcal A_0$ and $\mathcal B_0$ are increasing and
cross $3$-intersecting. Define their complement families by
\[
 \mathcal C=\{[6]\setminus A:A\in\mathcal A_0\},
 \qquad
 \mathcal E=\{[6]\setminus B:B\in\mathcal B_0\}.
\]
Then $\mathcal C$ and $\mathcal E$ are decreasing. Moreover, for
$C=[6]\setminus A$ and $E=[6]\setminus B$,
$
 |C\cup E|
 =
 6-|A\cap B|
 \leq3.
$

Put
$
 u=(1-p_1)/p_1
$
and
$
 v=(1-p_2)/p_2.
$
Since $p_1,p_2<1/4$, we have $u,v>3$. For
$A=[6]\setminus C$,
\[
 p_1^{|A|}(1-p_1)^{6-|A|}
 =
 p_1^{6-|C|}(1-p_1)^{|C|}
 =
 p_1^6u^{|C|},
\]
and the analogous identity holds for $\mathcal B_0$. Summing these
identities and using
$
 p_1=(1+u)^{-1}
$
and
$
 p_2=(1+v)^{-1},
$
we obtain
\[
 \frac{\mu_{p_1}(\mathcal A)}{p_1^3}
 =
 \frac{W_u(\mathcal C)}{(1+u)^3},
 \qquad
 \frac{\mu_{p_2}(\mathcal B)}{p_2^3}
 =
 \frac{W_v(\mathcal E)}{(1+v)^3}.
\]
Lemma~\ref{lem:six} now gives
\[
 \mu_{p_1}(\mathcal A)\mu_{p_2}(\mathcal B)
 \leq(p_1p_2)^3.
\]
If equality holds, then
$
 \mathcal C=\mathcal E=2^R
$
for some $R\in\binom{[6]}{3}$. Taking complements and extending
outside $[6]$ gives
\[
 \mathcal A=\mathcal B
 =
 \{F\subseteq[n]:[6]\setminus R\subseteq F\}.
\]
Thus, $\mathcal A$ and $\mathcal B$ form the common $3$-star with
center $[6]\setminus R$.

\subsection{Completion of the proof for
\texorpdfstring{$t\geq3$}{t>=3}}

We now combine the preceding lemmas to prove the product inequality
and characterize the equality case for $t\geq3$.

\begin{proof}[Proof of Theorem~\ref{thm:product-measure} for
$t\geq3$ in the open range]
Fix $0<p_1,p_2<1/(t+1)$, and let
$(\mathcal A,\mathcal B)$ be a shifted product-maximizing pair as in
Subsection~\ref{sec:boundary-tools}. Product maximality gives
$
 \mu_{p_1}(\mathcal A)\mu_{p_2}(\mathcal B)\geq(p_1p_2)^t,
$
as witnessed by a common $t$-star.

Every generating set has size at least $t$.  If $s=t$, then
$
 G(\mathcal A)=G(\mathcal B)=\{[t]\},
$
and hence
$
 \mathcal A=\mathcal B=\mathcal S_{[t]}.
$
Assume that $s>t$.
Choose
$
 i=\min\{r:\mathcal X_r\cup\mathcal Y_r\neq\emptyset\}.
$
After interchanging the two families if necessary, assume that
$\mathcal X_i\neq\emptyset$, and put $j=s+t-i$.
Lemma~\ref{lem:boundary} gives $\mathcal Y_j\neq\emptyset$, so
$i\leq j$. If $i=t$, Lemma~\ref{lem:size-t} gives a product strictly
smaller than $(p_1p_2)^t$, a contradiction. Therefore,
\[
 t+1\leq i\leq j,
 \qquad
 s+t=i+j\geq2t+2,
\]
and hence $s\geq t+2$.

If $s=t+2$, the application of Lemma~\ref{lem:excess-two} above gives
$
 \mu_{p_1}(\mathcal A)\mu_{p_2}(\mathcal B)\leq(p_1p_2)^t.
$
Product maximality forces equality, so $\mathcal A$ and $\mathcal B$
are a common $t$-star. Shiftedness then gives
$
 \mathcal A=\mathcal B=\mathcal S_{[t]},
$
contradicting $s=t+2$.

Thus, $s\geq t+3$. If $(t,s)\neq(3,6)$,
Lemma~\ref{lem:algebra} verifies
\eqref{eq:boundary-criterion}, again a contradiction. In the remaining
case $t=3$ and $s=6$, the application of Lemma~\ref{lem:six} in the
preceding subsection gives a product at most $(p_1p_2)^3$, with
equality only for a common $3$-star. Product maximality forces
equality, and shiftedness then gives $s=3$, contradicting $s=6$.

Consequently, every shifted product-maximizing pair is
$(\mathcal S_{[t]},\mathcal S_{[t]})$, and the required inequality
follows.
Now let $\mathcal F_1,\mathcal F_2\subseteq2^{[n]}$ attain equality.
The pair is product-maximizing, so both families are increasing.
Simultaneous shifting produces a shifted equality pair
$
 (\mathcal F_1',\mathcal F_2')
 =
 (\mathcal S_{[t]},\mathcal S_{[t]}).
$
Reversing the shifts and applying
Lemma \ref{lem:inverse-shift-star} gives
$
 \mathcal F_i=\mathcal S_{T_i}
$
for some $T_i\in\binom{[n]}{t}$, $i=1,2$. Cross
$t$-intersection forces $T_1=T_2$. Conversely, every common
$t$-star attains equality.
\end{proof}

\section{The exact theorem for $t=2$}\label{sec:t2}

\subsection{Reduction to four boundary types}

Assume that $t=2$ and $0<p_1,p_2<1/3$. Let
$(\mathcal A,\mathcal B)$ be a shifted product-maximizing pair, and
retain the notation introduced in
Section~\ref{sec:generating-set}. We postpone the cases $s=2,3,4$ to
Subsection~\ref{sec:t2-completion} and first assume that $s\geq5$. 

Let $i$ be the smallest integer such that
$\mathcal X_i\cup\mathcal Y_i\neq\emptyset$. After interchanging
$(\mathcal A,p_1)$ and $(\mathcal B,p_2)$ if necessary, we may assume
that $\mathcal X_i\neq\emptyset$. Put $j=s+2-i$. By
Lemma~\ref{lem:boundary}, $\mathcal Y_j\neq\emptyset$, and the choice
of $i$ gives $i\leq j$. The case $i=2$ is excluded by
Lemma~\ref{lem:size-t}, since its conclusion would contradict the
product attained by a common $2$-star. Hence,
$3\leq i\leq j\leq s-1$ and $i+j=s+2$.

We first apply the boundary estimate obtained in
Subsection~\ref{sec:boundary-tools}. As
shown in the proof of Lemma~\ref{lem:algebra}, $H_{s,r}(p)$ is
strictly decreasing in $p$ whenever $r<s$. At $p=1/3$, we have
$H_{s,r}(1/3)=2(3s-r)/(3(s+r)).$
Then
\[
 H_{s,i}(1/3)H_{s,j}(1/3)
 =
 \frac{4(3s-i)(3s-j)}
      {9(s+i)(s+j)}.
\]
Since the denominator is positive, this product is at least one if
and only if $\Delta\geq0$, where
\[
 \Delta
 =
 4(3s-i)(3s-j)-9(s+i)(s+j)=
 6s(s-7)-5ij.
\]
Moreover,
$
 ij\leq\frac{(i+j)^2}{4}
      =\frac{(s+2)^2}{4}.
$
It follows that
\[
 4\Delta=24s(s-7)-20ij\geq24s(s-7)-5(s+2)^2=(s-10)(19s+2).
\]
Consequently, $\Delta\geq0$ for every $s\geq10$, with equality only
when $(s,i,j)=(10,6,6)$. If $\Delta>0$, then
$H_{s,i}(1/3)H_{s,j}(1/3)>1$. When $(s,i,j)=(10,6,6)$, we have
$H_{10,6}(1/3)^2=1$. Since $H_{10,6}(p)$ is strictly decreasing and
$p_1,p_2<1/3$, it follows that
$H_{10,6}(p_1)H_{10,6}(p_2)>1$. Therefore,
\eqref{eq:boundary-criterion} excludes every case with $s\geq10$.

It remains to consider $5\leq s\leq9$. We refine the preceding
estimate by using all upper levels of $\mathcal X_i^-$ and
$\mathcal Y_j^-$.

For $i\leq\ell\leq s-1$, let
$\nabla_\ell\mathcal X_i^-$ be the family of all $\ell$-subsets of
$[s-1]$ containing a member of $\mathcal X_i^-$. Counting pairs
$(X,R)$ with $X\in\mathcal X_i^-$,
$R\in\nabla_\ell\mathcal X_i^-$, and $X\subseteq R$ gives
\[
 \frac{|\nabla_\ell\mathcal X_i^-|}
      {\binom{s-1}{\ell}}
 \geq
 \frac{|\mathcal X_i^-|}
      {\binom{s-1}{i-1}}.
\]
Let $R\in\nabla_\ell\mathcal X_i^-$ and choose
$X\in\mathcal X_i^-$ with $X\subseteq R$. Since $\ell\geq i$, there
is some $a\in R\setminus X$. Now $X\cup\{s\}\in\mathcal X_i$, and
shiftedness gives $X\cup\{a\}\in\mathcal A$. As $\mathcal A$ is
increasing and $X\cup\{a\}\subseteq R$, it follows that
$R\in\mathcal A$ and hence
$\mathcal D_{s-1}(R)\subseteq\mathcal A$.

As $\ell$ and $R$ vary, these families are pairwise disjoint. They
are also disjoint from $\mathcal D_s(\mathcal X_i)$, because their
members contain at least $i$ elements of $[s-1]$, whereas every
member of $\mathcal D_s(\mathcal X_i)$ contains exactly $i-1$
elements of $[s-1]$. For $1\leq r<s$, define
\begin{align*}
 \widehat K_{s,r}(p)
 =
 1+
 \frac{\displaystyle\sum_{\ell=r}^{s-1}
 \binom{s-1}{\ell}p^\ell(1-p)^{s-1-\ell}}
 {\displaystyle\binom{s-1}{r-1}p^r(1-p)^{s-r}},\qquad
 \widehat H_{s,r}(p)=
 \frac{1-p}{p}\,
 \frac{\widehat K_{s,r}(p)}
 {\widehat K_{s,r}(p)+(1-p)/p}.
\end{align*}
Recall that
\[
 q_1=1-p_1,\qquad q_2=1-p_2,\qquad
 \alpha=\mu_{p_1}(\mathcal A),\qquad
 \beta=\mu_{p_2}(\mathcal B),
\]
and 
\begin{align*}
 x
 &=\mu_{p_1}\bigl(\mathcal D_s(\mathcal X_i)\bigr)
   =|\mathcal X_i|p_1^iq_1^{s-i},
&
 x^+
 &=\mu_{p_1}\bigl(\mathcal D_s(\mathcal X_i^-)\bigr)
   =\frac{q_1}{p_1}x,\\
 y
 &=\mu_{p_2}\bigl(\mathcal D_s(\mathcal Y_j)\bigr)
   =|\mathcal Y_j|p_2^jq_2^{s-j},
&
 y^+
 &=\mu_{p_2}\bigl(\mathcal D_s(\mathcal Y_j^-)\bigr)
   =\frac{q_2}{p_2}y.
\end{align*}
For $R\in\nabla_\ell\mathcal X_i^-$, the family
$\mathcal D_{s-1}(R)$ has $\mu_{p_1}$-measure
$p_1^\ell q_1^{s-1-\ell}$. Since these families are pairwise
disjoint and are disjoint from $\mathcal D_s(\mathcal X_i)$, we have
\[
 \alpha
 \geq
 x+\sum_{\ell=i}^{s-1}
 |\nabla_\ell\mathcal X_i^-|
 p_1^\ell q_1^{s-1-\ell}\geq
 x+
 \frac{|\mathcal X_i^-|}{\binom{s-1}{i-1}}
 \sum_{\ell=i}^{s-1}
 \binom{s-1}{\ell}
 p_1^\ell q_1^{s-1-\ell}=x\widehat K_{s,i}(p_1),
\]
where the last equality follows from
$|\mathcal X_i^-|=|\mathcal X_i|$ and
$x=|\mathcal X_i|p_1^iq_1^{s-i}$. By symmetry, $\beta\geq y\widehat K_{s,j}(p_2)$. Since
$x^+=q_1x/p_1$ and $y^+=q_2y/p_2$, inequality
\eqref{eq:boundary-upper} now implies that every product-maximizing
pair must satisfy
\[\widehat H_{s,i}(p_1)\widehat H_{s,j}(p_2)\leq1\].

We next show that $\widehat H_{s,r}(p)$ is strictly decreasing in
$p$. Fix $1\leq r<s$, write $q=1-p$, and put
\[
 m=s-1,\qquad k=r-1,\qquad z=\frac{p}{q},
 \qquad d=m-k+1.
\]
Define
\[
 U(z)=\sum_{a=0}^{m-k}\binom{m}{k+a}z^a,
 \qquad
 W(z)=(1+z)U(z)=\sum_{a=0}^{d}w_az^a.
\]
We first rewrite the previously defined functions
$\widehat K_{s,r}$ and $\widehat H_{s,r}$ in terms of $z$. Since
$r=k+1$ and $s-1=m$, we have
\[
 \sum_{\ell=r}^{s-1}
 \binom{s-1}{\ell}p^\ell q^{s-1-\ell}
 =
 p^kq^{m-k}\bigl(U(z)-w_0\bigr),\qquad \binom{s-1}{r-1}p^rq^{s-r}
 =
 w_0p^{k+1}q^{m-k},
\]
where $w_0=\binom{m}{k}$. Therefore,
\[
 \widehat K_{s,r}(p)
 =
 1+\frac{U(z)-w_0}{w_0p}=
 \frac{(1+z)U(z)-w_0}{w_0z}
 =
 \frac{W(z)-w_0}{w_0z}.
\]
Since $p=z/(1+z)$ and $q/p=1/z$, the definition of
$\widehat H_{s,r}$ now gives
\begin{align*}
 \widehat H_{s,r}\left(\frac{z}{1+z}\right)
 =
 \frac{1}{z}
 \frac{\dfrac{W(z)-w_0}{w_0z}}
      {\dfrac{W(z)-w_0}{w_0z}+\dfrac{1}{z}}=
 \frac{W(z)-w_0}{zW(z)}.
\end{align*}
By Pascal's identity, the coefficients of $W$ satisfy
$
 w_0=\binom{m}{k}
$
and
$
 w_a=\binom{m+1}{k+a}
$
for
$
1\leq a\leq d$.
 Put
\[
 \rho_a=\frac{w_{a+1}}{w_a}\quad(0\leq a<d),
 \qquad
 \rho_d=0.
\]
Then
\[
 \rho_0=\frac{m+1}{k+1},
 \qquad
 \rho_a=\frac{m+1-k-a}{k+a+1}
 \quad(1\leq a<d).
\]
These numbers are strictly decreasing. Indeed,
\[
 \rho_0-\rho_1
 =
 \frac{m+k^2+2k+2}{(k+1)(k+2)}>0,\quad
 \rho_a-\rho_{a+1}
 =
 \frac{m+2}{(k+a+1)(k+a+2)}>0 \quad(1\leq a<d).
\]
Set
\[
 \Phi(z)
 =
 \widehat H_{s,r}\left(\frac{z}{1+z}\right).
\]
Observe that
\[
 \frac{W(z)-w_0}{z}
 =
 \sum_{a=1}^{d}w_az^{a-1}
 =
 \sum_{a=0}^{d-1}w_{a+1}z^a=
 \sum_{a=0}^{d}\rho_aw_az^a,
\]
where the last equality follows from
$\rho_a=w_{a+1}/w_a$ for $0\leq a<d$ and $\rho_d=0$.
We may write
\[
 \Phi(z)
 =
 \frac{\displaystyle\sum_{a=0}^{d}\rho_aw_az^a}
      {\displaystyle\sum_{a=0}^{d}w_az^a}.
\]
Differentiating the preceding expression for $\Phi$ gives
\begin{align*}
 z\Phi'(z)
 &=
 \frac{
 \left(\sum_{a=0}^{d}a\rho_aw_az^a\right)W(z)
 -
 \left(\sum_{a=0}^{d}\rho_aw_az^a\right)
 \left(\sum_{b=0}^{d}bw_bz^b\right)}
 {W(z)^2}\\
 &=
 \frac{1}{W(z)^2}
 \sum_{a,b=0}^{d}
 (a-b)\rho_aw_aw_bz^{a+b}\\
 &=
 \frac{1}{2W(z)^2}
 \sum_{a,b=0}^{d}
 (a-b)(\rho_a-\rho_b)w_aw_bz^{a+b}.
\end{align*}
The last equality follows by interchanging $a$ and $b$ in the double
sum and then averaging the two expressions.
For $a\neq b$, the factors $a-b$ and $\rho_a-\rho_b$ have opposite
signs because $\rho_0>\rho_1>\cdots>\rho_d$. Hence,
$\Phi'(z)<0$. Finally, $z=p/(1-p)$ is strictly increasing in $p$, so
$\widehat H_{s,r}(p)$ is strictly decreasing in $p$.

Recall that every product-maximizing pair with boundary sizes $i$ and
$j$ must satisfy
\[
 \widehat H_{s,i}(p_1)\widehat H_{s,j}(p_2)\leq1.
\]
Since $\widehat H_{s,r}(p)$ is strictly decreasing in $p$ and
$p_1,p_2<1/3$, we have
\[
 \widehat H_{s,i}(p_1)\widehat H_{s,j}(p_2)
 >
 \widehat H_{s,i}(1/3)\widehat H_{s,j}(1/3).
\]
Consequently, a boundary type $(s,i,j)$ is impossible whenever
$\widehat H_{s,i}(1/3)\widehat H_{s,j}(1/3)>1$.

For $(s,i,j)=(7,3,6)$, direct calculation gives
$\widehat H_{7,3}(1/3)\widehat H_{7,6}(1/3)=6260/6149>1$.
For $s=8$, the values of
$\widehat H_{8,i}(1/3)\widehat H_{8,j}(1/3)$ corresponding to
$(i,j)=(3,7),(4,6),(5,5)$ are, respectively, $79084/72495$,
$291884/278883$, and $1331716/1292769$. For $s=9$, the corresponding
values for $(i,j)=(3,8),(4,7),(5,6)$ are $931684/807993$,
$4487716/4050729$, and $9522340/8812521$. Each of these six values is
also greater than one. Hence, all seven boundary types are excluded.
The only possibilities that remain are
\begin{equation}
 (s,i,j)=(5,3,4),\ (6,3,5),\ (6,4,4),\ (7,4,5).
 \label{eq:t2-remaining}
\end{equation}
The preceding estimate does not exclude these four types. We treat them separately in the next two
subsections.

\subsection{Kruskal--Katona reduction}
\label{sec:t2-kk}

Fix one of the four types in \eqref{eq:t2-remaining}, and put
\[
 N=s-1,\qquad
 \mathcal X=\mathcal X_i^-\subseteq\binom{[N]}{i-1},
 \qquad
 \mathcal Y=\mathcal Y_j^-\subseteq\binom{[N]}{j-1}.
\]
Write
\[
 u=|\mathcal X|,
 \qquad
 v=|\mathcal Y|.
\]
For a nonempty family
$\mathcal H\subseteq\binom{[N]}r$ and
$0\leq a\leq N-r$, let $\nabla_{r+a}\mathcal H$ be the family of
all $(r+a)$-subsets of $[N]$ that contain a member of $\mathcal H$,
and put
\[
 u_a(\mathcal H)=|\nabla_{r+a}\mathcal H|,
 \qquad
 P_{\mathcal H}(z)
 =
 \sum_{a=0}^{N-r}u_a(\mathcal H)z^a.
\]
Thus, $u_0(\mathcal H)=|\mathcal H|$. For $z>0$, define
\begin{equation}
 W_{\mathcal H}(z)
 =
 \frac{(1+z)P_{\mathcal H}(z)}{|\mathcal H|},
 \qquad
 L_{\mathcal H}(z)
 =
 \frac{W_{\mathcal H}(z)-1}{zW_{\mathcal H}(z)}.
 \label{eq:upper-shadow-functions}
\end{equation}

We first derive the inequality satisfied by
$L_{\mathcal X}$ and $L_{\mathcal Y}$.

\begin{lemma}\label{lem:upper-shadow-factor}
Let
\[
 z_1=\frac{p_1}{1-p_1},
 \qquad
 z_2=\frac{p_2}{1-p_2}.
\]
Let $(\mathcal A,\mathcal B)$ be a product-maximizing pair of type
$(s,i,j)$, and let
$\mathcal X=\mathcal X_i^-$ and $\mathcal Y=\mathcal Y_j^-$
be the corresponding boundary families. Then
\begin{equation}
 L_{\mathcal X}(z_1)L_{\mathcal Y}(z_2)\leq 1.
 \label{eq:upper-shadow-factor}
\end{equation}
\end{lemma}

\begin{proof}
We first consider $\mathcal X$. Write
$p=p_1$, $q=1-p$, $z=p/q$, and $k=i-1$. Then  $p=z/(1+z)$. By the definition of $x$
in Subsection~\ref{sec:boundary-tools},
\[
 x=up^{k+1}q^{N-k}.
\]
Let $R\in\nabla_{k+a}\mathcal X$ for some $a\geq1$. Choose
$X\in\mathcal X$ with $X\subseteq R$ and
$b\in R\setminus X$. Since
$X\cup\{s\}\in\mathcal X_i\subseteq\mathcal A$ and $b<s$,
shiftedness gives $X\cup\{b\}\in\mathcal A$. Increasingness then
implies that $R\in\mathcal A$, and hence
$\mathcal D_N(R)\subseteq\mathcal A$.

As $a$ and $R$ vary, the families $\mathcal D_N(R)$ are pairwise
disjoint. They are also disjoint from
$\mathcal D_s(\mathcal X_i)$ because every set $R$ considered above has at
least $i$ elements in $[N]$, whereas every member of
$\mathcal D_s(\mathcal X_i)$ has exactly $i-1$ elements in $[N]$.
Consequently,
\begin{align*}
 \alpha
 \geq
 x+\sum_{a=1}^{N-k}
 u_a(\mathcal X)p^{k+a}q^{N-k-a},\qquad \frac{\alpha}{x}
 \geq
 1+\frac{P_{\mathcal X}(z)/u-1}{p}
 =
 \frac{W_{\mathcal X}(z)-1}{z}.
\end{align*}
Since $x^+=(q/p)x=x/z$,  we have
\[
 \frac{q}{p}\frac{\alpha}{\alpha+x^{+}}
 =
 \frac{\alpha/x}{z(\alpha/x)+1}.
\]
The right-hand side is increasing in $\alpha/x$. Therefore, using
$
 \frac{\alpha}{x}
 \geq
 \frac{W_{\mathcal X}(z)-1}{z},
$
we obtain
\[
 \frac{q}{p}\frac{\alpha}{\alpha+x^{+}}
 \geq
 \frac{(W_{\mathcal X}(z)-1)/z}
      {W_{\mathcal X}(z)}
 =
 L_{\mathcal X}(z).
\]
The same argument applied to $\mathcal Y$ gives
\[
 \frac{q_2}{p_2}\frac{\beta}{\beta+y^+}
 \geq
 L_{\mathcal Y}(z_2).
\]
Combining these inequalities with
\eqref{eq:boundary-upper}, we obtain
\[
 L_{\mathcal X}(z_1)L_{\mathcal Y}(z_2)
 \leq
 \frac{q_1q_2}{p_1p_2}
 \frac{\alpha}{\alpha+x^+}
 \frac{\beta}{\beta+y^+}\leq1,
\]
as desired.
\end{proof}

We first recall the upper shadow form of the Kruskal--Katona
theorem \cite{Katona,Kruskal}. For distinct
$A,B\in\binom{[N]}{k}$, write $A<_{\mathrm{lex}}B$ if
$\min(A\triangle B)\in A$, and let $\mathcal L_{N,k}(m)$ denote
the first $m$ members of $\binom{[N]}{k}$ in this order. If
$\mathcal H\subseteq\binom{[N]}{k}$ and $m=|\mathcal H|$, then
\[
 |\nabla_{k+a}\mathcal H|
 \geq
 |\nabla_{k+a}\mathcal L_{N,k}(m)|
 \qquad
 (0\leq a\leq N-k).
\]
This formulation follows from the usual lower shadow version of the
Kruskal--Katona theorem by taking complements and reversing the order of the ground set. We now use this inequality to derive a lower bound
for $L_{\mathcal H}$.
Write
\[
 W_{N,r,m}(z)=W_{\mathcal L_{N,r}(m)}(z),
 \qquad
 L_{N,r,m}(z)=L_{\mathcal L_{N,r}(m)}(z).
\]

\begin{lemma}\label{lem:t2-kk-lex}
Let $\mathcal H\subseteq\binom{[N]}r$ be nonempty and put
$m=|\mathcal H|$. Then
\begin{equation}
 L_{\mathcal H}(z)\geq L_{N,r,m}(z)
 \qquad (z>0).
 \label{eq:lex-comparison}
\end{equation}
Moreover, $L_{N,r,m}(z)$ is strictly decreasing in $z$ for
\[
 (N,r)=(4,2),(4,3),(5,2),(5,3),(5,4),(6,3),(6,4).
\]
\end{lemma}

\begin{proof}
List the members of $\binom{[N]}r$ in lexicographic order as
$A_1,A_2,\ldots, A_{\binom{N}r}$. The sets containing $A_h$ but no earlier $A_\ell$ are precisely
\[
 A_h\cup T,
 \qquad
 T\subseteq\{\max A_h+1,\ldots,N\}.
\]
Indeed, suppose that $A_h\subseteq R$ and that
$R\setminus A_h$ contains an element $b<\max A_h$. Choose
$c\in A_h$ with $c>b$, and let
\[
 B=(A_h\setminus\{c\})\cup\{b\}.
\]
Then $B\subseteq R$ and
$\min(B\triangle A_h)=b\in B$, so
$B<_{\mathrm{lex}}A_h$. Hence, $R$ contains a $r$-set preceding
$A_h$.
Conversely, suppose that $R=A_h\cup T$, where every element of $T$
is larger than $\max A_h$. If $B\subseteq R$ is a $r$-set distinct
from $A_h$, then
\[
 \min(A_h\triangle B)\in A_h,
\]
and hence $A_h<_{\mathrm{lex}}B$. Thus, $R$ contains no $r$-set
preceding $A_h$.

For each $1\leq h\leq m$, let
\[
 \mathcal C_h
 =
 \{R\subseteq[N]:
   A_h\subseteq R
   \text{ and }
   A_\ell\nsubseteq R\text{ for every }\ell<h\}.
\]
By the preceding argument,
\[
 \mathcal C_h
 =
 \{A_h\cup T:
   T\subseteq\{\max A_h+1,\ldots,N\}\}.
\]
Now let $R\subseteq[N]$ contain at least one of
$A_1,\ldots,A_m$. There is a unique smallest index $h$ such that
$A_h\subseteq R$, and the definition above gives $R\in\mathcal C_h$.
Hence, the collections $\mathcal C_1,\ldots,\mathcal C_m$ are
pairwise disjoint and together contain exactly all sets containing
at least one member of $\mathcal L_{N,r}(m)$.
For a fixed $h$, each $R\in\mathcal C_h$ has the unique form
$R=A_h\cup T$, and its contribution to
$P_{\mathcal L_{N,r}(m)}(z)$ is $z^{|T|}$. Therefore,
\[
 \sum_{R\in\mathcal C_h}z^{|R|-r}
 =
 \sum_{T\subseteq\{\max A_h+1,\ldots,N\}}z^{|T|}
 =
 (1+z)^{N-\max A_h}.
\]
Summing over $h$ gives
\begin{align}\label{eq:lex-decomposition}
 P_{\mathcal L_{N,r}(m)}(z)
=
 \sum_{h=1}^{m}(1+z)^{N-\max A_h},\qquad
 W_{N,r,m}(z)
 =
 \frac{1}{m}\sum_{h=1}^{m}
 (1+z)^{N-\max A_h+1}.
\end{align}

\textbf{By the Kruskal--Katona theorem,} for every
$0\leq a\leq N-r$,
\[
 u_a(\mathcal H)
 \geq
 u_a(\mathcal L_{N,r}(m)).
\]
Since $z>0$, summing these inequalities after multiplication by
$z^a$ gives
\[
 P_{\mathcal H}(z)
 =
 \sum_{a=0}^{N-r}u_a(\mathcal H)z^a\geq
 \sum_{a=0}^{N-r}
 u_a(\mathcal L_{N,r}(m))z^a
 =
 P_{\mathcal L_{N,r}(m)}(z).
\]
Both families have size $m$. Therefore,
\[
 W_{\mathcal H}(z)
 =
 \frac{(1+z)P_{\mathcal H}(z)}{m}
 \geq
 \frac{(1+z)P_{\mathcal L_{N,r}(m)}(z)}{m}
 =
 W_{N,r,m}(z).
\]
Consequently,
\[
 L_{\mathcal H}(z)-L_{N,r,m}(z)
 =
 \frac{1}{z}
 \left(
 \frac{1}{W_{N,r,m}(z)}
 -
 \frac{1}{W_{\mathcal H}(z)}
 \right)
 \geq0,
\]
which proves \eqref{eq:lex-comparison}.

\textbf{It remains to prove the asserted monotonicity.} Put
\[
d=N-r+1, \qquad  e_h=N-\max A_h+1.
\]
By \eqref{eq:lex-decomposition},
\begin{align}\label{adf1}
 mW_{N,r,m}(z)
 =
 \sum_{h=1}^{m}(1+z)^{e_h}.
\end{align}
Since $1\leq e_h\leq d$, expanding each term by the binomial theorem
gives
\[
 mW_{N,r,m}(z)
 =
 \sum_{a=0}^{d}c_az^a,
 \qquad
 c_a=
 \sum_{1\leq h\leq m, e_h\geq a}
 \binom{e_h}{a}.
\]
The coefficients $c_a$ depend on $d$, $r$, and $m$. Since these
parameters are fixed in each of the following cases, we suppress
this dependence from the notation.
In particular, $c_0=m$.
Moreover, $e_1=d$ and $e_h<d$ for $h>1$, so $c_d=1$.
Using the definition of $L_{N,r,m}(z)$, we obtain
\[
 L_{N,r,m}(z)
 =
 \frac{mW_{N,r,m}(z)-m}
      {zmW_{N,r,m}(z)}=
 \frac{\displaystyle\sum_{a=1}^{d}c_az^a}
      {\displaystyle z\sum_{a=0}^{d}c_az^a}=
 \frac{\displaystyle\sum_{a=0}^{d-1}c_{a+1}z^a}
      {\displaystyle\sum_{a=0}^{d}c_az^a}.
\]

To determine the possible coefficient sequences, extend
$A_1,\ldots,A_m$ to the lexicographic ordering of all
$r$-subsets of $[d+r-1]$. For such a set $A$, define the exponent 
\[
 e(A)=d+r-\max A.
\]
Thus, $e(A_h)=e_h$ for $1\leq h\leq m$. Let
$\omega_{d,r}$ denote the ordered sequence of the values $e(A)$
as $A$ runs through all $r$-subsets of $[d+r-1]$ in
lexicographic order. In particular, $e_1,\ldots,e_m$ are the first
$m$ terms of $\omega_{d,r}$.
 The
$r$-sets are arranged according to their smallest elements. All sets
with smallest element $1$ come first, followed by those with smallest
element $2$, and so on, until the sets with smallest element $d$.
Fix $t\in[d]$ and consider the sets whose smallest element is $t$.
For each such set $A$, let
\[
 B=\{x-t:x\in A\setminus\{t\}\}.
\]
Then $B$ is an $(r-1)$-subset of $[d+r-1-t]$, and
\[
 d+r-\max A
 =
 (d-t+1)+(r-1)-\max B.
\]

Every $r$-set $A$ with smallest element $t$ can be written uniquely
as
$
 A=\{t\}\cup\{t+x:x\in B\},
$
where $B\in\binom{[d+r-1-t]}{r-1}$. If $A,A'$ correspond to
$B,B'$, respectively, then
$
 \min(A\triangle A')
 =
 t+\min(B\triangle B').
$
Moreover, this element belongs to $A$ if and only if
$\min(B\triangle B')$ belongs to $B$. Hence,
\[
 A<_{\mathrm{lex}}A'
 \quad\text{if and only if}\quad
 B<_{\mathrm{lex}}B'.
\]
Thus, as the sets $A$ with smallest element $t$ are listed in
lexicographic order, the corresponding sets $B$ run through all
$(r-1)$-subsets of $[d+r-1-t]$ in the same order. Since
$
 d+r-\max A
 =
 (d-t+1)+(r-1)-\max B,
$
the resulting sequence of exponents is exactly
$\omega_{d-t+1,r-1}$.
It follows that $\omega_{d,r}$ is obtained by first writing the terms
of $\omega_{d,r-1}$, then those of $\omega_{d-1,r-1}$, and continuing
in this way until $\omega_{1,r-1}$. When $r=1$, the ground set is $[d]$, and its $1$-subsets appear in
lexicographic order as
$
 \{1\},\{2\},\ldots,\{d\}.
$
For the set $\{t\}$, the corresponding exponent is
$
 e(\{t\})=d+1-t.
$
Thus,
\[
 \omega_{d,1}=(d,d-1,\ldots,1).
\]

\textbf{Suppose that $d=2$.} The first $r$-set in lexicographic order is
$[r]$, whose exponent is $2$.
Every other $r$-subset of $[r+1]$ contains $r+1$ and hence has
exponent $1$. Thus, the first $m$ terms of $\omega_{2,r}$ consist of
one $2$ and $m-1$ copies of $1$. 
It follows from \eqref{adf1} that
\[
 mW_{N,r,m}(z)
 =(1+z)^2+(m-1)(1+z)=m+(m+1)z+z^2.
\]
Hence,
$
 (c_0,c_1,c_2)=(m,m+1,1).
$
Consequently
\[
 c_1^2-c_0c_2
 =(m+1)^2-m
 =m^2+m+1>0.
\]
Therefore,
\[
 \frac{c_1}{c_0}>
 \frac{c_2}{c_1}.
\]

\textbf{Suppose that $d=3$.}
For $r=1$, we have
$
 \omega_{3,1}=(3,2,1).
$
For $r\geq2$, the relation established above gives
\[
 \omega_{3,r}
 =
 \bigl(
 \omega_{3,r-1},
 \omega_{2,r-1},
 \omega_{1,r-1}
 \bigr),
\]
where the right-hand side means that the terms of the three sequences
are listed one after another in the indicated order.
By the case $d=2$, the sequence $\omega_{2,r-1}$ consists of one
$2$ followed by $r-1$ terms equal to $1$, while
$\omega_{1,r-1}=(1)$. Hence,
\[
 \omega_{3,r}
 =
 \bigl(
 \omega_{3,r-1},(2,1,\ldots,1)
 \bigr),
\]
where the final parenthesized group consists of one $2$ followed by
$r$ terms equal to $1$. Starting with
$\omega_{3,1}=(3,(2,1))$ and applying this relation successively
gives
\[
 \omega_{3,r}
 =
 \bigl(
 3,\,(2,1),\,(2,1,1),\ldots,(2,1,\ldots,1)
 \bigr),
\]
where, for each $1\leq j\leq r$, the $j$th parenthesized group
consists of one $2$ followed by $j$ terms equal to $1$. The inner
parentheses are used only to separate these consecutive groups.
Let
\[
 a=|\{h\in[m]:e_h=1\}|,
 \qquad
 b=|\{h\in[m]:e_h=2\}|.
\]
 If $b=0$,
then $a=0$. If $b\geq1$, then
\[
 \binom{b}{2}\leq a\leq\binom{b+1}{2}.
\]
Since the first $m$ exponents  consist of one $3$, together with
$b$ terms equal to $2$ and $a$ terms equal to $1$, we obtain
\[
\begin{aligned}
 mW_{N,r,m}(z)
 &=
 (1+z)^3+b(1+z)^2+a(1+z)\\
 &=
 (a+b+1)+(a+2b+3)z+(b+3)z^2+z^3.
\end{aligned}
\]
Thus,
\[
 (c_0,c_1,c_2,c_3)
 =
 (a+b+1,a+2b+3,b+3,1).
\]
A direct calculation gives
\begin{align*}
 c_1^2-c_0c_2
 &=
 a^2+3b^2+3ab+3a+8b+6>0,\\
 c_2^2-c_1c_3
 &=
 b^2+4b+6-a\geq
 b^2+4b+6-\binom{b+1}{2}=
 \frac{b^2+7b+12}{2}>0.
\end{align*}
Therefore,
\[
 \frac{c_1}{c_0}>
 \frac{c_2}{c_1}>
 \frac{c_3}{c_2}.
\]

\textbf{Finally, suppose that $d=4$.} For $r=1$, we have
$
 \omega_{4,1}=(4,3,2,1).
$
For $r\geq2$, the decomposition according to the smallest element
shows that
$
 \omega_{4,r}
 =
 \bigl(\omega_{4,r-1},\omega_{3,r}\bigr),
$
where the terms of the two sequences on the right are listed one
after another. Applying this relation successively gives
\[
 \omega_{4,r}
 =
 \bigl(4,\omega_{3,1},\omega_{3,2},\ldots,\omega_{3,r}\bigr).
\]
Let $a$, $b$, and $g$ be the numbers of terms equal to $1$, $2$,
and $3$, respectively, among $e_1,\ldots,e_m$. The term $4$ occurs
once and is the first term. If $g=0$, then $m=1$ and $a=b=0$.
Suppose that $g\geq1$. Since each $\omega_{3,j}$ begins with its
unique term equal to $3$, the first $m$ terms contain all terms of
\[
 \omega_{3,1},\omega_{3,2},\ldots,\omega_{3,g-1}
\]
and an initial segment of $\omega_{3,g}$.
By the case $d=3$, the full sequence $\omega_{3,j}$ contains one
term equal to $3$, $j$ terms equal to $2$, and
$\binom{j+1}{2}$ terms equal to $1$. Let $u$ and $v$ be the numbers
of terms equal to $2$ and $1$, respectively, in the part taken from
$\omega_{3,g}$. Again by the case $d=3$,
\[
 0\leq u\leq g,
 \qquad
 \binom{u}{2}\leq v\leq\binom{u+1}{2}.
\]
Therefore,
\[
\begin{aligned}
 b
 &=
 \sum_{j=1}^{g-1}j+u
 =
 \binom{g}{2}+u,\\
 a
 &=
 \sum_{j=1}^{g-1}\binom{j+1}{2}+v
 =
 \binom{g+1}{3}+v.
\end{aligned}
\]
These formulas also hold when $g=0$ by taking $u=v=0$.
The first $m$ exponents consist of one $4$, together with
$g$ terms equal to $3$, $b$ terms equal to $2$, and $a$ terms equal
to $1$. Hence,
\[
\begin{aligned}
 mW_{N,r,m}(z)
 =&
 (1+z)^4+g(1+z)^3+b(1+z)^2+a(1+z)\\
 =&
 (a+b+g+1)
 +(a+2b+3g+4)z\\
 &
 +(b+3g+6)z^2+(g+4)z^3+z^4.
\end{aligned}
\]
Thus,
\[
 (c_0,c_1,c_2,c_3,c_4)
 =
 (a+b+g+1,a+2b+3g+4,b+3g+6,g+4,1).
\]
A direct calculation gives
\[
 c_1^2-c_0c_2
 =
 a^2+3b^2+6g^2+3ab+3ag+8bg
 +2a+9b+15g+10>0.
\]
Moreover, since
\[
 b=\binom{g}{2}+u
 \leq
 \binom{g+1}{2},
\]
we have
\[
\begin{aligned}
 c_3^2-c_2c_4
 &=
 g^2+5g+10-b\\
 &\geq
 g^2+5g+10-\binom{g+1}{2}\\
 &=
 \frac{g^2+9g+20}{2}>0.
\end{aligned}
\]
It remains to prove that $c_2^2>c_1c_3$. Put
$
 \Delta=c_2^2-c_1c_3.
$
Then
\[
 \Delta
 =
 b^2+4bg+4b+6g^2+20g+20-a(g+4).
\]
Substituting
\[
 b=\binom{g}{2}+u,
 \qquad
 a=\binom{g+1}{3}+v
\]
and simplifying, we obtain
\[
 12\Delta
 =
 g^4+10g^3+77g^2+224g+240+12u(g^2+3g+4)+12u^2-12(g+4)v.
\]
Since $v\leq\binom{u+1}{2}$,
\[
 12\Delta
 \geq g^4+10g^3+77g^2+224g+240+6u\bigl(2g^2+5g+4-(g+2)u\bigr).
\]
Finally, $0\leq u\leq g$ implies
\[
 2g^2+5g+4-(g+2)u
 \geq
 g^2+3g+4>0.
\]
Therefore, $\Delta>0$, and hence
\[
 \frac{c_1}{c_0}>
 \frac{c_2}{c_1}>
 \frac{c_3}{c_2}>
 \frac{c_4}{c_3}.
\]

\textbf{For every pair $(N,r)$ listed in the statement of the lemma,
$d=N-r+1$ belongs to $\{2,3,4\}$.} The three cases above therefore
prove that
\[
 c_\ell^2>c_{\ell-1}c_{\ell+1}
 \qquad
 (1\leq\ell<d)
\]
for every pair under consideration. Define
\[
 \rho_\ell=\frac{c_{\ell+1}}{c_\ell}
 \quad(0\leq\ell<d),
 \qquad
 \rho_d=0.
\]
Since all the coefficients are positive, the preceding inequalities
give
\[
 \rho_0>\rho_1>\cdots>\rho_{d-1}>\rho_d=0.
\]
Furthermore,
\[
 L_{N,r,m}(z)
 =
 \frac{\displaystyle\sum_{\ell=0}^{d-1}
 c_{\ell+1}z^\ell}
 {\displaystyle\sum_{\ell=0}^{d}c_\ell z^\ell}=
 \frac{\displaystyle\sum_{\ell=0}^{d}
 \rho_\ell c_\ell z^\ell}
 {\displaystyle\sum_{\ell=0}^{d}c_\ell z^\ell}.
\]
It follows that
\begin{align*}
 zL_{N,r,m}'(z)
 &=
 \frac{
 \left(\displaystyle\sum_{\ell=0}^{d}
 \ell\rho_\ell c_\ell z^\ell\right)
 \left(\displaystyle\sum_{k=0}^{d}c_kz^k\right)}
 {\left(\displaystyle\sum_{\ell=0}^{d}c_\ell z^\ell\right)^2}-
 \frac{
 \left(\displaystyle\sum_{\ell=0}^{d}
 \rho_\ell c_\ell z^\ell\right)
 \left(\displaystyle\sum_{k=0}^{d}kc_kz^k\right)}
 {\left(\displaystyle\sum_{\ell=0}^{d}c_\ell z^\ell\right)^2}\\
 &=
 \frac{\displaystyle
 \sum_{\ell,k=0}^{d}
 (\ell-k)\rho_\ell c_\ell c_kz^{\ell+k}}
 {\displaystyle
 \left(\sum_{\ell=0}^{d}c_\ell z^\ell\right)^2}=
 \frac{\displaystyle
 \sum_{0\leq\ell<k\leq d}
 (\ell-k)(\rho_\ell-\rho_k)c_\ell c_kz^{\ell+k}}
 {\displaystyle
 \left(\sum_{\ell=0}^{d}c_\ell z^\ell\right)^2}.
\end{align*}
For $\ell<k$, we have
\[
 \ell-k<0
 \qquad\text{and}\qquad
 \rho_\ell-\rho_k>0.
\]
Since $z>0$ and all $c_\ell$ are positive, every term in the
numerator is negative. Therefore,
\[
 L_{N,r,m}'(z)<0,
\]
as required.
\end{proof}

The descriptions of the exponent sequences in the preceding proof
give
\begin{equation}
\begin{aligned}
 \omega_{2,3}
 &=
 (2,1,1,1),\\
 \omega_{3,4}
 &=
 (3,2,1,2,1,1,2,1,1,1, 2,1,1,1,1),\\
 \omega_{4,3}
 &=
 (4,3,2,1,3,2,1,2,1,1, 3,2,1,2,1,1,2,1,1,1).
\end{aligned}
\label{eq:t2-exponent-sequences}
\end{equation}
The first $4$, $14$, and $19$ terms of these sequences, respectively,
are used in the calculations below.

We next compute the values of $L_{N,r,m}(z)$ at $z=1/2$.  
Recall that $m=|\mathcal H|$ and
$\mathcal H\subseteq\binom{[N]}r$. Since $N=d+r-1$, we have
\[
 1\leq m\leq\binom{N}{r}
 =
 \binom{d+r-1}{r}.
\]
For fixed $d$, the decomposition of the exponent sequences established above shows that $\omega_{d,r}$ is an initial segment of
$\omega_{d,r+1}$. Since
$
 m\leq\binom{d+r-1}{r},
$
the first $m$ exponents in these two sequences are identical.
Let $e_h^{(r)}$ and $e_h^{(r+1)}$ denote the $h$th terms of
$\omega_{d,r}$ and $\omega_{d,r+1}$, respectively. Since
$\omega_{d,r}$ is an initial segment of $\omega_{d,r+1}$, we have
\[
 e_h^{(r)}=e_h^{(r+1)}
 \qquad (1\leq h\leq m).
\]
Equation \eqref{adf1} therefore gives
\[
 W_{d+r-1,r,m}(1/2)
 =
 \frac{1}{m}\sum_{h=1}^{m}
 \left(\frac32\right)^{e_h^{(r)}}=
 \frac{1}{m}\sum_{h=1}^{m}
 \left(\frac32\right)^{e_h^{(r+1)}}=
 W_{d+r,r+1,m}(1/2).
\]
Since
\[
 L_{N,r,m}(z)
 =
 \frac{W_{N,r,m}(z)-1}{zW_{N,r,m}(z)},
\]
the equality of the two values of $W$ at $z=1/2$ implies that
\[
 L_{d+r-1,r,m}(1/2)
 =
 L_{d+r,r+1,m}(1/2)=
 2-\frac{2m}
 {\displaystyle\sum_{h=1}^{m}(3/2)^{e_h}},
\]
where $e_h= e_h^{(r)}=e_h^{(r+1)}$.
For brevity, set
\[
 L_m^{(d)}
 :=
 L_{d+r-1,r,m}(1/2)
 =
 L_{d+r,r+1,m}(1/2).
\]
The preceding equality shows that this definition is independent of
the choice of the integer $r$ whenever
$1\leq m\leq\binom{d+r-1}{r}$, and hence the notation is well defined.

For $1\leq h\leq m$, define
\[
 w_h^{(d)}=3^{e_h}2^{d-e_h},
 \qquad
 S_m^{(d)}=\sum_{h=1}^{m}w_h^{(d)}.
\]
Then
\begin{equation}
 L_m^{(d)}
 =
 2-\frac{2^{d+1}m}{S_m^{(d)}}.
 \label{eq:t2-half-formula}
\end{equation}
Taking the first $4$, $14$, and $19$ terms, respectively, from the
sequences in \eqref{eq:t2-exponent-sequences} and substituting them
into
$
 w_h^{(d)}=3^{e_h}2^{d-e_h},
$
we obtain
\[
\begin{aligned}
 \bigl(w_1^{(2)},\ldots,w_4^{(2)}\bigr)
 &=
 (9,6,6,6),\\
 \bigl(w_1^{(3)},\ldots,w_{14}^{(3)}\bigr)
 &=
 (27,18,12,18,12,12,18,12,12,12, 18,12,12,12),\\
 \bigl(w_1^{(4)},\ldots,w_{19}^{(4)}\bigr)
 &=
 (81,54,36,24,54,36,24,36,24,24,54,36,24,36,24,24,36,24,24).
\end{aligned}
\]
Taking partial sums and applying \eqref{eq:t2-half-formula}, we obtain
\begin{equation}
\begin{array}{r@{\quad}l@{\quad}l@{\quad}l}
 d=2:
 &
 \displaystyle\min_{1\leq m\leq2}L_m^{(2)}=\frac{14}{15}
 &
 \displaystyle\min_{1\leq m\leq3}L_m^{(2)}=\frac67
 &
 \displaystyle\min_{1\leq m\leq4}L_m^{(2)}=\frac{22}{27};
 \\[2mm]
 d=3:
 &
 \displaystyle\min_{1\leq m\leq2}L_m^{(3)}=\frac{58}{45}
 &
 \displaystyle\min_{1\leq m\leq5}L_m^{(3)}=\frac{94}{87}
 &
 \displaystyle\min_{1\leq m\leq7}L_m^{(3)}=\frac{34}{33},
 \\
 &
 \displaystyle\min_{1\leq m\leq8}L_m^{(3)}=\frac{130}{129}
 &
 \displaystyle\min_{1\leq m\leq9}L_m^{(3)}=\frac{46}{47}
 &
 \displaystyle\min_{1\leq m\leq14}L_m^{(3)}=\frac{190}{207};
 \\[2mm]
 d=4:
 &
 \displaystyle\min_{1\leq m\leq8}L_m^{(4)}=\frac{434}{345}
 &
 \displaystyle\min_{1\leq m\leq9}L_m^{(4)}=\frac{50}{41}
 &
 \displaystyle\min_{1\leq m\leq19}L_m^{(4)}=\frac{742}{675}.
\end{array}
\label{eq:t2-lex-minima}
\end{equation}

Recall that, for the fixed type $(s,i,j)$, we have
\[
 N=s-1,\qquad
 \mathcal X=\mathcal X_i^-
 \subseteq\binom{[N]}{i-1},\qquad
 \mathcal Y=\mathcal Y_j^-
 \subseteq\binom{[N]}{j-1},\qquad u=|\mathcal X|,
 \qquad
 v=|\mathcal Y|.
\]
We call $\mathcal X$ \textit{full} if
$\mathcal X=\binom{[N]}{i-1}$, and \textit{proper} otherwise. We use the same
terminology for $\mathcal Y$.

\begin{lemma}
\label{lem:t2-proper-reduction}
Suppose that both $\mathcal X$ and $\mathcal Y$ are proper. Then the
type $(s,i,j)=(7,4,5)$ cannot occur, and the parameters of the fixed
product-maximizing pair satisfy one of the following conditions:
\begin{equation}
\begin{aligned}
 &(s,i,j)=(5,3,4),
 &&3\leq u\leq5,\quad v=3;\\
 &(s,i,j)=(6,3,5),
 &&(u,v)=(9,4);\\
 &(s,i,j)=(6,4,4),
 &&(u,v)\in\{(8,9),(9,8),(9,9)\}.
\end{aligned}
\label{eq:t2-dense-cases}
\end{equation}
\end{lemma}

\begin{proof}
 Let $z_h=\frac{p_h}{1-p_h}$.
Since $p_1,p_2<1/3$, we have $z_1,z_2<1/2$. For each type in \eqref{eq:t2-remaining}, both
$(N,i-1)$ and $(N,j-1)$ are among the parameter pairs covered by
Lemma~\ref{lem:t2-kk-lex}. It follows that
\[
\begin{aligned}
 L_{\mathcal X}(z_1)
 &\geq
 L_{N,i-1,u}(z_1)
 >
 L_{N,i-1,u}(1/2)
 =
 L_u^{(s-i+1)},\\
 L_{\mathcal Y}(z_2)
 &\geq
 L_{N,j-1,v}(z_2)
 >
 L_{N,j-1,v}(1/2)
 =
 L_v^{(s-j+1)}.
\end{aligned}
\]

First, let $(s,i,j)=(5,3,4)$. Since  both $\mathcal X$ and $\mathcal Y$ are
proper,
\[
 u\leq\binom{4}{2}-1=5,
 \qquad
 v\leq\binom{4}{3}-1=3.
\]
If $u\leq2$, then \eqref{eq:t2-lex-minima} gives
\[
 L_{\mathcal X}(z_1)L_{\mathcal Y}(z_2)
 >
 \frac{58}{45}\cdot\frac67
 =
 \frac{116}{105}>1.
\]
If $v\leq2$, then
\[
 L_{\mathcal X}(z_1)L_{\mathcal Y}(z_2)
 >
 \frac{94}{87}\cdot\frac{14}{15}
 =
 \frac{1316}{1305}>1.
\]
Both alternatives contradict \eqref{eq:upper-shadow-factor}.
Therefore,
\[
 3\leq u\leq5,
 \qquad
 v=3.
\]

Next, let $(s,i,j)=(6,3,5)$. In this case,
\[
 u\leq\binom{5}{2}-1=9,
 \qquad
 v\leq\binom{5}{4}-1=4.
\]
If $u\leq8$, then
\[
 L_{\mathcal X}(z_1)L_{\mathcal Y}(z_2)
 >
 \frac{434}{345}\cdot\frac{22}{27}
 =
 \frac{9548}{9315}>1.
\]
If $v\leq3$, then
\[
 L_{\mathcal X}(z_1)L_{\mathcal Y}(z_2)
 >
 \frac{50}{41}\cdot\frac67
 =
 \frac{300}{287}>1.
\]
Again, either alternative contradicts
\eqref{eq:upper-shadow-factor}. Hence,
\[
 (u,v)=(9,4).
\]

Now let $(s,i,j)=(6,4,4)$. Properness gives
\[
 u,v\leq\binom{5}{3}-1=9.
\]
If $\min\{u,v\}\leq7$, then
\[
 L_{\mathcal X}(z_1)L_{\mathcal Y}(z_2)
 >
 \frac{34}{33}\cdot\frac{46}{47}
 =
 \frac{1564}{1551}>1,
\]
contrary to \eqref{eq:upper-shadow-factor}. Thus $u,v\geq8$.
If $u,v\leq8$, then
\[
 L_{\mathcal X}(z_1)L_{\mathcal Y}(z_2)
 >
 \left(\frac{130}{129}\right)^2
 =
 \frac{16900}{16641}>1,
\]
which is again impossible. Therefore, $u,v\in\{8,9\}$ and at least
one of them equals $9$. Hence,
\[
 (u,v)\in\{(8,9),(9,8),(9,9)\}.
\]

Finally, let $(s,i,j)=(7,4,5)$. Since  both $\mathcal X$ and $\mathcal Y$ are
proper,
\[
 u\leq\binom{6}{3}-1=19,
 \qquad
 v\leq\binom{6}{4}-1=14.
\]
It follows that
\[
 L_{\mathcal X}(z_1)L_{\mathcal Y}(z_2)
 >
 \frac{742}{675}\cdot\frac{190}{207}
 =
 \frac{28196}{27945}>1,
\]
contradicting \eqref{eq:upper-shadow-factor}. 
\end{proof}

Subsection~\ref{sec:t2-dense} treats the cases in which at least one
of $\mathcal X$ and $\mathcal Y$ is full, as well as the three
configurations listed in \eqref{eq:t2-dense-cases}.

\subsection{The remaining cases}\label{sec:t2-dense}

We retain the notation
\[
 N=s-1,\qquad
 \mathcal X=\mathcal X_i^-,
 \qquad
 \mathcal Y=\mathcal Y_j^-,
 \qquad
 u=|\mathcal X|,
 \qquad
 v=|\mathcal Y|,
\]
and put
\[
 z_1=\frac{p_1}{1-p_1},
 \qquad
 z_2=\frac{p_2}{1-p_2}.
\]
Recall that the fixed pair $(\mathcal A,\mathcal B)$ is shifted,
mutually inclusion-maximal, and product-maximizing. For
$2\leq r\leq s$, define
\[
 \mathcal T_{s,r}
 =
 \{A\subseteq[n]:|A\cap[s]|\geq r\}.
\]

We first consider the case in which $\mathcal X$ or $\mathcal Y$ is
full.

\begin{lemma}\label{lem:full-boundary-threshold}
Suppose that $(s,i,j)$ is one of the four types in
\eqref{eq:t2-remaining} and that at least one of $\mathcal X$ and
$\mathcal Y$ is full. Then
\begin{equation}
 \mathcal A=\mathcal T_{s,i},
 \qquad
 \mathcal B=\mathcal T_{s,j}.
 \label{eq:full-boundary-form}
\end{equation}
Moreover,
\[
 \mu_{p_1}(\mathcal A)\mu_{p_2}(\mathcal B)
 <
 (p_1p_2)^2.
\]
\end{lemma}

\begin{proof}
By symmetry, suppose that
\[
 \mathcal X=\binom{[s-1]}{i-1}.
\]
We first show that every generating set of $\mathcal A$ has size at
least $i$. Otherwise, let $K\in G(\mathcal A)$ satisfy $|K|<i$.
Since every generating set is contained in $[s]$, we can choose an
$i$-set $E\subseteq[s]$ such that
$
 K\cup\{s\}\subseteq E.
$
The fullness of $\mathcal X$ implies that
$E\in\mathcal X_i\subseteq G(\mathcal A)$. This contradicts the
minimality of $E$, since $K\subsetneq E$.

Every $i$-subset of $[s]$ containing $s$ belongs to
$\mathcal X_i$ and hence to $\mathcal A$. Now let
$H\in\binom{[s-1]}i$. Choose $a\in H$ and set
$
 E=(H\setminus\{a\})\cup\{s\}.
$
Then $E\in\mathcal X_i$. Since $\mathcal A$ is shifted and $a<s$,
we have $H\in\mathcal A$. Thus, every $i$-subset of $[s]$ belongs to
$\mathcal A$, and increasingness gives
$
 \mathcal T_{s,i}\subseteq\mathcal A.
$
Conversely, every member of $\mathcal A$ contains a generating set
of size at least $i$ lying in $[s]$. Therefore,
$\mathcal A\subseteq\mathcal T_{s,i}$, and hence
$
 \mathcal A=\mathcal T_{s,i}.
$

For a set $F\subseteq[n]$, put $b=|F\cap[s]|$. The smallest possible
intersection of $F\cap[s]$ with an $i$-subset of $[s]$ has size
$
 \max\{0,i+b-s\}.
$
Consequently, $F$ has at least two common elements with every member
of $\mathcal T_{s,i}$ if and only if
$
 b\geq s-i+2=j.
$
Mutual inclusion-maximality  gives
$
 \mathcal B=\mathcal T_{s,j}.
$

For $3\leq r\leq s$, define
\[
 Q_{s,r}(p)
 =
 \frac{\mu_p(\mathcal T_{s,r})}{p^2}
 =
 \sum_{a=r}^{s}
 \binom{s}{a}p^{a-2}(1-p)^{s-a}.
\]
For $r\leq a\leq s$ and $0<p<1$, differentiation gives
\[
 \frac{d}{dp}
 \left(p^{a-2}(1-p)^{s-a}\right)
 =
 p^{a-3}(1-p)^{s-a-1}
 \bigl((a-2)-(s-2)p\bigr).
\]
Since $a\geq r$, the derivative is nonnegative for
$0<p\leq(r-2)/(s-2)$. Hence, $Q_{s,r}(p)$ is nondecreasing on this
interval.
For each $(s,i,j)$ in \eqref{eq:t2-remaining}, we have
\[
 \frac13\leq\frac{r-2}{s-2}
 \qquad\text{for }r\in\{i,j\},
\]
except when $(s,r)=(6,3)$. Thus, all the functions
$Q_{s,i}$ and $Q_{s,j}$ needed below are nondecreasing on
$(0,1/3]$, apart from $Q_{6,3}$.
In that case, for $0\leq p\leq1/3$,
\begin{align*}
 Q_{6,3}(p)
 &=20p-45p^2+36p^3-10p^4,\\
 Q_{6,3}'(p)
 &=20-90p+108p^2-40p^3,\\
 Q_{6,3}''(p)
 &=-90+216p-120p^2<0.
\end{align*}
Hence, $Q_{6,3}'$ is decreasing on this interval, and
$
 Q_{6,3}'(p)
 \geq
 Q_{6,3}'(1/3)
 =
 \frac{14}{27}>0.
$
Thus, all the required functions $Q_{s,r}$ are nondecreasing on
$(0,1/3]$.
Direct calculation gives
\begin{equation}
\begin{aligned}
 Q_{5,3}(1/3)Q_{5,4}(1/3)
 &=\frac{187}{243},
&
 Q_{6,3}(1/3)Q_{6,5}(1/3)
 &=\frac{3029}{6561},\\
 Q_{6,4}(1/3)^2
 &=\frac{5329}{6561},
&
 Q_{7,4}(1/3)Q_{7,5}(1/3)
 &=\frac{4169}{6561}.
\end{aligned}
\label{eq:full-boundary-products}
\end{equation}
Each of these values is smaller than one. Therefore,
\[
 \frac{\mu_{p_1}(\mathcal A)\mu_{p_2}(\mathcal B)}
      {p_1^2p_2^2}
 =
 Q_{s,i}(p_1)Q_{s,j}(p_2)\leq
 Q_{s,i}(1/3)Q_{s,j}(1/3)
 <1,
\]
which proves the result.
\end{proof}

It remains to exclude the three possibilities in
\eqref{eq:t2-dense-cases}.

\begin{lemma}\label{lem:dense-boundary}
Suppose that both $\mathcal X$ and $\mathcal Y$ are proper. Then none
of the three configurations in \eqref{eq:t2-dense-cases} can occur
for the fixed product-maximizing pair.
\end{lemma}

\begin{proof}
We first record a consequence of mutual inclusion-maximality. If
$F\notin\mathcal B$, then there is some $A_0\in\mathcal A$ such that
$
 |A_0\cap F|\leq1.
$
Choosing $K\in G(\mathcal A)$ with $K\subseteq A_0$, we obtain
\begin{equation}
 |K\cap F|\leq1.
 \label{eq:t2-generator-witness}
\end{equation}
 Recall
also that, by the choice of $i$, every generating set containing $s$
has size at least $i$.

\textbf{First, suppose that}
\[
 (s,i,j)=(5,3,4),
 \qquad
 3\leq u\leq5,
 \qquad
 v=3.
\]
The family $\mathcal Y$ contains three of the four $3$-subsets of
$[4]$. Let $H$ be the unique member of
$\binom{[4]}3\setminus\mathcal Y$, and put
\[
 F=H\cup\{5\}.
\]

Suppose first that $F\in\mathcal B$, and choose
$G\in G(\mathcal B)$ with $G\subseteq F$. If $5\notin G$, then
$G\subseteq H$. For every $X\in\mathcal X$, the sets
$X\cup\{5\}\in\mathcal A$ and $G\in\mathcal B$ must have at least
two common elements. Hence, $X\subseteq G\subseteq H$. Since
$u\geq3$ and $H$ contains exactly three $2$-subsets, it follows that
\begin{equation}
 u=3,
 \qquad
 \mathcal X=\binom H2,
 \qquad
 G=H.
 \label{eq:dense-first-structure}
\end{equation}
Now suppose that $5\in G$. 
Since $H\notin\mathcal Y=\mathcal Y_4^-$, the set
$F=H\cup\{5\}$ is not a generating set of $\mathcal B$.
Since $|G|\geq3$ and
$F\notin G(\mathcal B)$, we must have
$
 G=G_0\cup\{5\}
$
for some $G_0\in\binom H2$.
 For every $Y\in\mathcal Y$, the set
$Y\cup\{5\}$ belongs to $G(\mathcal B)$. If $G_0\subseteq Y$, then
\[
 G=G_0\cup\{5\}
 \subsetneq
 Y\cup\{5\},
\]
which is impossible because both sets belong to $G(\mathcal B)$ and
the members of $G(\mathcal B)$ are inclusion-minimal in
$\mathcal B$. Hence, every member of $\mathcal Y$ must avoid
containing $G_0$.
There are four $3$-subsets of $[4]$. Exactly two of them contain the
fixed $2$-set $G_0$. Therefore, only two $3$-subsets of $[4]$ do
not contain $G_0$. This contradicts $|\mathcal Y|=3$. Thus the case
$5\in G$ is impossible, and consequently, if $F\in\mathcal B$, then
\eqref{eq:dense-first-structure} holds.

We next show that $F\notin\mathcal B$ is impossible. Suppose that
$F\notin\mathcal B$. Let
$K\in G(\mathcal A)$ satisfy \eqref{eq:t2-generator-witness}. If
$5\in K$, then $|K|\geq3$, so $K\setminus\{5\}$ contains at least two
elements of $[4]$ and therefore meets the $3$-set $H$. Together with
$5\in K\cap F$, this gives $|K\cap F|\geq2$, a contradiction.
Thus, $5\notin K$ and $K\subseteq[4]$. If $|K|\geq3$, then
$|K\cap H|\geq2$, again contradicting
\eqref{eq:t2-generator-witness}. If $|K|\leq2$, cross
$2$-intersection with the three sets
$Y\cup\{5\}$, $Y\in\mathcal Y$, forces $|K|=2$ and
\[
 K\subseteq\bigcap_{Y\in\mathcal Y}Y.
\]
The intersection on the right consists of the unique element of
$[4]\setminus H$, so it cannot contain a $2$-set. This is also a
contradiction. 

Hence, $F\in\mathcal B$, and
\eqref{eq:dense-first-structure} must hold.
In this case,
\[
 \mathcal X=\binom H2,
 \qquad
 \mathcal Y=\binom{[4]}3\setminus\{H\}.
\]
Recall that the coefficient of $z^a$ in $P_{\mathcal H}(z)$ is the
number of $(r+a)$-subsets of the ground set that contain a member of
$\mathcal H$, where $\mathcal H$ is $r$-uniform.
Since $s-1=4, |H|=3$ and $\mathcal X=\binom{H}{2}$,
\[
 P_{\mathcal X}(z)=3+4z+z^2.
\]
For $\mathcal Y=\binom{[4]}3\setminus\{H\}$, we have
\[
 P_{\mathcal Y}(z)=3+z.
\]
Using \eqref{eq:upper-shadow-functions}, we obtain
\[
 L_{\mathcal X}(z)
 =
 \frac{z^2+5z+7}{z^3+5z^2+7z+3},
 \qquad
 L_{\mathcal Y}(z)
 =
 \frac{z+4}{z^2+4z+3}.
\]
Moreover, for $z>0$,
\begin{align*}
 L_{\mathcal X}'(z)
 &=
 -\frac{z^4+10z^3+39z^2+64z+34}
        {(z^3+5z^2+7z+3)^2}<0,\\
 L_{\mathcal Y}'(z)
 &=
 -\frac{z^2+8z+13}
        {(z^2+4z+3)^2}<0.
\end{align*}
Since $z_1,z_2<1/2$, it follows that
\begin{equation}
 L_{\mathcal X}(z_1)L_{\mathcal Y}(z_2)
 >
 L_{\mathcal X}(1/2)L_{\mathcal Y}(1/2)=
 \frac{26}{21}\cdot\frac67
 =
 \frac{52}{49}>1, \label{eq:dense-first-product}
\end{equation}
contradicting \eqref{eq:upper-shadow-factor}.

\textbf{Next, suppose that}
\[
 (s,i,j)=(6,3,5),
 \qquad
 (u,v)=(9,4).
\]
Let $H$ be the unique member of
$\binom{[5]}4\setminus\mathcal Y$, and put
\[
 F=H\cup\{6\}.
\]
Suppose that $F\in\mathcal B$, and choose
$G\in G(\mathcal B)$ with $G\subseteq F$. If $6\notin G$, then
$G\subseteq H$. Cross $2$-intersection implies that
$X\subseteq G$ for every $X\in\mathcal X$. This is impossible,
because $G\subseteq H$ contains at most
$\binom42=6$ different $2$-subsets, whereas $u=9$.
Suppose that $6\in G$. Since $H\notin\mathcal Y=\mathcal Y_5^-$, the set
$F=H\cup\{6\}$ does not belong to $G(\mathcal B)$. 
Now suppose that $G\in G(\mathcal B)$ satisfies
$6\in G\subseteq F$. By the choice of $i=3$, every generating set
containing $6$ has size at least $3$, so $|G|\geq3$. On the other
hand, $G\neq F$ because $F\notin G(\mathcal B)$. Since $|F|=5$, it
follows that $3\leq|G|\leq4$. Therefore,
\[
 G=G_0\cup\{6\},
 \qquad
 G_0\subseteq H,
 \qquad
 |G_0|\in\{2,3\}.
\]
Cross $2$-intersection implies that every member of $\mathcal X$ must meet $G_0$. If $|G_0|=2$, then
only
$
 \binom52-\binom32=7
$
of the $2$-subsets of $[5]$ meet $G_0$, contradicting $u=9$. It remains to consider $|G_0|=3$. For every $Y\in\mathcal Y$, the
set $Y\cup\{6\}$ belongs to $G(\mathcal B)$. If $G_0\subseteq Y$,
then
\[
 G=G_0\cup\{6\}
 \subsetneq
 Y\cup\{6\},
\]
which is impossible because both sets are inclusion-minimal members
of $\mathcal B$. Hence, no member of $\mathcal Y$ contains $G_0$.
There are five $4$-subsets of $[5]$. Exactly two of them contain the
fixed $3$-set $G_0$.  Therefore, only three $4$-subsets of $[5]$ do
not contain $G_0$. This contradicts $|\mathcal Y|=v=4$.

Therefore, we must have $F\notin\mathcal B$.
Choose $K\in G(\mathcal A)$ satisfying
\eqref{eq:t2-generator-witness}. If $6\in K$, then $|K|\geq3$, so
$K\setminus\{6\}$ contains at least two elements of $[5]$ and must
meet the $4$-set $H$. This gives $|K\cap F|\geq2$, a contradiction.
Hence, $6\notin K$. Since $[5]\setminus H$ has one element,
$|K\cap H|\leq1$ implies that $|K|\leq2$. Cross
$2$-intersection with every $Y\cup\{6\}$, $Y\in\mathcal Y$, then
forces $|K|=2$ and
\[
 K\subseteq\bigcap_{Y\in\mathcal Y}Y.
\]
The intersection of the four members of $\mathcal Y$ is the unique
element of $[5]\setminus H$, which cannot contain $K$. This
contradiction excludes the second configuration.

\textbf{Finally, suppose that}
\[
 (s,i,j)=(6,4,4),
 \qquad
 u,v\geq8,
 \qquad
 \max\{u,v\}=9.
\]
Since $\mathcal Y$ is proper, choose
$H\in\binom{[5]}3\setminus\mathcal Y$ and put
\[
 F=H\cup\{6\}.
\]
Suppose that $F\in\mathcal B$, and choose
$G\in G(\mathcal B)$ with $G\subseteq F$. If $6\in G$, then the
choice of $i=4$ gives $|G|\geq4$. Since $|F|=4$, we must have
$G=F$. It would follow that
\[
 H=G\setminus\{6\}\in\mathcal Y,
\]
contrary to the choice of $H$. Hence $6\notin G$. Since
$G\subseteq F=H\cup\{6\}$, we conclude that $G\subseteq H$.
For every $X\in\mathcal X$, cross $2$-intersection gives
\[
 |X\cap G|
 =
 |(X\cup\{6\})\cap G|
 \geq2.
\]
Since $G\subseteq H$, this implies $|X\cap H|\geq2$. The number of
$3$-subsets of $[5]$ having at least two elements in common with the
fixed $3$-set $H$ is
\[
 \binom32\binom21+\binom33=7.
\]
Thus, $|\mathcal X|\leq7$, contradicting $u\geq8$. 

Therefore, $F\notin\mathcal B$.
By \eqref{eq:t2-generator-witness}, there is
$K\in G(\mathcal A)$ such that
$
 |K\cap F|\leq1.
$
If $6\in K$, then $|K|\geq4$, so $K\setminus\{6\}$ contains at
least three elements of $[5]$. Since $[5]\setminus H$ has only two
elements, we have
\[
 (K\setminus\{6\})\cap H\neq\emptyset.
\]
Together with $6\in K\cap F$, this gives $|K\cap F|\geq2$, a
contradiction. Hence $6\notin K$.
Now $K\subseteq[5]$ and
\[
 |K\cap H|=|K\cap F|\leq1.
\]
Since $|[5]\setminus H|=2$, it follows that $|K|\leq3$. For every
$Y\in\mathcal Y$, cross $2$-intersection between
$K\in\mathcal A$ and $Y\cup\{6\}\in\mathcal B$ gives
\[
 |K\cap Y|\geq2.
\]
For $|K|\leq3$, at most
\[
 \binom32\binom21+\binom33=7
\]
members of $\binom{[5]}3$ can meet $K$ in at least two elements.
This contradicts $|\mathcal Y|=v\geq8$. 
Thus, the third configuration is excluded.
\end{proof}

\subsection{Completion of the proof for
\texorpdfstring{$t=2$}{t=2}}\label{sec:t2-completion}

We now combine the preceding results to prove the product inequality
and characterize the equality case for $t=2$.

\begin{proof}[Proof of Theorem~\ref{thm:product-measure} for $t=2$
in the open range]
Fix $0<p_1,p_2<1/3$, and let $(\mathcal A,\mathcal B)$ be a shifted
product-maximizing pair as in Subsection~\ref{sec:boundary-tools}.
A common $2$-star shows that
\[
 \mu_{p_1}(\mathcal A)\mu_{p_2}(\mathcal B)
 \geq(p_1p_2)^2.
\]
Every generating set has size at least $2$. If $s=2$, then
$
 G(\mathcal A)=G(\mathcal B)=\{[2]\},
$
and hence
$\mathcal A=\mathcal B=\mathcal S_{[2]}$.

Assume that $s>2$, and choose
$
 i=\min\{r:\mathcal X_r\cup\mathcal Y_r\neq\emptyset\}.
$
After interchanging the two families if necessary, assume that
$\mathcal X_i\neq\emptyset$, and put $j=s+2-i$.
Lemma~\ref{lem:boundary} gives $\mathcal Y_j\neq\emptyset$, and the
choice of $i$ implies that $i\leq j$. If $i=2$, then
Lemma~\ref{lem:size-t} gives a product strictly smaller than
$(p_1p_2)^2$, a contradiction. Therefore,
\[
 3\leq i\leq j,
 \qquad
 i+j=s+2,
\]
so $s\geq4$.
If $s=4$, Lemma~\ref{lem:excess-two} gives
\[
 \mu_{p_1}(\mathcal A)\mu_{p_2}(\mathcal B)
 \leq(p_1p_2)^2.
\]
Product maximality forces equality, so $\mathcal A$ and $\mathcal B$
are the same $2$-star. Since both families are shifted, this star is
$\mathcal S_{[2]}$, contradicting the definition of $s$.

Thus, $s\geq5$. The arguments leading to
\eqref{eq:t2-remaining} show that the only possible types are
\[
 (s,i,j)=(5,3,4),\ (6,3,5),\ (6,4,4),\ (7,4,5).
\]
For any of these types, put
$
 N=s-1,
 \mathcal X=\mathcal X_i^-,
 \mathcal Y=\mathcal Y_j^-.
$
If at least one of $\mathcal X$ and $\mathcal Y$ is full,
Lemma~\ref{lem:full-boundary-threshold} gives a product strictly
smaller than $(p_1p_2)^2$. If both are proper,
Lemma~\ref{lem:t2-proper-reduction} either excludes the type
$(7,4,5)$ or reduces the parameters to one of the three
configurations in \eqref{eq:t2-dense-cases}. All three are excluded
by Lemma~\ref{lem:dense-boundary}. Hence, no product-maximizing pair
has $s>2$.
Consequently, every shifted product-maximizing pair is
$(\mathcal S_{[2]},\mathcal S_{[2]})$, and the required inequality
follows.

Now let $\mathcal F_1,\mathcal F_2\subseteq2^{[n]}$ attain equality.
The pair is product-maximizing, so both families are increasing.
Simultaneous shifting produces a shifted equality pair
$
 (\mathcal F_1',\mathcal F_2')
 =
 (\mathcal S_{[2]},\mathcal S_{[2]}).
$
Reversing the shifts and applying
Lemma~\ref{lem:inverse-shift-star} gives
$
 \mathcal F_i=\mathcal S_{T_i}
 $
for some $T_i\in\binom{[n]}2$ and $ i=1,2.$
Cross $2$-intersection forces $T_1=T_2$. Conversely, every common
$2$-star attains equality.
\end{proof}

\section{Dimension-free stability}\label{sec:stability}

\subsection{Quantitative comparison}

The proof of Theorem~\ref{thm:stability} is based on a quantitative
comparison between biased measures of increasing families. We begin
with several results on biased measures and stability that will be used to establish this comparison.

For a finite set $U$, a family $\mathcal G\subseteq2^U$, and
$0<r<1$, write
\[
 \mu_r^U(\mathcal G)
 =
 \sum_{A\in\mathcal G}
 r^{|A|}(1-r)^{|U|-|A|}.
\]
When $U=[n]$, we omit the superscript.
For $\mathcal F\subseteq2^{[n]}$ and $j\in[n]$, define
\[
 \mathcal F(\bar j)
 =
 \{A\subseteq[n]\setminus\{j\}:A\in\mathcal F\},
 \qquad
 \mathcal F(j)
 =
 \{A\subseteq[n]\setminus\{j\}:A\cup\{j\}\in\mathcal F\}.
\]
 For $j\in[n]$, define the \textit{$j$-influential set} of
$\mathcal F$ by
\[
 \mathcal I_j(\mathcal F)
 =
 \left\{
 A\subseteq[n]:
 \left|
 \{A,A\mathbin\triangle\{j\}\}\cap\mathcal F
 \right|=1
 \right\}.
\]
For $0<r<1$, the \textit{$j$-influence} of $\mathcal F$ under
$\mu_r$ is defined by
\[
 I_j^r(\mathcal F)
 =
 \mu_r\bigl(\mathcal I_j(\mathcal F)\bigr),
\]
and its total influence is
\[
 I_r(\mathcal F)
 =
 \sum_{j=1}^n I_j^r(\mathcal F).
\]
This agrees with the definition of biased influence in
\cite{E24}.
If $\mathcal F$ is increasing, then
$\mathcal F(\bar j)\subseteq\mathcal F(j)$. Moreover, for every
$B\subseteq[n]\setminus\{j\}$,
\[
 \{B,B\cup\{j\}\}\subseteq\mathcal I_j(\mathcal F)
 \quad\text{if and only if}\quad
 B\in\mathcal F(j)\setminus\mathcal F(\bar j).
\]
Consequently,
\begin{align*}
 I_j^r(\mathcal F)
 &=
 (1-r)
 \mu_r^{[n]\setminus\{j\}}
 \bigl(\mathcal F(j)\setminus\mathcal F(\bar j)\bigr)+
 r
 \mu_r^{[n]\setminus\{j\}}
 \bigl(\mathcal F(j)\setminus\mathcal F(\bar j)\bigr)\\
 &=
 \mu_r^{[n]\setminus\{j\}}
 \bigl(\mathcal F(j)\setminus\mathcal F(\bar j)\bigr).
\end{align*}

With this notation, the Margulis--Russo formula
\cite[Lemma 2.2]{E24} reads as follows.

\begin{lemma}[Margulis--Russo formula]
\label{lem:stability-margulis-russo}
Let $\mathcal F\subseteq2^{[n]}$ be increasing. Then
\[
 \frac{\mathrm d}{\mathrm dr}\mu_r(\mathcal F)
 =
 I_r(\mathcal F),
 \qquad 0<r<1.
\]
\end{lemma}

We shall also use the  biased edge isoperimetric inequality
\cite[Theorem~2.1]{E24}.

\begin{lemma}
\label{lem:stability-isoperimetric}
Let $\mathcal F\subseteq2^{[n]}$ be increasing, and suppose that
$0<r<1$. Then
\[
 rI_r(\mathcal F)
 \geq
 \mu_r(\mathcal F)\log_r\mu_r(\mathcal F).
\]
\end{lemma}

The preceding two lemmas yield the following comparison between
different biases; see also \cite[Lemma~3.2]{ChangLiuLiu}.

\begin{lemma}
\label{lem:two-bias-comparison}
Let $\mathcal F\subseteq2^{[n]}$ be increasing, let
$0<p\leq q<1$, and put
$
 \theta=\frac{\log q}{\log p}.
$
Then
\[
 \mu_q(\mathcal F)
 \geq
 \mu_p(\mathcal F)^\theta.
\]
\end{lemma}

\begin{proof}
The assertion is immediate if
$\mathcal F=\emptyset$ or $\mathcal F=2^{[n]}$. Otherwise,
$0<\mu_r(\mathcal F)<1$ for every $0<r<1$. Define
\[
 h(r)=\log_r\mu_r(\mathcal F).
\]
Differentiating $h(r)=\log\mu_r(\mathcal F)/\log r$ and applying
Lemmas~\ref{lem:stability-margulis-russo} and
\ref{lem:stability-isoperimetric}, we obtain
\begin{align*}
 h'(r)
 &=
 \frac{
 I_r(\mathcal F)\log r/\mu_r(\mathcal F)
 -
 \log\mu_r(\mathcal F)/r
 }{(\log r)^2}\\
 &=
 \frac{
 rI_r(\mathcal F)/\mu_r(\mathcal F)-h(r)
 }{r\log r}
 \leq0,
\end{align*}
where the last inequality uses $r\log r<0$.
Thus, $h(q)\leq h(p)$. Since $\log q<0$,
\[
 \log\mu_q(\mathcal F)
 =
 h(q)\log q
 \geq
 h(p)\log q
 =
 \theta\log\mu_p(\mathcal F),
\]
which proves the result.
\end{proof}

The biased edge-isoperimetric stability theorem of Ellis, Keller and
Lifshitz \cite[Theorem~1.8]{EllisKellerLifshitz} is stated in terms of Boolean functions. 
Under the identification of $A\subseteq[n]$ with its characteristic
vector, let $f=\mathbf1_{\mathcal F}$, where
\[
 \mathbf1_{\mathcal F}(A)
 =
 \begin{cases}
 1,&A\in\mathcal F,\\
 0,&A\notin\mathcal F.
 \end{cases}
\]
Then
\[
 \mu_r(f)=\mu_r(\mathcal F),
 \qquad
 I_j^r[f]
 =
 \mu_r\bigl(\mathcal I_j(\mathcal F)\bigr)
 =
 I_j^r(\mathcal F),
 \qquad
 I^r[f]=I_r(\mathcal F).
\]
Moreover, the monotone increasing subcubes of $\{0,1\}^n$ correspond
to the families $\mathcal S_W=\{A\subseteq[n]:W\subseteq A\}$, where $W\subseteq[n]$. Therefore, the theorem has the following formulation in the present notation.

\begin{lemma}[\cite{EllisKellerLifshitz}]\label{lem:biased-isoperimetric-stability}
For every $\eta>0$, there exist constants
$C_1=C_1(\eta), c_0=c_0(\eta)>0$ such that the following
holds. Let $0<r\leq1-\eta$, and let
$
 0<\xi\leq\frac{c_0}{\log(1/r)}.
$
If $\mathcal F\subseteq2^{[n]}$ is increasing and
\[
 rI_r(\mathcal F)
 \leq
 \mu_r(\mathcal F)
 \bigl(\log_r\mu_r(\mathcal F)+\xi\bigr),
\]
then there exists $W\subseteq[n]$ such that
\[
 \mu_r(\mathcal F\mathbin\triangle\mathcal S_W)
 \leq
 C_1
 \frac{\xi\log(1/r)}
 {\log\bigl(1/(\xi\log(1/r))\bigr)}
 \mu_r(\mathcal F).
\]
\end{lemma}

We now obtain a quantitative stability refinement of
Lemma~\ref{lem:two-bias-comparison}.

\begin{lemma}
\label{lem:quantitative-comparison}
Fix $t\geq1$ and $0<p<q<1/2$, and put
$
 \theta=\frac{\log q}{\log p}.
$
There exist constants $\delta_0>0$ and $C>0$, depending only on
$t,p,q$, with the following property. Suppose that
$n\geq t$, $\mathcal F\subseteq2^{[n]}$ is increasing,
$0<\delta<\delta_0$, and
\begin{equation}
 \left|
 \frac{\mu_p(\mathcal F)}{p^t}-1
 \right|
 \leq\delta,
 \qquad
 1\leq
 \frac{\mu_q(\mathcal F)}
      {\mu_p(\mathcal F)^\theta}
 \leq1+\delta.
 \label{eq:stability-4}
\end{equation}
Then there exists $T\in\binom{[n]}t$ such that
\begin{equation}
 \mu_p(\mathcal F\mathbin\triangle\mathcal S_T)
 \leq C\delta.
 \label{eq:stability-5}
\end{equation}
\end{lemma}

\begin{proof}
Write
\[
 \mu_p(\mathcal F)=p^t(1+u),
 \qquad |u|\leq\delta.
\]
If $0<\delta<1/2$, then, since $p<1/2$ and $t\geq1$,
\[
 0<
 \frac12p^t
 <
 \mu_p(\mathcal F)
 <
 \frac32p^t
 <
 1.
\]
Thus, $\mathcal F$ is neither empty nor equal to $2^{[n]}$. It follows that
\[
 0<\mu_r(\mathcal F)<1,
 \qquad r\in[p,q].
\]
Then the function
$
 h(r)=\log_r\mu_r(\mathcal F)
$
is  well defined on $[p,q]$. Moreover,
\begin{equation}
 h(p)
 =
 t+\frac{\log(1+u)}{\log p}
 =
 t+O_p(\delta).
 \label{eq:stability-6}
\end{equation}
Differentiating $h$ and applying
Lemmas~\ref{lem:stability-margulis-russo} and
\ref{lem:stability-isoperimetric}, we obtain
\begin{equation}
 h'(r)
 =
 \frac{
 rI_r(\mathcal F)/\mu_r(\mathcal F)-h(r)
 }{r\log r}
 \leq0.
\label{eq:stability-7}
\end{equation}
where the last inequality uses $r\log r<0$.

Set
\[
 R=
 \frac{\mu_q(\mathcal F)}
      {\mu_p(\mathcal F)^\theta}.
\]
The second inequality in \eqref{eq:stability-4} gives
\[
 0\leq\log R\leq\log(1+\delta)\leq\delta.
\]
By the definitions of $h$ and $\theta$,
\begin{align*}
 \log R
 &=
 \log\mu_q(\mathcal F)
 -
 \theta\log\mu_p(\mathcal F)=
 h(q)\log q
 -
 \frac{\log q}{\log p}\,h(p)\log p\\
 &=
 \bigl(h(q)-h(p)\bigr)\log q=
 \log(1/q)\bigl(h(p)-h(q)\bigr).
\end{align*}
Consequently,
\begin{equation}
 0
 \leq
 h(p)-h(q)
 \leq
 \frac{\delta}{\log(1/q)}.
 \label{eq:stability-8}
\end{equation}
Since $h$ is continuous on $[p,q]$ and differentiable on $(p,q)$,
the mean value theorem gives some $\rho\in(p,q)$ such that
\[
 h(q)-h(p)=h'(\rho)(q-p).
\]
Equivalently,
\[
 -h'(\rho)
 =
 \frac{h(p)-h(q)}{q-p}.
\]
Since $h$ is nonincreasing, the quantity on the right is
nonnegative. Hence, by \eqref{eq:stability-8},
\[
 0\leq-h'(\rho)
 \leq
 \frac{\delta}
 {(q-p)\log(1/q)}.
\]

Define
\begin{equation}
 e_\rho
 =
 \frac{\rho I_\rho(\mathcal F)}{\mu_\rho(\mathcal F)}
 -
 \log_\rho\mu_\rho(\mathcal F).
 \label{eq:stability-9}
\end{equation}
It follows from \eqref{eq:stability-7} that
\[
 e_\rho
 =
 \rho\log(1/\rho)\bigl(-h'(\rho)\bigr).
\]
Since $\rho\log(1/\rho)\leq1$, we have
\[
 0\leq e_\rho
 \leq
 \frac{\delta}{(q-p)\log(1/q)}.
\]
Set
\[
 K=
 \frac{1}{(q-p)\log(1/q)}
 \qquad\text{and}\qquad
 \xi=K\delta.
\]
Then $0\leq e_\rho\leq\xi$.

\textbf{Apply Lemma~\ref{lem:biased-isoperimetric-stability} with
$\eta=1-q$,} and denote the corresponding constants by $C_1=C_1(q)$ and
$c_0=c_0(q)$. Define
\[
 \delta_1
 =
 \min\left\{
 \frac12,\,
 \frac{\min\{c_0,1/2\}}
 {K\log(1/p)}
 \right\}.
\]
For every $0<\delta<\delta_1$, we have
\[
 K\delta\log(1/p)\leq\min\{c_0,1/2\}.
\]
Since $\rho\in(p,q)$,
$
 \log(1/\rho)\leq\log(1/p),
$
and hence
\[
 0<\xi
 \leq
 \frac{c_0}{\log(1/\rho)},
 \qquad
 \xi\log(1/\rho)\leq\frac12.
\]
Moreover, by the definition of $e_\rho$,
\[
 \rho I_\rho(\mathcal F)
 =
 \mu_\rho(\mathcal F)
 \bigl(
 \log_\rho\mu_\rho(\mathcal F)+e_\rho
 \bigr)\leq
 \mu_\rho(\mathcal F)
 \bigl(
 \log_\rho\mu_\rho(\mathcal F)+\xi
 \bigr).
\]
Thus Lemma~\ref{lem:biased-isoperimetric-stability} applies at bias
$\rho$. It yields a set $W\subseteq[n]$ such that
\begin{equation}
\begin{aligned}
 \kappa
 &:=
 \mu_\rho(\mathcal F\mathbin\triangle\mathcal S_W)\leq
 C_1
 \frac{\xi\log(1/\rho)}
 {\log\bigl(1/(\xi\log(1/\rho))\bigr)}
 \mu_\rho(\mathcal F)\\
 &\leq
 \frac{C_1K\log(1/p)}{\log2}\,\delta=:
 C_2\delta.
\end{aligned}
 \label{eq:stability-10}
\end{equation}
Here we used $\mu_\rho(\mathcal F)\leq1$ and
$\xi\log(1/\rho)\leq1/2$.

Since $h$ is nonincreasing and $\rho\in(p,q)$,
$
 h(q)\leq h(\rho)\leq h(p).
$
It follows from \eqref{eq:stability-6} and
\eqref{eq:stability-8} that
\[
 |h(\rho)-t|
 \leq
 |h(p)-t|+h(p)-h(q)\leq
 C_3\delta
\]
for some constant $C_3=C_3(p,q)>0$. 
By the definition of $h$,
$
 h(\rho)
 =
 \log\mu_\rho(\mathcal F)/\log\rho.
$
Thus,
\[
 \mu_\rho(\mathcal F)
 =
 \rho^{h(\rho)}.
\]
Since $0<\delta<\delta_1\leq1/2$, we have
\[
 |h(\rho)-t|\leq\frac{C_3}{2}.
\]
Applying the mean value theorem to the function $x\mapsto\rho^x$,
there exists $\zeta$ between $0$ and $h(\rho)-t$ such that
\[
 \left|\rho^{h(\rho)-t}-1\right|
 =
 \rho^\zeta\log(1/\rho)\,|h(\rho)-t|.
\]
Moreover,
\[
 |\zeta|
 \leq
 |h(\rho)-t|
 \leq
 C_3\delta
 \leq
 \frac{C_3}{2},
\]
since $\delta<\delta_1\leq1/2$. As $\rho\in(p,q)$, it follows that
\[
 \rho^\zeta\leq p^{-C_3/2}
 \qquad\text{and}\qquad
 \log(1/\rho)\leq\log(1/p).
\]
Consequently,
\[
 \left|\rho^{h(\rho)-t}-1\right|
 \leq
 C_3p^{-C_3/2}\log(1/p)\,\delta.
\]
Therefore,
\[
 \bigl|\mu_\rho(\mathcal F)-\rho^t\bigr|
 =
 \rho^t
 \left|\rho^{h(\rho)-t}-1\right|\leq
 C_3p^{-C_3/2}\log(1/p)\,\delta
 =:
 C_4\delta.
\]
Since $\mu_\rho(\mathcal S_W)=\rho^{|W|}$,
\eqref{eq:stability-10} gives
\begin{equation}
\begin{aligned}
 \bigl|\rho^{|W|}-\rho^t\bigr|
 &\leq
 \bigl|
 \mu_\rho(\mathcal S_W)-\mu_\rho(\mathcal F)
 \bigr|
 +
 \bigl|\mu_\rho(\mathcal F)-\rho^t\bigr|\\
 &\leq
 \mu_\rho(\mathcal F\mathbin\triangle\mathcal S_W)
 +
 C_4\delta\\
 &\leq
 (C_2+C_4)\delta.
\end{aligned}
 \label{eq:stability-11}
\end{equation}

Put
$
 C_5=C_2+C_4.
$
For every nonnegative integer $k\neq t$,
\[
 |\rho^k-\rho^t|
 =
 \rho^{\min\{k,t\}}
 \bigl(1-\rho^{|k-t|}\bigr)\geq
 p^t(1-q).
\]
Indeed,
$
 \rho^{\min\{k,t\}}\geq\rho^t\geq p^t
$
and
$
 1-\rho^{|k-t|}
 \geq
 1-\rho
 \geq
 1-q.
$
Define
\[
 \delta_0
 =
 \min\left\{
 \delta_1,\,
 \frac{p^t(1-q)}{2C_5}
 \right\}.
\]
If $0<\delta<\delta_0$, then
\eqref{eq:stability-11} implies
\[
 \bigl|\rho^{|W|}-\rho^t\bigr|
 <
 \frac12 p^t(1-q).
\]
Since $|W|$ is a nonnegative integer, if $|W|\neq t$, the preceding
separation estimate would give
$
 \bigl|\rho^{|W|}-\rho^t\bigr|
 \geq
 p^t(1-q),
$
contradicting the inequality above. Hence, $|W|=t$. Setting $T=W$, we
have $T\in\binom{[n]}{t}$.

\textbf{It remains to transfer the approximation from bias $\rho$ to bias
$p$.} For each $D\subseteq T$, define
\[
 \mathcal F_T(D)
 =
 \bigl\{
 A\subseteq[n]\setminus T:
 A\cup D\in\mathcal F
 \bigr\}.
\]
Since $\mathcal F$ is increasing, each $\mathcal F_T(D)$ is
increasing. 
Then
Lemma~\ref{lem:stability-margulis-russo} shows that its biased measure
is nondecreasing in the bias. Hence, as $p<\rho$,
\[
 \mu_p^{[n]\setminus T}\bigl(\mathcal F_T(D)\bigr)
 \leq
 \mu_\rho^{[n]\setminus T}\bigl(\mathcal F_T(D)\bigr).
\]
Put
\[
 L=
 \left(\frac{1-p}{1-q}\right)^t.
\]
For every $D\subseteq T$, since $p<\rho<q$,
\[
 \frac{
 p^{|D|}(1-p)^{t-|D|}
 }{
 \rho^{|D|}(1-\rho)^{t-|D|}
 }
 =
 \left(\frac{p}{\rho}\right)^{|D|}
 \left(\frac{1-p}{1-\rho}\right)^{t-|D|}\leq
 \left(\frac{1-p}{1-q}\right)^t=L.
\]
Summing over all $D\subsetneq T$, we obtain
\begin{equation}
\begin{aligned}
 \mu_p(\mathcal F\setminus\mathcal S_T)
 &=
 \sum_{D\subsetneq T}
 p^{|D|}(1-p)^{t-|D|}
 \mu_p^{[n]\setminus T}\bigl(\mathcal F_T(D)\bigr)\\
 &\leq
 L
 \sum_{D\subsetneq T}
 \rho^{|D|}(1-\rho)^{t-|D|}
 \mu_\rho^{[n]\setminus T}\bigl(\mathcal F_T(D)\bigr)\\
 &=
 L\mu_\rho(\mathcal F\setminus\mathcal S_T)\\
 &\leq
 L\kappa.
\end{aligned}
 \label{eq:stability-12}
\end{equation}
Since $\mu_p(\mathcal S_T)=p^t$,
\[
\begin{aligned}
 \mu_p(\mathcal F\mathbin\triangle\mathcal S_T)
 &=
 \mu_p(\mathcal F\setminus\mathcal S_T)
 +
 \mu_p(\mathcal S_T\setminus\mathcal F)\\
 &=
 p^t-\mu_p(\mathcal F)
 +
 2\mu_p(\mathcal F\setminus\mathcal S_T).
\end{aligned}
\]
The first inequality in \eqref{eq:stability-4} gives
\[
 \bigl|p^t-\mu_p(\mathcal F)\bigr|
 \leq
 p^t\delta.
\]
Combining this with \eqref{eq:stability-10} and
\eqref{eq:stability-12}, we obtain
\[
 \mu_p(\mathcal F\mathbin\triangle\mathcal S_T)
 \leq
 \bigl|p^t-\mu_p(\mathcal F)\bigr|
 +
 2L\kappa\leq
 \bigl(p^t+2LC_2\bigr)\delta.
\]
Taking
\[
 C=p^t+2LC_2
\]
proves \eqref{eq:stability-5}.
\end{proof}

\subsection{Identification of the center}

The quantitative comparison obtained above yields the following
$t$-star approximation for each family in a nearly extremal cross
$t$-intersecting pair.

\begin{lemma}
\label{lem:individual-stability}
Fix $t\geq1$ and
$p_1,p_2\in(0,1/(t+1))$. There exist constants
$\varepsilon_0>0$ and $C>0$, depending only on
$t,p_1,p_2$, with the following property. Suppose that
$n\geq t$, $0<\varepsilon<\varepsilon_0$, and
$\mathcal F_1,\mathcal F_2\subseteq2^{[n]}$ are increasing,
cross $t$-intersecting families satisfying
\[
 \mu_{p_1}(\mathcal F_1)
 \mu_{p_2}(\mathcal F_2)
 >
 (1-\varepsilon)^2(p_1p_2)^t.
\]
Then there exist $T_1,T_2\in\binom{[n]}t$ such that
\begin{equation}
 \mu_{p_i}
 \bigl(\mathcal F_i\mathbin\triangle\mathcal S_{T_i}\bigr)
 \leq C\varepsilon,
 \qquad i=1,2.
 \label{eq:stability-14}
\end{equation}
\end{lemma}

\begin{proof}
Choose $\theta\in(0,1)$ sufficiently close to $1$ that
\[
 r_i=p_i^\theta<\frac1{t+1},
 \qquad i=1,2.
\]
Then $p_i<r_i<1/2$.  We shall take
$\varepsilon_0\leq1/2$.  For
$0<\varepsilon<\varepsilon_0$, put
\[
 a_i=\mu_{p_i}(\mathcal F_i),
 \qquad
 x_i=\frac{a_i}{p_i^t},
 \qquad
 \lambda=(1-\varepsilon)^2.
\]
Thus, $0<\lambda<1$.
The hypothesis and Theorem~\ref{thm:product-measure} give
\[
 \lambda<x_1x_2\leq1.
\]

By Lemma~\ref{lem:two-bias-comparison},
\begin{equation}
 \mu_{r_i}(\mathcal F_i)\geq a_i^\theta,
 \qquad i=1,2.\label{f56}
\end{equation}
Applying Theorem~\ref{thm:product-measure} at
$(r_1,p_2)$ gives
\[
 a_1^\theta a_2
 \leq
 \mu_{r_1}(\mathcal F_1)\mu_{p_2}(\mathcal F_2)
 \leq
 (r_1p_2)^t=p_1^{\theta t}p_2^t.
\]
Similarly, applying it at $(p_1,r_2)$ gives
\[
 a_1a_2^\theta\leq(p_1r_2)^t=p_2^{\theta t}p_1^t.
\]
Dividing by the corresponding powers of $p_1$ and $p_2$, we obtain
\[
 x_1^\theta x_2\leq1,
 \qquad
 x_1x_2^\theta\leq1.
\]

Write $X_i=\log x_i$ and $z=X_1+X_2$. Then
$\log\lambda<z\leq0$. Moreover,
\[
 \theta X_1+X_2\leq0,
 \qquad
 X_1+\theta X_2\leq0.
\]
It follows that
\[
 z-(1-\theta)X_1\leq0
 \qquad\text{and}\qquad
 (1-\theta)X_1+\theta z\leq0.
\]
Since $\log\lambda<z\leq0$ and $0<\theta<1$, these inequalities give
\[
 \frac{\log\lambda}{1-\theta}
 <
 \frac{z}{1-\theta}
 \leq
 X_1
 \leq
 -\frac{\theta z}{1-\theta}
 <
 -\frac{\theta\log\lambda}{1-\theta}.
\]
Similarly,
\[
 \frac{\log\lambda}{1-\theta}
 <
 X_2
 <
 -\frac{\theta\log\lambda}{1-\theta}.
 \]
Recall that $\varepsilon_0\leq1/2$. 
Since $\lambda=(1-\varepsilon)^2$, $\varepsilon<\varepsilon_0\leq1/2$, and
\[
 -\log(1-\varepsilon)
 =
 \int_0^\varepsilon\frac{\mathrm ds}{1-s}
 \leq
 \frac{\varepsilon}{1-\varepsilon},
\]
we have
\[
 -\log\lambda
 =
 -2\log(1-\varepsilon)
 \leq
 \frac{2\varepsilon}{1-\varepsilon}
 \leq
 4\varepsilon.
\]
It follows that
\[
 |X_i|
 \leq
 \frac{4}{1-\theta}\varepsilon,
 \qquad i=1,2.
\]
Since $\varepsilon\leq1/2$, it follows that
\[
 e^{|X_i|}
 \leq
 e^{\frac{2}{1-\theta}}.
\]
Recalling that $x_i=e^{X_i}$ and applying the mean value theorem to
the exponential function, we obtain
\begin{equation}
 |x_i-1|
 \leq
 e^{|X_i|}|X_i|
 \leq
 A\varepsilon,
 \qquad i=1,2,
 \label{eq:stability-18}
\end{equation}
where
\[
 A=
 \frac{4}{1-\theta}
 e^{\frac{2}{1-\theta}}.
\]

For $i=1,2$, set
\[
 R_i=
 \frac{\mu_{r_i}(\mathcal F_i)}{a_i^\theta}.
\]
By \eqref{f56}, $R_i\geq1$. Applying
Theorem~\ref{thm:product-measure} at the pair of biases $(r_1,r_2)$,
we obtain
\[
 R_1R_2(a_1a_2)^\theta
 =
 \mu_{r_1}(\mathcal F_1)
 \mu_{r_2}(\mathcal F_2)\leq
 (r_1r_2)^t
 =
 (p_1p_2)^{\theta t}.
\]
Since $a_1a_2=(p_1p_2)^t x_1x_2$, this implies
\[
 R_1R_2
 \leq
 (x_1x_2)^{-\theta}
 <
 \lambda^{-\theta}.
\]
As $R_1,R_2\geq1$, each $R_i$ is at most $R_1R_2$. Moreover,
\[
 0\leq-\theta\log\lambda\leq4\theta\varepsilon\leq2\theta.
\]
Therefore, applying the mean value theorem gives
\[
 \lambda^{-\theta}-1
 =
 e^{-\theta\log\lambda}-e^0\leq
 e^{2\theta}(-\theta\log\lambda)\leq
 4\theta e^{2\theta}\varepsilon.
\]
Hence,
\begin{equation}
 1\leq R_i\leq 1+B\varepsilon,
 \qquad i=1,2,
 \label{eq:stability-19}
\end{equation}
where
$
 B=4\theta e^{2\theta}.
$

For each $i\in\{1,2\}$, let $\delta_i^*>0$ and $C_i^*>0$ be the
constants provided by Lemma~\ref{lem:quantitative-comparison} for the
parameters $t,p_i,r_i$. Put
\[
 D=\max\{A,B\}
\]
and define
\[
 \varepsilon_0
 =
 \min\left\{
 \frac12,\,
 \frac{\delta_1^*}{D},\,
 \frac{\delta_2^*}{D}
 \right\}.
\]
If $0<\varepsilon<\varepsilon_0$ and $\delta=D\varepsilon$, then
$0<\delta<\delta_i^*$ for $i=1,2$. Furthermore,
\eqref{eq:stability-18} and \eqref{eq:stability-19} give
\[
 \left|
 \frac{\mu_{p_i}(\mathcal F_i)}{p_i^t}-1
 \right|
 =
 |x_i-1|
 \leq A\varepsilon\leq 
 \delta
\]
and
\[
 1
 \leq
 \frac{\mu_{r_i}(\mathcal F_i)}
      {\mu_{p_i}(\mathcal F_i)^\theta}
 =
 R_i
 \leq 1+B\varepsilon \leq
 1+\delta.
\]
Therefore,  applying Lemma~\ref{lem:quantitative-comparison} yields
$T_i\in\binom{[n]}t$ such that
\[
 \mu_{p_i}
 \bigl(\mathcal F_i\mathbin\triangle\mathcal S_{T_i}\bigr)
 \leq
 C_i^*\delta= C_i^*D\varepsilon
 \leq
 C\varepsilon,
 \qquad i=1,2,
\]
where
$
 C=D\max\{C_1^*,C_2^*\}.
$
This proves \eqref{eq:stability-14}.
\end{proof}

Lemma~\ref{lem:individual-stability} may produce two
different centers. The following separation estimate will be used to
show that the centers of a sufficiently near-extremal pair must
coincide.

\begin{lemma}
\label{lem:center-alignment}
Let $T_1,T_2\in\binom{[n]}t$, let
$0<p_1,p_2<1$ satisfy $p_1+p_2\leq1$, and let
$\mathcal F_1,\mathcal F_2\subseteq2^{[n]}$ be cross
$t$-intersecting. If $T_1\neq T_2$, then
\begin{equation}
 \frac{\mu_{p_1}
  (\mathcal S_{T_1}\setminus\mathcal F_1)}{p_1^t}
 +
 \frac{\mu_{p_2}
  (\mathcal S_{T_2}\setminus\mathcal F_2)}{p_2^t}
 \geq
 \bigl((1-p_1)(1-p_2)\bigr)^t.
 \label{eq:stability-20}
\end{equation}
\end{lemma}

\begin{proof}
Set
\[
 T_0=T_1\cap T_2,\qquad
 D_1=T_1\setminus T_2,\qquad
 D_2=T_2\setminus T_1,
 \qquad
 R=[n]\setminus(T_1\cup T_2).
\]
Since $|T_1|=|T_2|=t$, we have
\[
 d:=|D_1|=|D_2|=t-|T_0|.
\]
The assumption $T_1\neq T_2$ implies
\[
 1\leq d\leq t
 \qquad\text{and}\qquad
 |T_0|=t-d<t.
\]

\textbf{For each $j\in[n]$, choose a random vector
$(X_j,Y_j)\in\{0,1\}^2$, independently for different coordinates
$j$.} The entry in the column $(x,y)$ of the following table is
$\Pr((X_j,Y_j)=(x,y))$:
\[
\begin{array}{c|cccc}
 & (1,1) & (1,0) & (0,1) & (0,0)\\
\hline
 j\in T_0 & 1   & 0       & 0       & 0\\
 j\in D_1 & p_2 & 1-p_2   & 0       & 0\\
 j\in D_2 & p_1 & 0       & 1-p_1   & 0\\
 j\in R   & 0   & p_1     & p_2     & 1-p_1-p_2
\end{array}
\]
This defines a probability distribution because
$p_1+p_2\leq1$. Define the random sets
\[
 A=\{j\in[n]:X_j=1\},
 \qquad
 B=\{j\in[n]:Y_j=1\}.
\]

 By construction, every element of $T_1$ belongs to $A$. For
$j\notin T_1$, the variables $X_j$ are independent and satisfy
\[
 \Pr(X_j=1)=p_1,
 \qquad
 \Pr(X_j=0)=1-p_1.
\]
Therefore, for every fixed $E\subseteq[n]$,
\[
 \Pr(A=E)
 =
 \begin{cases}
 p_1^{|E|-t}(1-p_1)^{n-|E|},
   & T_1\subseteq E,\\
 0,
   & T_1\nsubseteq E.
 \end{cases}
\]
Consequently, for every $\mathcal H\subseteq2^{[n]}$,
\begin{align*}
 \Pr(A\in\mathcal H)
 &=
 \sum_{E\in\mathcal H, T_1\subseteq E}
 p_1^{|E|-t}(1-p_1)^{n-|E|}\\
 &=
 \frac{1}{p_1^t}
 \sum_{E\in\mathcal H, T_1\subseteq E}
 p_1^{|E|}(1-p_1)^{n-|E|}\\
 &=
 \frac{
  \mu_{p_1}(\mathcal H\cap\mathcal S_{T_1})
 }{p_1^t}.
\end{align*}
Since
$
 \mu_{p_1}(\mathcal S_{T_1})=p_1^t,
$
this also shows that $A$ has the $p_1$-biased distribution
conditioned on containing $T_1$. Similarly,
\[
 \Pr(B\in\mathcal H)
 =
 \frac{
  \mu_{p_2}(\mathcal H\cap\mathcal S_{T_2})
 }{p_2^t}.
\]
In particular,
\begin{equation}
 \Pr(A\notin\mathcal F_1)
 =
 \frac{
  \mu_{p_1}(\mathcal S_{T_1}\setminus\mathcal F_1)
 }{p_1^t},
 \qquad
 \Pr(B\notin\mathcal F_2)
 =
 \frac{
  \mu_{p_2}(\mathcal S_{T_2}\setminus\mathcal F_2)
 }{p_2^t}.
 \label{eq:center-marginals}
\end{equation}

\textbf{Now consider the event}
\[
 \mathcal E
 =
 \bigl\{X_j=0\text{ for every }j\in D_2\bigr\}
 \cap
 \bigl\{Y_j=0\text{ for every }j\in D_1\bigr\}.
\]
For each $j\in D_2$, the probability that $X_j=0$ is $1-p_1$,
and for each $j\in D_1$, the probability that $Y_j=0$ is $1-p_2$.
Independence across coordinates gives
\begin{equation}
 \Pr(\mathcal E)
 =
 (1-p_1)^d(1-p_2)^d.
 \label{eq:center-event}
\end{equation}

We next determine $A\cap B$ on $\mathcal E$. 
For $j\in T_0$, the first row of the table gives
$
 (X_j,Y_j)=(1,1).
$
Hence, $j\in A\cap B$. For $j\in D_1$, we have $X_j=1$ by the
construction, while the event $\mathcal E$ requires $Y_j=0$.
Hence,
\[
 (X_j,Y_j)=(1,0)
 \qquad\text{for every }j\in D_1
 \text{ on }\mathcal E.
\]
Similarly, for $j\in D_2$, the construction gives $Y_j=1$, while
$\mathcal E$ requires $X_j=0$. Thus
\[
 (X_j,Y_j)=(0,1)
 \qquad\text{for every }j\in D_2
 \text{ on }\mathcal E.
\]
Finally, for $j\in R$, the last row of the table shows that
\[
 (X_j,Y_j)\in\{(1,0),(0,1),(0,0)\}.
\]
On the event $\mathcal E$, these cases show that $T_0\subseteq A\cap B$, and no element of
$D_1\cup D_2\cup R=[n]\setminus T_0$ belongs to $A\cap B$.
Therefore,
\[
 A\cap B=T_0
 \qquad\text{and}\qquad
 |A\cap B|=|T_0|=t-d<t
 \quad\text{on }\mathcal E.
\]

Since $\mathcal F_1$ and $\mathcal F_2$ are cross
$t$-intersecting, the events $A\in\mathcal F_1$ and
$B\in\mathcal F_2$ cannot occur simultaneously on $\mathcal E$.
Therefore,
\[
 \mathcal E
 \subseteq
 \{A\notin\mathcal F_1\}
 \cup
 \{B\notin\mathcal F_2\}.
\]
Taking probabilities in the preceding event inclusion, we obtain
\begin{align*}
 \bigl((1-p_1)(1-p_2)\bigr)^d
 &=
 \Pr(\mathcal E)\leq
 \Pr\bigl(
   \{A\notin\mathcal F_1\}
   \cup
   \{B\notin\mathcal F_2\}
 \bigr)\\
 &\leq
 \Pr(A\notin\mathcal F_1)
 +
 \Pr(B\notin\mathcal F_2)\\
 &=
 \frac{
  \mu_{p_1}(\mathcal S_{T_1}\setminus\mathcal F_1)
 }{p_1^t}
 +
 \frac{
  \mu_{p_2}(\mathcal S_{T_2}\setminus\mathcal F_2)
 }{p_2^t},
\end{align*}
where the first equality follows from \eqref{eq:center-event} and
the last equality follows from \eqref{eq:center-marginals}.
Since
\[
 0<(1-p_1)(1-p_2)<1
 \qquad\text{and}\qquad
 d\leq t,
\]
we have
\[
 \bigl((1-p_1)(1-p_2)\bigr)^d
 \geq
 \bigl((1-p_1)(1-p_2)\bigr)^t.
\]
Combining the preceding estimates yields \eqref{eq:stability-20},
as desired.
\end{proof}

We are now ready to prove Theorem~\ref{thm:stability}.

\begin{proof}[Proof of Theorem~\ref{thm:stability}]
Let $\varepsilon_{\mathrm{ind}}>0$ and
$C_{\mathrm{ind}}>0$ be the constants provided by
Lemma~\ref{lem:individual-stability}. Define
\[
 \beta=\bigl((1-p_1)(1-p_2)\bigr)^t,
 \qquad
 Q=p_1^{-t}+p_2^{-t},
\]
and set
\begin{equation}
 \varepsilon_*
 =
 \min\left\{
 \frac12,\,
 \varepsilon_{\mathrm{ind}},\,
 \frac{\beta}{2C_{\mathrm{ind}}Q}
 \right\}.
 \label{eq:stability-epsilon-threshold}
\end{equation}
Since $0<p_1,p_2<1$, we have $\beta>0$, and hence
$\varepsilon_*>0$. All the quantities in
\eqref{eq:stability-epsilon-threshold} depend only on
$t,p_1,p_2$. We shall prove the theorem with
\begin{equation}
 C
 =
 \max\left\{
 C_{\mathrm{ind}}+2,\,
 \frac{2}{\varepsilon_*}
 \right\}.
 \label{eq:stability-final-constant}
\end{equation}

\textbf{First suppose that}
\[
 0<\varepsilon<\varepsilon_*.
\]
For $i=1,2$, let
\[
 \widehat{\mathcal F}_i
 =
 \mathcal F_i^\uparrow
 =
 \left\{
 A\subseteq[n]:
 F\subseteq A
 \text{ for some }F\in\mathcal F_i
 \right\},
\]
and write
\[
 a_i=\mu_{p_i}(\mathcal F_i),
 \qquad
 \widehat a_i=\mu_{p_i}(\widehat{\mathcal F}_i),
 \qquad
 M=(p_1p_2)^t,
 \qquad
 \lambda=(1-\varepsilon)^2.
\]

The families $\widehat{\mathcal F}_1$ and
$\widehat{\mathcal F}_2$ are increasing. They are also cross
$t$-intersecting.
By Theorem~\ref{thm:product-measure},
\[
 \widehat a_1\widehat a_2\leq M.
\]
On the other hand, the hypothesis gives
$
 a_1a_2>\lambda M.
$
Since $\varepsilon<1/2$, we have $\lambda>0$, so
$a_i,\widehat a_i>0$. Put
\[
 \rho_i=\frac{a_i}{\widehat a_i},
 \qquad i=1,2.
\]
Because
$\mathcal F_i\subseteq\widehat{\mathcal F}_i$, we have
$0<\rho_i\leq1$. Moreover,
\[
 \rho_1\rho_2
 =
 \frac{a_1a_2}{\widehat a_1\widehat a_2}
 >
 \frac{\lambda M}{\widehat a_1\widehat a_2}
 \geq\lambda.
\]
Since each $\rho_i$ is at most $1$, it follows that
\[
 \rho_i\geq\rho_1\rho_2>\lambda,
 \qquad i=1,2.
\]
Consequently,
\begin{align}
 \mu_{p_i}
 \bigl(\widehat{\mathcal F}_i\setminus\mathcal F_i\bigr)
 &=
 \widehat a_i-a_i=
 \widehat a_i(1-\rho_i) <
 \widehat a_i(1-\lambda) \notag\\
 &\leq
 1-\lambda
 =
 2\varepsilon-\varepsilon^2
 <
 2\varepsilon,
 \qquad i=1,2.
 \label{eq:stability-21}
\end{align}
We also have
\[
 \widehat a_1\widehat a_2
 \geq
 a_1a_2
 >
 (1-\varepsilon)^2(p_1p_2)^t.
\]
Since
$\varepsilon<\varepsilon_*\leq\varepsilon_{\mathrm{ind}}$,
Lemma~\ref{lem:individual-stability}, applied to
$\widehat{\mathcal F}_1$ and
$\widehat{\mathcal F}_2$, gives
$T_1,T_2\in\binom{[n]}t$ such that
\begin{equation}
 \mu_{p_i}
 \bigl(
 \widehat{\mathcal F}_i
 \mathbin\triangle
 \mathcal S_{T_i}
 \bigr)
 \leq
 C_{\mathrm{ind}}\varepsilon,
 \qquad i=1,2.
 \label{eq:stability-22}
\end{equation}

\textbf{We next show that these two centers are equal.}  The assumptions on
the biases give
\[
 p_1+p_2
 <
 \frac{2}{t+1}
 \leq1,
\]
so Lemma~\ref{lem:center-alignment} is applicable. Suppose, to the
contrary, that $T_1\neq T_2$. That lemma gives
\[
 \beta
 \leq
 \frac{
  \mu_{p_1}
  (\mathcal S_{T_1}\setminus\widehat{\mathcal F}_1)
 }{p_1^t}
 +
 \frac{
  \mu_{p_2}
  (\mathcal S_{T_2}\setminus\widehat{\mathcal F}_2)
 }{p_2^t}.
\]
For $i=1,2$,
\[
 \mathcal S_{T_i}\setminus\widehat{\mathcal F}_i
 \subseteq
 \widehat{\mathcal F}_i
 \mathbin\triangle
 \mathcal S_{T_i}.
\]
Therefore, \eqref{eq:stability-22} implies
\begin{align*}
 \beta
 \leq
 C_{\mathrm{ind}}\varepsilon
 \bigl(p_1^{-t}+p_2^{-t}\bigr)=
 C_{\mathrm{ind}}Q\varepsilon<
 C_{\mathrm{ind}}Q
 \frac{\beta}{2C_{\mathrm{ind}}Q}
 =
 \frac{\beta}{2},
\end{align*}
where the strict inequality follows from the definition of
$\varepsilon_*$. This is impossible because $\beta>0$. Hence,
\[
 T_1=T_2=:T.
\]

\textbf{It remains to transfer the approximation from the upsets
to the original families.} Since $\mathcal F_i\subseteq\widehat{\mathcal F}_i$, we have
$
 \mathcal F_i\setminus\mathcal S_T
 \subseteq
 \widehat{\mathcal F}_i\setminus\mathcal S_T
 \subseteq
 \widehat{\mathcal F}_i
 \mathbin\triangle
 \mathcal S_T.
$
Moreover,
$
 \mathcal S_T\setminus\mathcal F_i
 \subseteq
 \bigl(
 \mathcal S_T\setminus\widehat{\mathcal F}_i
 \bigr)
 \cup
 \bigl(
 \widehat{\mathcal F}_i\setminus\mathcal F_i
 \bigr).
$
Thus,
\[
 \mathcal F_i\mathbin\triangle\mathcal S_T
 \subseteq
 \bigl(
 \widehat{\mathcal F}_i
 \mathbin\triangle
 \mathcal S_T
 \bigr)
 \cup
 \bigl(
 \widehat{\mathcal F}_i\setminus\mathcal F_i
 \bigr).
\]
Combining \eqref{eq:stability-21} and
\eqref{eq:stability-22}, we obtain
\begin{align*}
 \mu_{p_i}
 \bigl(\mathcal F_i\mathbin\triangle\mathcal S_T\bigr)
 &\leq
 \mu_{p_i}
 \bigl(
 \widehat{\mathcal F}_i
 \mathbin\triangle
 \mathcal S_T
 \bigr)
 +
 \mu_{p_i}
 \bigl(
 \widehat{\mathcal F}_i\setminus\mathcal F_i
 \bigr)\\
 &<
 (C_{\mathrm{ind}}+2)\varepsilon\leq
 C\varepsilon,
 \qquad i=1,2.
\end{align*}

\textbf{Finally, suppose that $\varepsilon\geq\varepsilon_*$.} Since
$n\geq t$, choose any $T\in\binom{[n]}t$. For $i=1,2$,
\[
 \mu_{p_i}
 \bigl(\mathcal F_i\mathbin\triangle\mathcal S_T\bigr)
 \leq1.
\]
By \eqref{eq:stability-final-constant} and
$\varepsilon\geq\varepsilon_*$,
\[
 1
 <
 2
 \leq
 \frac{2\varepsilon}{\varepsilon_*}
 \leq
 C\varepsilon.
\]
Hence,
\[
 \mu_{p_i}
 \bigl(\mathcal F_i\mathbin\triangle\mathcal S_T\bigr)
 <
 C\varepsilon,
 \qquad i=1,2.
\]
Therefore, the constant $C=C(t,p_1,p_2)$ defined in
\eqref{eq:stability-final-constant} works for every
$\varepsilon>0$.
\end{proof}

\section{Cross-intersection for integer sequences}\label{sec:integer-sequences}

\subsection{Proof of Theorem~\ref{thm:integer-sequences}}

The following statement is obtained from Lemma~5.2 of Frankl, Lee,
Siggers and Tokushige~\cite{FranklLeeSiggersTokushige}. The first two
assertions follow from parts~(i)--(iii) and~(v), while the final
assertion follows by applying the argument in part~(iv) successively
to the finite procedure in part~(iii).

\begin{lemma}\label{lem:tokushige-reduction}
Let $n\geq t\geq1$ and $m\geq2$. If
$\mathcal H_1,\mathcal H_2\subseteq[m]^n$ are cross
$t$-intersecting, then there exist cross $t$-intersecting families
$\mathcal H'_1,\mathcal H'_2\subseteq[m]^n$ satisfying
\[
 |\mathcal H'_i|=|\mathcal H_i|,
 \qquad i=1,2.
\]
For $i=1,2$, let
\[
 \mathcal F_i
 =
 \left\{
  \{r\in[n]:x_r=1\}:
  \boldsymbol{x}=(x_1,\ldots,x_n)\in\mathcal H'_i
 \right\}.
\]
Then $\mathcal F_1$ and $\mathcal F_2$ are cross
$t$-intersecting families in $2^{[n]}$.
The families $\mathcal H'_1$ and $\mathcal H'_2$ may be chosen with
the following additional property. If $m>t+1$ and there exists
$T\in\binom{[n]}t$ such that
\[
 \mathcal H'_1=\mathcal H'_2
 =
 \left\{
  \boldsymbol{x}\in[m]^n:
  x_r=1\text{ for every }r\in T
 \right\},
\]
then there exist $T'\in\binom{[n]}t$ and
$a_r\in[m]$ for every $r\in T'$ such that
\[
 \mathcal H_1=\mathcal H_2
 =
 \left\{
  \boldsymbol{x}\in[m]^n:
  x_r=a_r\text{ for every }r\in T'
 \right\}.
\]
\end{lemma}

We now apply this reduction to prove
Theorem~\ref{thm:integer-sequences}.

\begin{proof}[Proof of Theorem~\ref{thm:integer-sequences}]
Apply Lemma~\ref{lem:tokushige-reduction}, and retain the notation
$\mathcal H'_1,\mathcal H'_2,\mathcal F_1,\mathcal F_2$ introduced
there. Fix $F\subseteq[n]$. A sequence
$\boldsymbol{x}=(x_1,\ldots,x_n)\in[m]^n$ satisfies
$
 \{r\in[n]:x_r=1\}=F
$
precisely when $x_r=1$ for every $r\in F$ and
$x_r\in\{2,\ldots,m\}$ for every $r\notin F$. Hence, there are
exactly $(m-1)^{n-|F|}$ such sequences. Consequently, for
$i=1,2$,
\begin{align}
 |\mathcal H_i|
 &=|\mathcal H'_i| \notag\leq
 \sum_{F\in\mathcal F_i}(m-1)^{n-|F|} \notag\\
 &=m^n\sum_{F\in\mathcal F_i}
 \left(\frac1m\right)^{|F|}
 \left(1-\frac1m\right)^{n-|F|} \notag\\
 &=m^n\mu_{1/m}(\mathcal F_i).
 \label{eq:sequence-counting}
\end{align}

Since $m\geq t+1$, we have $1/m\leq1/(t+1)$. The families
$\mathcal F_1, \mathcal F_2\subseteq 2^{[n]}$ are cross $t$-intersecting, so
Theorem~\ref{thm:product-measure}, applied with
$p_1=p_2=1/m$, gives
\[
 \mu_{1/m}(\mathcal F_1)
 \mu_{1/m}(\mathcal F_2)
 \leq
 \left(\frac1m\right)^{2t}.
\]
Together with \eqref{eq:sequence-counting}, this yields
\[
 |\mathcal H_1||\mathcal H_2|
 \leq
 m^{2n}\mu_{1/m}(\mathcal F_1)
 \mu_{1/m}(\mathcal F_2)\leq
 m^{2n}\left(\frac1m\right)^{2t}
 =m^{2(n-t)}.
\]

We now consider the equality case when $m>t+1$. Suppose that
$
 |\mathcal H_1||\mathcal H_2|=m^{2(n-t)}.
$
By \eqref{eq:sequence-counting} and
Theorem~\ref{thm:product-measure},
\[
 m^{2(n-t)}
 \leq
 m^{2n}\mu_{1/m}(\mathcal F_1)
 \mu_{1/m}(\mathcal F_2)
 \leq
 m^{2(n-t)}.
\]
Thus equality holds throughout. Since
$1/m<1/(t+1)$, the equality statement in
Theorem~\ref{thm:product-measure} gives a set
$T\in\binom{[n]}t$ such that
\[
 \mathcal F_1=\mathcal F_2=\mathcal S_T.
\]
It follows from the definition of $\mathcal F_i$ that
\[
 \mathcal H'_i
 \subseteq
 \left\{
  \boldsymbol{x}\in[m]^n:
  x_r=1\text{ for every }r\in T
 \right\},
 \qquad i=1,2.
\]
The family on the right has size $m^{n-t}$. Since
$
 |\mathcal H'_1||\mathcal H'_2|
 =
 |\mathcal H_1||\mathcal H_2|
 =
 m^{2(n-t)},
$
both inclusions must be equalities. Therefore,
\[
 \mathcal H'_1=\mathcal H'_2
 =
 \left\{
  \boldsymbol{x}\in[m]^n:
  x_r=1\text{ for every }r\in T
 \right\}.
\]
The final assertion of Lemma~\ref{lem:tokushige-reduction} now yields
a set $T'\in\binom{[n]}t$ and symbols $a_r\in[m]$, $r\in T'$, such
that
\[
 \mathcal H_1=\mathcal H_2
 =
 \left\{
  \boldsymbol{x}\in[m]^n:
  x_r=a_r\text{ for every }r\in T'
 \right\}.
\]
Conversely, any family of this form has size $m^{n-t}$, and any two
of its members agree on every coordinate in $T'$. Hence two
identical families of this form are cross $t$-intersecting and attain
equality.
\end{proof}

\subsection{Proof of Theorem~\ref{thm:vector-sequence}}

For $P\subseteq[m]$ and
$\boldsymbol{x}=(x_1,\ldots,x_n),\boldsymbol{y}=(y_1,\ldots,y_n)\in[m]^n$, write
$\boldsymbol{x}\preceq_P\boldsymbol{y}$ if $y_r=x_r$ for every
$r\in[n]$ such that $x_r\in P$.
In other words, if $x_r\in P$, then the entry of $\boldsymbol{y}$ at coordinate $r$
cannot be changed, whereas coordinates with $x_r\notin P$ may be
changed arbitrarily.
A family $\mathcal H\subseteq[m]^n$ is called
\textit{$P$-complete} if $\boldsymbol{y}\in\mathcal H$ whenever
$\boldsymbol{x}\in\mathcal H$ and
$\boldsymbol{x}\preceq_P\boldsymbol{y}$.
For $\mathcal H\subseteq[m]^n$, define
\[
 \mathcal H^P
 =
 \left\{
  \boldsymbol{y}\in[m]^n:
  \boldsymbol{x}\preceq_P\boldsymbol{y}
  \text{ for some }\boldsymbol{x}\in\mathcal H
 \right\}.
\]
Then $\mathcal H^P$ is $P$-complete and
$\mathcal H\subseteq\mathcal H^P$.

We use the following correlation inequality of Frankl and
Kupavskii~\cite{FranklKupavskii}.

\begin{lemma}
\label{lem:frankl-kupavskii-correlation}
Let $P,Q\subseteq[m]$ be nonempty and disjoint. If
$\mathcal H_1,\mathcal H_2\subseteq[m]^n$ are respectively
$P$-complete and $Q$-complete, then
\[
 \frac{|\mathcal H_1\cap\mathcal H_2|}{m^n}
 \leq
 \frac{|\mathcal H_1|}{m^n}
 \frac{|\mathcal H_2|}{m^n}.
\]
\end{lemma}

\begin{proof}[Proof of Theorem~\ref{thm:vector-sequence}]
Put
\[
 \boldsymbol{t}=(t_1,\ldots,t_m),
 \qquad
 s=\sum_{i=1}^m t_i.
\]
For a vector $\boldsymbol{u}=(u_1,\ldots,u_m)$ of nonnegative
integers, define
\[
 q(n,m,\boldsymbol{u})
 =
 \max\left\{
  \frac{|\mathcal A_1||\mathcal A_2|}{m^{2n}}:
  \begin{array}{c}
   \mathcal A_1,\mathcal A_2\subseteq[m]^n\text{ are cross}
   (u_1,\ldots,u_m)\text{-intersecting}
  \end{array}
 \right\}.
\]
For $P\subseteq[m]$, let
\[
 \boldsymbol{t}_P=(u_1,\ldots,u_m),
 \qquad
 u_i=
 \begin{cases}
  t_i,&i\in P,\\
  0,&i\notin P.
 \end{cases}
\]

\textbf{We first establish the factorization used in the proof.} Let
$P,Q\subseteq[m]$ be nonempty and satisfy
\[
 P\cap Q=\emptyset,
 \qquad
 P\cup Q=[m].
\]
Choose cross $\boldsymbol{t}$-intersecting families
$\mathcal A_1,\mathcal A_2$ attaining
$q(n,m,\boldsymbol{t})$. For $\ell=1,2$, the families
$\mathcal A_\ell^P$ and $\mathcal A_\ell^Q$ are respectively
$P$-complete and $Q$-complete. Moreover,
$\mathcal A_1^P,\mathcal A_2^P$ are cross
$\boldsymbol{t}_P$-intersecting. Indeed, let
$\boldsymbol{y}\in\mathcal A_1^P$ and
$\boldsymbol{z}\in\mathcal A_2^P$. By the definition of
$\mathcal A_1^P$ and $\mathcal A_2^P$, there exist
$\boldsymbol{x}\in\mathcal A_1$ and
$\boldsymbol{w}\in\mathcal A_2$ such that
\[
 \boldsymbol{x}\preceq_P\boldsymbol{y}
 \qquad\text{and}\qquad
 \boldsymbol{w}\preceq_P\boldsymbol{z}.
\]
Hence, every coordinate satisfying $x_r=w_r=i\in P$ also satisfies
$y_r=z_r=i$. Therefore, $\mathcal A_1^P$ and
$\mathcal A_2^P$ are cross $\boldsymbol{t}_P$-intersecting.
Since
$
 \mathcal A_\ell
 \subseteq
 \mathcal A_\ell^P\cap\mathcal A_\ell^Q
$
for $\ell=1,2$, 
Lemma~\ref{lem:frankl-kupavskii-correlation} gives
\begin{align}
 q(n,m,\boldsymbol{t})
 &=
 \frac{|\mathcal A_1||\mathcal A_2|}{m^{2n}}
 \leq
 \prod_{\ell=1}^2
 \frac{|\mathcal A_\ell^P\cap\mathcal A_\ell^Q|}{m^n}
 \notag\\
 &\leq
 \frac{|\mathcal A_1^P||\mathcal A_2^P|}{m^{2n}}
 \frac{|\mathcal A_1^Q||\mathcal A_2^Q|}{m^{2n}}
 \notag\\
 &\leq
 q(n,m,\boldsymbol{t}_P)
 q(n,m,\boldsymbol{t}_Q).
 \label{eq:q-factorization}
\end{align}
Applying \eqref{eq:q-factorization} successively gives
\[
 q(n,m,\boldsymbol{t})
 \leq
 \prod_{i=1}^m q(n,m,\boldsymbol{t}_{\{i\}}).
\]

If $t_i=0$, then
$
 q(n,m,\boldsymbol{t}_{\{i\}})=1.
$
If $t_i>0$, any pair of cross
$\boldsymbol{t}_{\{i\}}$-intersecting families is also cross
$t_i$-intersecting, because every two sequences from the respective
families agree in at least $t_i$ coordinates whose common entry is
$i$. Since
$
 n\geq s\geq t_i
 $
and
$
 m\geq t_i+1,
$
Theorem~\ref{thm:integer-sequences} yields
\[
 q(n,m,\boldsymbol{t}_{\{i\}})
 \leq
 \frac{m^{2(n-t_i)}}{m^{2n}}
 =m^{-2t_i}.
\]
Consequently,
\[
 q(n,m,\boldsymbol{t})
 \leq
 \prod_{i=1}^m m^{-2t_i}
 =m^{-2s}.
\]
Therefore,
\[
 |\mathcal H_1||\mathcal H_2|
 \leq
 m^{2n}q(n,m,\boldsymbol{t})
 \leq
 m^{2(n-s)}.
\]

\textbf{We now determine the equality cases.}  Assume that
$
 m>t_i+1
$
 for every $i\in[m]$ with $t_i>0$,
and that
\[
 |\mathcal H_1||\mathcal H_2|=m^{2(n-s)}.
\]
In particular, both families are nonempty. For
$\ell\in\{1,2\}$ and $i\in[m]$, put
\[
 \mathcal H_{\ell,i}=\mathcal H_\ell^{\{i\}}.
\]
For each $i\in[m]$, the families $\mathcal H_{1,i}$ and
$\mathcal H_{2,i}$ are cross
$\boldsymbol{t}_{\{i\}}$-intersecting. In particular, they are cross
$t_i$-intersecting in the usual sense.
For $2\leq r\leq m$, the family
$
 \bigcap_{i=1}^{r-1}\mathcal H_{\ell,i}
$
is $[r-1]$-complete, whereas $\mathcal H_{\ell,r}$ is
$\{r\}$-complete. Applying
Lemma~\ref{lem:frankl-kupavskii-correlation} successively, and using
$
 \mathcal H_\ell
 \subseteq
 \bigcap_{i=1}^m\mathcal H_{\ell,i},
$
we obtain
\begin{equation}
 \frac{|\mathcal H_\ell|}{m^n}
 \leq
 \prod_{i=1}^m\frac{|\mathcal H_{\ell,i}|}{m^n},
 \qquad \ell=1,2.
 \label{eq:completion-correlation}
\end{equation}
For $t_i>0$, Theorem~\ref{thm:integer-sequences} gives
\[
 \frac{|\mathcal H_{1,i}||\mathcal H_{2,i}|}{m^{2n}}
 \leq m^{-2t_i}.
\]
The same inequality is immediate when $t_i=0$. Hence,
\[
 \frac{|\mathcal H_1||\mathcal H_2|}{m^{2n}}
 \leq
 \prod_{i=1}^m
 \frac{|\mathcal H_{1,i}||\mathcal H_{2,i}|}{m^{2n}}\leq
 \prod_{i=1}^m m^{-2t_i}
 =m^{-2s}.
\]
The equality assumption shows that equality holds throughout. All
$\mathcal H_{\ell,i}$ are nonempty, so every factor in the product
is positive. Since each factor is at most $m^{-2t_i}$, we must have
\begin{equation}
 |\mathcal H_{1,i}||\mathcal H_{2,i}|
 =m^{2(n-t_i)},
 \qquad i\in[m].
 \label{eq:single-completion-equality}
\end{equation}

Suppose first that $t_i>0$. By the equality statement in
Theorem~\ref{thm:integer-sequences}, there exist
$T_i\in\binom{[n]}{t_i}$ and symbols $a_{i,r}\in[m]$, $r\in T_i$,
such that
\[
 \mathcal H_{1,i}=
 \mathcal H_{2,i}
 =
 \left\{
  \boldsymbol{x}\in[m]^n:
  x_r=a_{i,r}\text{ for every }r\in T_i
 \right\}.
\]
We claim that $a_{i,r}=i$ for every $r\in T_i$. Suppose that
$a_{i,r_0}\neq i$ for some $r_0\in T_i$. Choose a symbol
$j\in[m]\setminus\{i\}$, and define
$\boldsymbol{x}\in[m]^n$ by
\[
 x_r=
 \begin{cases}
  a_{i,r},&r\in T_i,\\
  j,&r\notin T_i.
 \end{cases}
\]
Then $\boldsymbol{x}$ belongs to both $\mathcal H_{1,i}$ and
$\mathcal H_{2,i}$. However, the coordinates at which the two
copies of $\boldsymbol{x}$ have common entry $i$ are precisely
\[
 \{r\in T_i:a_{i,r}=i\}.
\]
Since $a_{i,r_0}\neq i$ and $|T_i|=t_i$, this set has size at most
$t_i-1$, contradicting the fact that $\mathcal H_{1,i}$ and
$\mathcal H_{2,i}$ are  cross
$\boldsymbol{t}_{\{i\}}$-intersecting. Thus,
$a_{i,r}=i$ for every $r\in T_i$. Therefore,
\[
 \mathcal H_{1,i}=
 \mathcal H_{2,i}
 =
 \left\{
  \boldsymbol{x}\in[m]^n:
  x_r=i\text{ for every }r\in T_i
 \right\}.
\]
If $t_i=0$, let $T_i=\emptyset$. In this case,
\eqref{eq:single-completion-equality} implies
\[
 \mathcal H_{1,i}=\mathcal H_{2,i}=[m]^n,
\]
so the preceding description remains valid.

The sets $T_1,\ldots,T_m$ are pairwise disjoint. Indeed,
$\mathcal H_1$ is nonempty and is contained in
$\bigcap_{i=1}^m\mathcal H_{1,i}$, whereas a coordinate belonging
to $T_i\cap T_j$ for distinct $i,j\in[m]$ would be required to have
both entries $i$ and $j$. Therefore,
\[
 \mathcal H_1,\mathcal H_2
 \subseteq
 \mathcal C
 :=
 \left\{
  \boldsymbol{x}\in[m]^n:
  x_r=i
  \text{ for every }i\in[m]\text{ and }r\in T_i
 \right\}.
\]
Since $|\mathcal C|=m^{n-s}$ and
$
 |\mathcal H_1||\mathcal H_2|
 =m^{2(n-s)}
 =|\mathcal C|^2,
$
both inclusions are equalities. Hence,
\[
 \mathcal H_1=\mathcal H_2=\mathcal C.
\]

Conversely, let $T_1,\ldots,T_m\subseteq[n]$ be pairwise disjoint
and satisfy $|T_i|=t_i$ for every $i\in[m]$. The family
\[
 \left\{
  \boldsymbol{x}\in[m]^n:
  x_r=i
  \text{ for every }i\in[m]\text{ and }r\in T_i
 \right\}
\]
has size $m^{n-s}$, and any two of its members have at least $t_i$
common coordinates equal to $i$ for every $i\in[m]$. Thus, two
identical copies of this family attain equality.
\end{proof}

The corollary follows by representing every subset of $[n]$ as a
binary sequence.

\begin{proof}[Proof of Corollary~\ref{cor:cross-iu}]
Recall that, for each $F\subseteq[n]$, the sequence
$
 \boldsymbol{x}(F)
 =
 \bigl(x_1(F),\ldots,x_n(F)\bigr)\in[2]^n
$
is defined by
\[
 x_r(F)=
 \begin{cases}
  1,&r\in F,\\
  2,&r\notin F.
 \end{cases}
\]
For $\ell=1,2$, define
\[
 \mathcal H_\ell
 =
 \{\boldsymbol{x}(F):F\in\mathcal F_\ell\}.
\]
Then
$
 |\mathcal H_\ell|=|\mathcal F_\ell|
$
for $\ell=1,2.$
Let $F_1\in\mathcal F_1$ and $F_2\in\mathcal F_2$. Since
$F_1\cap F_2\neq\emptyset$, there exists
$r_1\in F_1\cap F_2$, and hence
$
 x_{r_1}(F_1)=x_{r_1}(F_2)=1.
$
Similarly, since $F_1\cup F_2\neq[n]$, there exists
$r_2\in[n]\setminus(F_1\cup F_2)$, and hence
$
 x_{r_2}(F_1)=x_{r_2}(F_2)=2.
$
Thus, $\mathcal H_1$ and $\mathcal H_2$ are cross
$(1,1)$-intersecting. Applying
Theorem~\ref{thm:vector-sequence} with $m=2$ and
$t_1=t_2=1$, we obtain
\[
 |\mathcal F_1||\mathcal F_2|
 =
 |\mathcal H_1||\mathcal H_2|
 \leq
 2^{2(n-2)}
 =
 2^{2n-4},
\]
as desired.
\end{proof}

\section*{Declaration of competing interest}
The authors declare that they have no competing interests.

\section*{Data availability}
No data was used for the research described in this article.

\section*{Acknowledgments}
%The authors would like to express their sincere thanks to the referee for the valuable suggestions which greatly improved the presentation of the %manuscript.
The authors thank Yongtao Li and Zhiyi Liu for carefully reading the manuscript and
for several helpful comments. In preparing the manuscript, the authors
used ChatGPT as an auxiliary tool to discuss computational and technical
details and to edit parts of the exposition.

\end{document}